\documentclass[hidelinks, 12pt, reqno]{amsart}

\usepackage[
  margin=1in,           
  textwidth=6.2in,      
  textheight=8.5in      
]{geometry} 
\usepackage[utf8]{inputenc}
\usepackage{amsmath, bm}
\usepackage[psamsfonts]{amssymb}
\usepackage{amsfonts}
\usepackage{algorithm2e}
\usepackage{graphicx}
\usepackage{verbatim}
\usepackage{tikz-cd}
\usepackage{mathrsfs}
\usepackage{color}
\usepackage[english]{babel}
\usepackage{indentfirst}
\usepackage{tikz}
\usetikzlibrary{matrix,arrows,decorations.pathmorphing}
\usepackage{algorithm2e}
\usepackage[all]{xy}
\usepackage[T1]{fontenc}
\usepackage{hyperref}
\usepackage{mathtools}
\usepackage{multicol}
\usepackage{float}
\usepackage{physics}
\usepackage{dsfont}
\usepackage{multirow}
\usetikzlibrary{tikzmark}
\usepackage{caption}
\usepackage{subcaption}
\usepackage{float}
\usepackage{enumitem}

\newtheorem{theorem}{Theorem}[section]
\newtheorem{prop}[theorem]{Proposition}

\newtheorem{lemma}[theorem]{Lemma}
\newtheorem{cor}[theorem]{Corollary}

\newtheorem{ex}[theorem]{Example}
\newtheorem{dfn}[theorem]{Definition}
\newtheorem{remark}[theorem]{Remark}

\newtheorem{thm}{Theorem}

\newcommand{\bC}{\mathbb{C}}

\newcommand{\bH}{\mathbb{H}}

\newcommand{\bK}{\mathbb{K}}

\newcommand{\bN}{\mathbb{N}}

\newcommand{\bQ}{\mathbb{Q}}
\newcommand{\bR}{\mathbb{R}}

\newcommand{\bV}{\mathbb{V}}
\newcommand{\bW}{\mathbb{W}}

\newcommand{\bZ}{\mathbb{Z}}

\newcommand{\cA}{\mathcal{A}}
\newcommand{\cB}{\mathcal{B}}
\newcommand{\cC}{\mathcal{C}}

\newcommand{\cH}{\mathcal{H}}

\newcommand{\cP}{\mathcal{P}}

\newcommand{\cV}{\mathcal{V}}

\newcommand{\sD}{\mathscr{D}}

\newcommand{\sL}{\mathscr{L}}

\newcommand{\ra}{\rightarrow}

\newcommand{\xL}{{\hookrightarrow}}
\newcommand{\lH}{\hat{\lambda}|_{\partial W_H}}

\DeclareMathOperator{\im}{Im}
\DeclareMathOperator{\ham}{Ham}

\DeclareMathOperator{\id}{Id}

\DeclareMathOperator{\CF}{CF}

\DeclareMathOperator{\HF}{HF}
\DeclareMathOperator{\SH}{SH}

\DeclareMathOperator{\Vol}{Vol}
\DeclareMathOperator{\Cal}{Cal}

\DeclareMathOperator{\PCont}{PCont}
\DeclareMathOperator{\Sys}{Sys}

\newcommand{\jznote}[1]{#1}

\begin{document}

\author{Qi Feng}
\email{feqi@mail.ustc.edu.cn}
\address{School of Mathematical Sciences, University of Science and Technology of China, 96
Jinzhai Road, Hefei Anhui, 230026, China}

\author{Jun Zhang}
\email{jzhang4518@ustc.edu.cn}
\address{The Institute of Geometry and Physics, University of Science and Technology of China, 96 Jinzhai Road, Hefei Anhui, 230026, China}

\title{Symplectic dynamics via a tube construction}

\maketitle

\begin{abstract}

This is the first in a series of papers studying symplectic dynamics on Liouville domains $(W, \lambda)$ from a geometric approach. We employ the tube construction, combining \cite{Ush22} and \cite{ABE25}, to build a higher-dimensional Liouville domain from $(W,\lambda; H)$, where $H$ is a Hamiltonian function on $(W, \lambda)$, possibly non-autonomous. We establish dynamical stabilities with respect to a new Thompson-type pseudo-metric (inspired by \cite{Tho63}) on Hamiltonians, bounded below by Banach--Mazur type distances. We also explore relations between the topological entropies of the Hamiltonian dynamics and its induced Reeb dynamics, examine symplectic capacities, especially the Gromov width, under the tube construction, and give a categorification of the Thompson-type metric via a generalized tube construction (called the tunnel construction). Finally, we \jznote{investigate the relation between the filtered} symplectic homology of the resulting tube and the Hamiltonian Floer homologies of the input Hamiltonians, as well as their iterates, on the given base $(W, \lambda)$.

\end{abstract}

\tableofcontents

\section{Introduction}\label{sec-intro}

Since the introduction of a non-linear analogue of the Banach--Mazur distance between Liouville domains, proposed by Ostrover and Polterovich in 2015 to measure distances between domains from a symplectic viewpoint, metric geometry has been successfully integrated into quantitative symplectic geometry, particularly in the search for large-scale geometric phenomena and geometrically ``exotic'' domains. To facilitate the discussion, let us recall some basic concepts and notations. A Liouville domain $(W, \lambda)$ is a compact exact symplectic manifold $(W,\lambda)$ with boundary, whose Liouville vector field $Y_\lambda$, defined by $\iota_{Y_\lambda} d\lambda=\lambda$, is outward transverse to $\partial W$.  This condition implies that $Y_\lambda$ conformally rescales $d\lambda$; that is, $\mathcal L_{Y_\lambda} d\lambda = d\lambda$. Denote the flow of $Y_{\lambda}$ by $\sL_{\lambda}^t$. The rescaling of $(W, \lambda)$ is then defined by $aW\coloneqq\sL_{\lambda}^{\ln a}(W) \text{ for } 0<a\leq 1.$
A typical example of $(W, \lambda)$ is a star-shaped domain $(U \subset \bC^n, \lambda_{\rm std})$, where the Liouville vector field is the standard radial vector field. In this case, with the notation above, $aU = \{\sqrt{a}x\mid x\in U\}$ for $0< a \leq 1$. 

Denote $\sD_{2n}\coloneqq\{\text{Liouville domains of dimension }2n\}.$
For any $(W,\lambda_{W}),(V,\lambda_V)\in \sD_{2n}$, an embedding $\varphi\colon W\ra V$ satisfying $\varphi^*\lambda_V-\lambda_W=df$ for some smooth function $f\colon W\ra \bR$ is called a {\em Liouville embedding} (and a Liouville diffeomorphism if $\varphi$ is additionally a diffeomorphism). In particular, every Liouville embedding is a symplectic embedding, and the two notions coincide when $H^1(W;\bR) = 0$. Throughout the paper, we use the symbol ${\hookrightarrow}$ to denote the existence of a Liouville embedding. Moreover, a Liouville embedding $\varphi\colon W\xL V$ is called a {\em strict Liouville embedding} if $\varphi^*\lambda_V=\lambda_W$ (and a strict Liouville diffeomorphism if $\varphi$ is furthermore a diffeomorphism). The strict Liouville diffeomorphisms on $(W,\lambda)$ form a group, denoted by ${\rm Symp}(W,\lambda)$.

The following two definitions are the most commonly used notions for measuring the symplectic geometric distance between two given Liouville domains (see \cite{GU19,SZ21,Ush22}): for any $(W,\lambda_W), (V,\lambda_V)\in\sD_{2n}$,

 \[
 d_{\rm BM}((W,\lambda_W),(V,\lambda_V))\coloneqq\inf\left\{\ln C \geq 0\,\left|\,  \frac{1}{C}W\xL V, \,\, \frac{1}{C}V\xL W\right\}\right.
 \]
\[
d_{\rm SBM}\left((W,\lambda_W),(V,\lambda_V)\right)\coloneqq\inf\left\{\ln C \geq 0\,\left|\,
\begin{array}{l}
	\exists i_{W,V}\colon\frac{1}{C}W{\hookrightarrow}V\,{\rm{s.t.}}\,\frac{1}{C^2}V\subset i_{W,V}\left(\frac{1}{C}W\right)\subset V\\
	\exists i_{V,W}\colon \frac{1}{C}V{\hookrightarrow}W\, {\rm{s.t.}}\,\frac{1}{C^2}W\subset i_{V,W}\left(\frac{1}{C}V\right)\subset W
\end{array}\right\}\right.. \vspace{2mm}
\]
Clearly, $d_{\mathrm{BM}} \leq d_{\mathrm{SBM}}$, and both $d_{\mathrm{BM}}$ and $d_{\mathrm{SBM}}$ satisfy the triangle inequality. Note that the $d_{\mathrm{BM}}$ and $d_{\mathrm{SBM}}$ defined above satisfy $d_{\mathrm{BM}} = \frac{1}{2} \ln d_c$ and 
$d_{\mathrm{SBM}} = \frac{1}{2} \ln d_f$, where $d_c$ denotes the coarse symplectic Banach--Mazur distance and $d_f$ the fine symplectic Banach--Mazur distance, both defined in Definition~1.2 of \cite{Ush22}. Moreover, it is straightforward to verify that if there exists a Liouville diffeomorphism $\varphi\colon (W,\lambda_W) \simeq (V,\lambda_V)$, then 
\begin{equation} \label{BM-0}
d_{\rm BM}((W,\lambda_W),(V,\lambda_V))=d_{\rm SBM}((W,\lambda_W),(V,\lambda_V))=0.
\end{equation}
Therefore, in general, both $d_{\mathrm{BM}}$ and $d_{\mathrm{SBM}}$ are only pseudo-metrics (which may vanish on distinct inputs). 

\medskip

There have been various approaches to constructing examples of Liouville domains. In this paper, we focus on the following ``dynamics-to-geometry'' approach, which is partially motivated by the global surface of section (GSS, for brevity). Explicitly, given a time-dependent Hamiltonian $H\colon S^1= \bR/\bZ\times  W\ra\bR$ on a Liouville domain $(W,\lambda)$, consider the domain defined as follows:
\begin{equation} \label{Usher-con}
W_H\coloneqq\{(w,z)\in W\times\bC\mid z=re^{2\pi i t } \,\,\mbox{and}\,\, \pi|z|^2\leq H(t,w)\}. 
\end{equation}
Suppose  that $H$ satisfies the following conditions:
\begin{equation}\label{H-con}
\begin{aligned}
&\text{$H$ is autonomous in a neighborhood of $\partial W$,}\\
&H|_{\partial W}\equiv 0,\text{ and }\tau_{\lambda,H}(t,w)\coloneqq \lambda(X_{H_t}(w))+H(t,w)>0
\end{aligned}
\end{equation}
where $H_t(w):=H(t,w)$ and $X_{H_t}$ is defined by $\iota_{X_{H_t}}d\lambda=d_wH_t.$ Then it is readily verified that $\hat{\lambda}\coloneqq\lambda+\pi r^2dt \in\Omega^1(W\times\bC)$ endows $W_H$ with the structure of a new Liouville domain $(W_H,\hat{\lambda})$, where 
\[ \partial W_H = \left(\{(w, z) \in {\rm Int}(W) \times \bC \,| \, \pi|z|^2 = H(t,w)\}\right) \cup (\partial W \times \{0\}). \]
In particular, when the Hamiltonian is autonomous, i.e.\ $H(t, \cdot) = H(\cdot)$, the first component of the boundary $\partial W_H$ above is simply the product ${\rm Int}(W) \times \bR/ \bZ$. 
For more details, see Section~4 of \cite{Ush22}; following the terminology of \cite{Ush22}, we refer to (\ref{Usher-con}), for a Hamiltonian $H$ satisfying (\ref{H-con}), as a {\it tube construction}. In fact, our construction combines and generalizes the constructions in \cite{Ush22, ABE25}. Note that the construction in \cite{Ush22} is restricted to autonomous Hamiltonians, while \cite{ABE25} focuses on time-dependent Hamiltonians defined on the unit ball\footnote{The construction in \cite{ABE25} uses a different normalization, so that when the Hamiltonian is taken to be zero on the unit ball, the construction yields a higher-dimensional unit ball (which is useful in \cite{ABE25} for treating domains close to the unit ball). In our construction (\ref{Usher-con}), when the Hamiltonian vanishes identically, the construction degenerates in the $\bC$-component.}. Here, our construction allows more general Hamiltonians on more general Liouville domains.
 Moreover, in analogy with GSS (in low-dimensional situations), the Reeb dynamics on the contact boundary $\partial (W_H, \hat{\lambda})$ corresponds in a rich way to the Hamiltonian dynamics of $H$ on $(W, d\lambda)$. This correspondence will be discussed in detail in Section~\ref{sec-Reeb-Ham}. 
 
 Equivalently, one may view \(W_H\) fiberwise over \(W\): for each \(w\in W\), the fiber is the star-shaped planar domain $(W_H)_w=\left\{re^{2\pi i t}\in \bC \mid \pi r^2\le H(t,w)\right\}\subset \bC$.
Thus \(W_H\) is obtained by replacing each point \(w\in W\) with the corresponding fiber \((W_H)_w\), as illustrated in Figure~\ref{fig:tube-fiber}. 
\begin{figure}[h]
\centering
\includegraphics[width=0.6\textwidth]{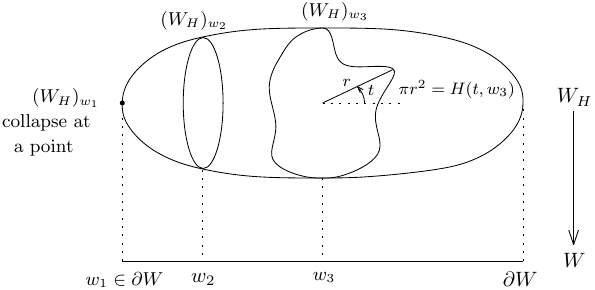}
\caption{The tube construction}
\label{fig:tube-fiber}
\end{figure}
Note that, by definition, for $w$ near $\partial W$, the fiber $(W_H)_w$ is a standard round disk in $\bC$, while for more inside $w \in W$, the fiber $(W_H)_w$ is only star-shaped, not necessarily being a round disk (in particular, as shown in Figure \ref{fig:tube-fiber}, the angle $2 \pi t$ and $H(t, w)$ are related in terms of the variable $t$). 
  
\begin{remark} \label{rmk-obd} Note that the tube construction in (\ref{Usher-con}) provides a special example of an open book decomposition; that is, $\partial W_H$ admits the following fibration structure: 
\begin{equation} \label{OB-bundle}
\pi_{\rm OB}: \partial W_H \backslash (\partial W \times \{0\}) \to S^1 = \bR/\bZ \,\,\,\,\mbox{by}\,\,\,\, \left(w, \sqrt{\frac{H(t,w)}{\pi}} e^{2 \pi i t}\right) \mapsto t. 
\end{equation}
Here, $\partial W \times \{0\}$ serves as the binding. Moreover, for each $t \in S^1$, the fiber is 
\[ \pi^{-1}_{\rm OB}(t) = \left\{\left(w, \sqrt{\frac{H(t,w)}{\pi}} e^{2 \pi i t}\right)  \,\bigg| \, w \in {\rm int}(W)\right\} \]
which is diffeomorphic to ${\rm int}(W)$ via the map $w \mapsto \left(w, \sqrt{H(t,w)/\pi} \, e^{2 \pi i t}\right)$. In particular, this map extends to the boundary $\partial W$ and sends $\partial W$ to the binding $\partial W \times \{0\}$. Therefore, the page of this open book decomposition is the closure \jznote{$\overline {\pi_{\rm OB}^{-1}(t)} =: P(t)$}, which is diffeomorphic to $W$. Furthermore, recall that the monodromy of an open book decomposition is the return diffeomorphism obtained by parallel transport once around the base $S^1 = \bR/\bZ$. Here, it is readily verified that the monodromy $\mu_{\rm OB}: P(0) \to P(0)$ is precisely the action of the Hamiltonian diffeomorphism $\varphi_H^1$. \jznote{Here, for the bundle over $S^1$ in (\ref{OB-bundle}), one chooses the horizontal distribution as the line field spanned by the Reeb vector field (i.e., the unique horizontal vector field in the tangent bundle of the total space that projects to $\partial_t$).} Explicitly, under the diffeomorphism between ${\rm int}(W)$ and $P(0) \backslash (\partial W \times \{0\})$ above, the monodromy can be expressed as $\mu_{\rm OB}(w) = \varphi_H^1(w)$. In particular, restricted to the boundary $\partial W$, $\mu_{\rm OB}$ acts as \jznote{a reparametrization of} the Reeb flow on $(\partial W, \lambda|_{\partial W})$. 

One can modify $\mu_{\rm OB}$ by composing with a Hamiltonian diffeomorphism on $W$ that restricts to \jznote{a reversed Reeb flow on $\partial W$ so that} the resulting Hamiltonian diffeomorphism on the page $P(0) \simeq W$ is \jznote{precisely the} identity on $\partial W$. In other words, up to isotopy of Hamiltonian diffeomorphisms, the monodromy $\mu_{\rm OB}$ lies in the group ${\rm Ham}(W, \partial W)$ (fixing the boundary $\partial W$). This also shows that our open book decomposition of $\partial W_H$ is special in the sense that the monodromy $\mu_{\rm OB}$ is further isotopic to ${\rm id}_W$. In general, the monodromy $\mu_{\rm OB}$ lies in the symplectomorphism group ${\rm Symp}(W, \partial W)$, which may contain elements not realized by Hamiltonian diffeomorphisms. \end{remark}

Here is an example of the tube construction (\ref{Usher-con}), which illustrates that symmetry in the Hamiltonian dynamics can be reflected, from a geometric perspective, in the resulting Liouville domain from the tube.

\begin{ex} \label{ex-domain-symmetry} Let $(W, \lambda)$ be a fixed Liouville domain and $H: S^1 \times W \to \bR$ be a given Hamiltonian function satisfying the condition (\ref{H-con}). For an integer $p\ge 2$, consider $H^{(p)}(t,w)\coloneqq pH(pt,w).$ Recall that if $H$ generates the Hamiltonian diffeomorphism $\varphi$, then $H^{(p)}$ generates the Hamiltonian diffeomorphism $\varphi^p$. Moreover, a direct computation gives $\tau_{\lambda,H^{(p)}}(t,w)=p\tau_{\lambda,H}(pt,w)>0$. Therefore, $H^{(p)}$ is also a Hamiltonian function satisfying the condition (\ref{H-con}). Then, applying the construction (\ref{Usher-con}), we obtain a Liouville domain $W_{H^{(p)}}$. 

Importantly, by definition we have a $\bZ/p\bZ$-symmetry in this Hamiltonian dynamics in the sense that  
\[ H^{(p)}\left(t+\frac{1}{p},w\right)=H^{(p)}(t,w).\]
This implies that the resulting tube $W_{H^{(p)}}$ carries a (non-free) $\bZ/p\bZ$-action $\rho_{H, p}: W_{H^{(p)}} \to W_{H^{(p)}}$ by rotation in the angular variable, that is, 
\begin{equation}\label{eq-cyclic-action-tube}
\rho_{H, p}\left(w,re^{2\pi i t}\right)\coloneqq \left(w,re^{2\pi i \left(t+\frac{1}{p}\right)}\right).
\end{equation}
Moreover, since on $W_{H^{(p)}}$ the Liouville 1-form is $\hat{\lambda}\coloneqq\lambda+\pi r^2dt$, this action preserves the Liouville structure in a strict way, i.e., $\rho_{H, p}^*\hat\lambda=\hat\lambda$. Therefore, the tube $(W_{H^{(p)}},\hat\lambda)$ admits a $\bZ/p\bZ$-symmetry via strict Liouville diffeomorphism $\rho_{H, p}$. The same argument works for any input Hamiltonian function $G: S^1 \times W \to \bR$ that admits a ``$p$-th root'' $H: S^1 \times W \to \bR$ in the sense that $G = H^{(p)}$ and $H$ also satisfies the condition (\ref{H-con}). In particular, observe that any autonomous Hamiltonian function $G: S^1 \times W \to \bR$ satisfying the condition (\ref{H-con}) is such a function. It admits a $p$-th root for any integer $p$, so one can upgrade the rotation in (\ref{eq-cyclic-action-tube}) to be an $S^1$-rotation on the $\bC$-factor, which shows that the resulting tube $W_G$ admits an $S^1$-symmetry (again via strict Liouville diffeomorphisms). 

Finally, if $(W, \lambda)$ admits a symmetry itself (for instance, $(W, \lambda)$ is a toric domain in $(\bC^n, \lambda_{\rm std})$), then it is possible to choose input Hamiltonian functions $G$ that admit further symmetries, which comes from both the $W$-factor and the $\bC$-factor. As an example, see the paragraph right below Theorem \ref{thm-large-scale-dT} in Section \ref{ssec-LSG}. 
\end{ex}

\jznote{At the level of dynamics, the Hamiltonian dynamics generated by \(H^{(p)}\) was investigated in the pioneering work of Polterovich--Shelukhin \cite{PS16}, where they introduced a homological criterion and applied it to quantitatively distinguish autonomous Hamiltonian systems from non-autonomous ones. Example \ref{ex-domain-symmetry} above, based on the tube construction (\ref{Usher-con}), shifts the discussion to the domain level. Explicitly, an analogous problem is to distinguish a given Liouville domain from a family of Liouville domains with prescribed symmetries, and to do so in a quantitative fashion. Achieving this requires developing several new pieces of machinery. In the present paper, we initiate this program by first studying the symplectic dynamics of the tube construction, with additional Floer-homological machinery to be established in the future work \cite{FZ-Floer}. Finally, we note that previous efforts toward this domain-level distinction have been made in \cite{Hut24, BG25}. Their works, however, mainly focus on toric domains, whereas the domains in our discussion admit much less symmetry. Moreover, from our perspective, rather than merely confirming whether a given domain is {\it not} symplectomorphic to a toric domain, we aim to derive a lower bound that quantifies this difference.}

\subsection{Thompson-type metric} Define
\begin{equation} \label{dfn-H-lambda}
\cH(W,\lambda)\coloneqq\{H\colon  S^1\times W\ra\bR\mid \mbox{$H$ satisfies condition (\ref{H-con})}\}.
\end{equation}
It is easy to see that $\cH(W,\lambda)$ is convex and invariant under $\bR_{>0}$-rescaling. The discussion above then yields the following map, given an initial Liouville domain $(W, \lambda)$: 
\begin{equation} \label{U-map}
{\rm U}: \cH(W,\lambda) \to \sD_{2n+2} \,\,\,\,\mbox{by} \,\,\,\, {\rm U}(H)  = (W_H, \hat{\lambda})
\end{equation}

From a quantitative perspective, one can upgrade $\sD_{2n+2}$ to a pseudo-metric space, either $(\sD_{2n+2}, d_{\rm BM})$ or $(\sD_{2n+2}, d_{\rm SBM})$. We now introduce the following definition, which similarly upgrades $\cH(W,\lambda)$ to a pseudo-metric space. 

\begin{dfn} \label{dfn-T-type}
A {\em Thompson-type distance} on $\cH(W,\lambda)$ is defined by 
\[
d_{\rm T}(H,G)\coloneqq\inf\left\{\ln C \geq 0\,\left|\,\exists\,\varphi\in {\rm Symp}(W,\lambda)\,\, {\rm{s.t.}}\,\, \frac{1}{C}H\leq G\circ\varphi \leq C H\right\}\right.
\]
for any $H, G \in \cH(W,\lambda)$. 
\end{dfn}

Here, the ``{\rm T}'' in the notation $d_{\rm T}(\cdot, \cdot)$ refers to the Thompson metric introduced in \cite{Tho63}. Recall that the original definition of the Thompson metric applies to a real linear space $X$ equipped with a partial order $\preceq$ (induced from a positive cone): for $x, y \in X$, one defines $d_{\rm T}(x,y)\coloneqq \inf\{s> 0\mid e^{-s}x\preceq y\preceq e^sx\}$ (in our setting, $X = \cH(W,\lambda)$ is equipped with the partial order $\preceq$ defined by $H \preceq G$ if and only if $H \leq G$, and $s = \ln C$). The only difference is that we incorporate the group action of ${\rm Symp}(W,\lambda)$ on $\cH(W,\lambda)$ into Definition~\ref{dfn-T-type}. This accounts for the fact that Hamiltonians differing only by the action of ${\rm Symp}(W,\lambda)$ correspond to the same domain in the sense of $d_{\rm BM}$ or $d_{\rm SBM}$, by Lemma~\ref{lemma-reparametrization}. Several metric properties of $d_{\rm T}(\cdot, \cdot)$ are verified in Proposition~\ref{prop-ln-metric} in Section~\ref{sec-thm-A}. 

\begin{remark} \label{rmk-T-H-comp}
We emphasize that the nature of $d_{\rm T}(\cdot, \cdot)$ in Definition~\ref{dfn-T-type} differs from that of another well-known metric in Hamiltonian dynamics---Hofer's metric---where $d_{\rm Hofer}(H, G) = \int_0^1 \max (H_t - G_t) - \min (H_t - G_t) \, dt$. One possible explanation is that, by Corollary~\ref{cor-Reeb-Ham}, the closed orbits of the Reeb dynamics on $\partial(W_H, \hat{\lambda})$ correspond to {\rm all} periodic orbits of the Hamiltonian dynamics on $(W, \lambda)$. Therefore, comparison between $\partial(W_H, \hat{\lambda})$ and $\partial(W_G, \hat{\lambda})$ should be based on the periodic orbits of $\varphi_H^k$ and $\varphi_G^k$ for all $k \in \bN$. By contrast, $d_{\rm Hofer}(H, G)$ above typically captures only the difference between the time-1 maps $\varphi_H^1$ and $\varphi_G^1$. \end{remark}

With the Thompson-type distance $d_{\rm T}(\cdot, \cdot)$ on $\cH(W,\lambda)$, the map in (\ref{U-map}) upgrades to the following: 
\begin{equation} \label{U-map-2}
{\rm U}: (\cH(W,\lambda), d_{\rm T}) \to (\sD_{2n+2}, \mbox{$d_{\rm SBM}$\,\,\mbox{or} \,\,$d_{\rm BM}$}) \,\,\,\,\mbox{by} \,\,\,\, {\rm U}(H)  = (W_H, \hat{\lambda}).
\end{equation}
Our first main result establishes a stability property under the map ${\rm U}$. Restricting to the subset of $\cH(W,\lambda)$ defined by $\cH^+(W,\lambda)\coloneqq\{H\in\cH(W,\lambda)\mid \lambda(X_{H_t})\ge 0 \text{ on } S^1\times W\}$, we have the following inequality. 

\begin{thm}\label{thm-estimation}
For any $H,G\in\cH^+(W,\lambda)$, we have $d_{\rm SBM}({\rm U}(H), {\rm U}(G))\leq d_{\rm T}(H,G)$.
\end{thm}


Obviously, the same conclusion also holds for $d_{\rm BM}$. Moreover, there are variants of Theorem~\ref{thm-estimation}. For instance, one can define an equivalence relation on $\cH(W,\lambda)$ by declaring $H\sim G$ if and only if there exist a smooth family $\{H^s\}_{s\in[0,1]}$ in $\cH(W,\lambda)$ and a neighborhood $U$ of $\partial W$ such that $H^{s}=H^{0} \text{ on } S^{1}\times U$ for all $s\in[0,1]$, with $H^0=H$, $H^1=G$, and such that the time-1 maps $\varphi_{H^s}^1$ are independent of $s$. By the homotopy invariance property in Proposition~\ref{prop-invariant}, ${\rm U}(H^0)$ is Liouville diffeomorphic to ${\rm U}(H^1)$. Therefore, one can also define an equivalence relation on $\sD_{2n+2}$ by declaring $(W,\lambda_W) \sim (V,\lambda_V)$ if and only if they are Liouville diffeomorphic. Again, denote the resulting set of equivalence classes by $\overline{\sD_{2n+2}}$, equipped with the induced $d_{\rm BM}$ and $d_{\rm SBM}$. Note that in this case, by (\ref{BM-0}), the values of $d_{\rm BM}$ and $d_{\rm SBM}$ on equivalence classes are independent of the choice of representatives.

The discussion above yields a further update of (\ref{U-map-2}): 
\begin{equation} \label{U-map-3}
{\rm U}: \left(\overline{\cH(W,\lambda)}, d_{\rm T}\right) \to \left(\overline{\sD_{2n+2}}, \mbox{$d_{\rm SBM}$\,\,\mbox{or} \,\,$d_{\rm BM}$}\right) \,\,\,\,\mbox{by} \,\,\,\, {\rm U}([H])  = [(W_H, \hat{\lambda})]
\end{equation}
which is, in particular, well-defined. The Thompson-type metric $d_{\rm T}(\cdot, \cdot)$ on $\cH(W,\lambda)$ from Definition~\ref{dfn-T-type} can then be restricted to $\cH^+(W,\lambda)$ and descends to a pseudo-metric (still denoted by) $d_{\rm T}(\cdot, \cdot)$ on $\overline{\cH^+(W,\lambda)}$: 
\[
d_{\rm T}([H],[G])\coloneqq\inf\{d_{\rm T}(H',G')\mid [H']=[H],\; [G']=[G]\}
\]
for $[H],[G] \in \overline{\cH^+(W,\lambda)}$.
We then obtain the following corollary of Theorem~\ref{thm-estimation}. 
 
\begin{cor} \label{cor-estimation-class} For any $H,G\in\cH^+(W,\lambda)$ and their corresponding classes $[H], [G] \in \overline{\cH^+(W,\lambda)}$, the following stability inequality holds: 
\begin{equation} \label{cor-geo-stab}
d_{\rm SBM}({\rm U}([H]), {\rm U}([G]))\leq d_{\rm T}([H],[G]).
\end{equation}
In particular, the same holds for $d_{\rm BM}({\rm U}([H]),{\rm U}([G]))$. 
\end{cor}

Exploring another direction as a variant of Theorem~\ref{thm-estimation}, one can consider the contact 1-form $\lH$ on the contact boundary of the Liouville domain ${\rm U}(H) (= (W_H, \hat{\lambda}))$. To simplify notation, we denote $\lH$ by $\alpha_H$. Although, for $H, G\in\cH(W,\lambda)$, the contact 1-forms $\alpha_H$ and $\alpha_G$ are defined on different spaces (namely, $\alpha_H$ on $\partial W_H$ and $\alpha_G$ on $\partial W_G$), Proposition~\ref{prop-contact-bdy} below shows that there exists a canonical diffeomorphism $\varphi_{H,G}: \partial W_H \to \partial W_G$ such that the pullback 1-form $\varphi_{H, G}^*\alpha_G$ defines the same contact structure as $\alpha_H$ on the contact boundary $\partial W_H$. 

Meanwhile, for a contact manifold $(M,\xi=\ker\alpha_0)$, denote by $O_{M,\xi}(\alpha_0)$ the family of contact $1$-forms of $\xi$, namely $\{e^f\alpha_0\mid f\in C^{\infty}(M,\bR)\}$. For any $\alpha_1,\alpha_2\in O_{M,\xi}(\alpha_0)$, recall that the contact Banach--Mazur distance, as in Definition~4 of \cite{RZ21}, is given by 
\begin{equation} \label{dfn-CBM}
d_{\rm CBM}(\alpha_1,\alpha_2)\coloneqq\inf\left\{\ln C\geq 0\,\left|\,\exists\varphi\in{\rm Cont}_0(M,\xi)\text{ s.t. }\frac{1}{C}\alpha_1 \preceq\varphi^*\alpha_2 \preceq C\alpha_1\right\}\right.
\end{equation}
where ${\rm Cont}_0(M,\xi)$ is the identity component of the contactomorphism group of $(M,\xi)$, and $\preceq$ is the partial order defined as follows: $\alpha_1\preceq \alpha_2$ if and only if the corresponding functions $f_1$ and $f_2$ (with $\alpha_1 = e^{f_1}\alpha_0$ and $\alpha_2=e^{f_2}\alpha_0$) satisfy $f_1(p) \leq f_2(p)$ for all $p\in M$. Combining this with the discussion of the previous paragraph, the following definition arises naturally:

\begin{dfn}\label{dfn-CBM-2}
	For any $H, G\in\cH(W,\lambda)$, the {\em contact Banach-Mazur distance} $d_{\rm CBM}$ between $\alpha_H$ and $\alpha_G$ is defined by 
		\begin{align*}
		d_{\rm CBM}\left(\alpha_H,\alpha_G\right)
		&\coloneqq d_{\rm CBM}(\alpha_H,\varphi^*_{H,G}(\alpha_G)) \\
		&=\inf\left\{\ln C\geq 0\left| 
	\begin{array}{l}
		\exists\varphi\in{\rm Cont}_0(\partial W_{H},\ker\alpha_H), \rm{ s.t.}\\
		\varphi^*\varphi_{H,G}^*(\alpha_G)=e^f\alpha_H,|f|\leq \ln C
	\end{array}
	\right\}\right..
		\end{align*}
\end{dfn} 
Proposition~\ref{prop-CBM} investigates several metric properties of $d_{\rm CBM}(\alpha_H, \alpha_G)$, in particular the symmetry of $d_{\rm CBM}(\cdot, \cdot)$. 

\begin{remark} It is not explicitly stated in \cite{RZ21}, but one readily verifies that the $d_{\rm CBM}$ defined in \cite{RZ21}---that is, (\ref{dfn-CBM}) above---is independent of the choice of the base contact 1-form $\alpha_0$ (used to formulate the orbit space $O_{M,\xi}(\alpha_0)$). Therefore, the second equality in Definition~\ref{dfn-CBM-2} is consistent with (\ref{dfn-CBM}) for $\alpha_0 = \alpha_H$.\end{remark}

The following result shows that the stability in Theorem~\ref{thm-estimation} admits an intermediate step, passing through a contact-geometric analogue of $d_{\rm SBM}$. 

\begin{thm}\label{thm-SBM-CBM-ln}
	For any $H,G\in\cH^+(W,\lambda)$, we have 
	\[
	d_{\rm SBM}({\rm U}(H), {\rm U}(G)) \leq d_{\rm CBM}(\alpha_H,\alpha_G)\leq d_{\rm T}(H,G).
	\]
\end{thm}

For both Theorem~\ref{thm-estimation} and Theorem~\ref{thm-SBM-CBM-ln}, once one passes to $d_{\rm SBM}$ between Liouville domains, standard homological machinery in symplectic geometry (such as a version of symplectic homology) comes into play. Through the now-standard process of packaging filtered Floer theory as a persistence module (for more background, see Section~\ref{sec-bottleneck}), one can bound $d_{\rm SBM}$ from below by a combinatorial measurement---the bottleneck distance, denoted by $d_{\rm bot}$. As a consequence, we obtain a further stability result. Denote by $\mathcal B({\rm U}(H))$ the barcode of the filtered $S^1$-equivariant symplectic homology of the Liouville domain ${\rm U}(H)$, for any $H \in \cH(W,\lambda)$. 

\begin{cor} \label{cor-stab-bottle} 
For any $H,G\in\cH^+(W,\lambda)$ and their corresponding classes $[H], [G] \in \overline{\cH^+(W,\lambda)}$, we have 
\[ d_{\rm bot}(\mathcal B({\rm U}(H)), \mathcal B({\rm U}(G))) \leq d_{\rm T}([H], [G]).\]
Here, $H$ and $G$ can be taken to be any representatives of the classes $[H]$ and $[G]$, respectively. 
\end{cor}

Owing to the strong computability of $d_{\rm bot}$, the stability in Corollary~\ref{cor-stab-bottle} provides an efficient approach to estimating (a lower bound for) the pseudo-metric $d_{\rm T}$ on $\overline{\cH^+(W,\lambda)}$, and paves the way to the study of the large-scale geometry of the pseudo-metric space $(\overline{\cH^+(W,\lambda)}, d_{\rm T})$. In fact, after equipping $d_{\rm T}$ with the ${\rm Symp}(W,\lambda)$-action from Definition~\ref{dfn-T-type} and passing to the equivalence classes from $\cH^+(W,\lambda)$ to $\overline{\cH^+(W,\lambda)}$, it is not even obvious that the induced pseudo-metric is non-trivial. Fortunately, by examining the construction of $(W_H, \hat{\lambda})$, the following global quantitative relation is readily derived (see Proposition~\ref{prop-Louville}): for $H\in\cH(W^{2n},\lambda)$, 
\begin{equation} \label{Vol-Cal}
{\rm Vol}(W_H)=\frac{1}{n!}\,\Cal(H) = \frac{1}{n!} \left(\int_{S^1\times W}H dt\wedge (d\lambda)^{n}\right)
\end{equation}
where $\Cal(H)$ denotes the Calabi invariant of $H$, and the third expression in (\ref{Vol-Cal}) follows directly from its definition. Then, by Corollary~\ref{cor-calabi}, we obtain the following result. 

\begin{cor} \label{cor-Cal} For any $H,G\in\cH^+(W^{2n},\lambda)$ and their corresponding classes $[H], [G] \in \overline{\cH^+(W,\lambda)}$, we have 
\[ \frac{1}{n+1}|\ln\Cal(H)-\ln\Cal(G)|\leq d_{\rm BM}({\rm U}([H]), {\rm U}([G]))  \leq d_{\rm T}([H], [G]). \]
Here, $H$ and $G$ can be taken to be any representatives of the classes $[H]$ and $[G]$, respectively. \end{cor}

Note that Corollary~\ref{cor-Cal} implies, in particular, that $d_{\rm T}$ is non-trivial on $\cH^+(W,\lambda)$. However, the Calabi invariant $\Cal(\cdot)$ is too coarse to detect the large-scale geometric phenomena of $d_{\rm T}$, whereas $d_{\rm bot}$ can provide a more refined lower bound for $d_{\rm T}$. The following diagram summarizes the stabilities discussed above. 
\begin{figure}[ht]
	\centering
\includegraphics[scale=1.1]{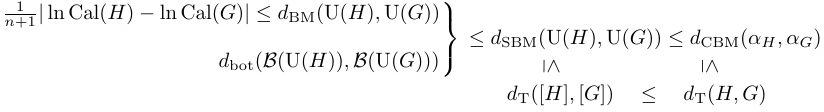}
\end{figure}
\vspace{-6mm}
\jznote{\begin{remark} When the given Liouville domains \((W, \lambda_W)\) and \((V, \lambda_V)\) admit certain \(\mathbb Z/p\mathbb Z\)-symmetries (hence, admitting \(\mathbb Z/p\mathbb Z\)-actions generated by \(\rho_W\) and \(\rho_V\), respectively), e.g., Example \ref{ex-domain-symmetry}, one can introduce an equivariant refinement of the symplectic Banach--Mazur distance. The key point is to consider Liouville embeddings \(\iota: W \hookrightarrow V\) satisfying \(\iota\circ \rho_W=\rho_V \circ\iota\); such an embedding is called {\rm \(\mathbb Z/p\mathbb Z\)-equivariant}. Note also that rescaled domains naturally inherit these \(\mathbb Z/p\mathbb Z\)-actions, since strict Liouville diffeomorphisms commute with the Liouville flow. Analogously, one obtains an equivariant symplectic Banach--Mazur distance, denoted by \(d_{\mathrm{SBM}}^{\mathbb Z/p\mathbb Z}\), which is obviously no smaller than \(d_{\mathrm{SBM}}\). Stability results hold in a similar manner. In particular, for the domains \((W_{H^{(p)}}, \hat{\lambda})\) considered in Example \ref{ex-domain-symmetry}, we have the following stability estimate:
\[
d_{\mathrm{SBM}}^{\mathbb Z/p\mathbb Z}\left((W_{H^{(p)}}, \hat{\lambda}), (W_{G^{(p)}}, \hat{\lambda})\right) \leq d_{\mathrm{T}}(H, G).
\]
For the lower bound of
\(d_{\mathrm{SBM}}^{\mathbb Z/p\mathbb Z}((W_{H^{(p)}}, \hat{\lambda}), (W_{G^{(p)}}, \hat{\lambda}))\),
the (strict) Liouville diffeomorphisms \(\rho_{H, p}\) and \(\rho_{G, p}\) in (\ref{eq-cyclic-action-tube}) induce a \(\mathbb Z/p\mathbb Z\)-action on the corresponding filtered \(S^1\)-equivariant symplectic homology groups. Then, by the general representation theory of persistence modules (see Section 4.4 in \cite{PRSZ20}), one can extract refined invariants from, for instance, the barcodes of subrepresentations, and use them to provide a lower bound. Further developments in this direction will be presented in \cite{FZ-Floer}. 
\end{remark}}

\subsection{Bounded relative growth rate} 

Before turning to a different topic, let us relate the Thompson-type metric $d_{\rm T}$ in Definition~\ref{dfn-T-type} to another classical quantity, the relative growth rate\footnote{This subject arose in a discussion with L.~Polterovich, who conjectured that a result of the form of Theorem~\ref{thm-Thompson-rate} below should hold.}. We briefly recall its definition within the framework of contact geometry (see \cite{EP00}). For a contact manifold $(M,\xi)$ with a co-oriented contact structure, denote by $\PCont_0(M,\xi)$ the set of contact isotopies, i.e., the paths (parametrized by $[0,1]$) in the contactomorphism group of $(M,\xi)$ starting from $\id_M$. Fix a contact $1$-form $\alpha\in\Omega^1(M)$; then each contact isotopy $\psi=\{\psi^t\}_{t\in[0,1]}\in\PCont_0(M,\xi)$ is generated by a unique contact Hamiltonian function $h_{\alpha}(\psi)\colon [0,1]\times M\ra\bR$. Explicitly,
\[
h_\alpha(\psi)=\alpha(X_{\psi})\,\,\, \text{ where } \,\,\,\dv{\psi}{t}=X_{\psi}\circ\psi
\]
We call $\psi$ positive if the corresponding $h_{\alpha}$ is pointwise positive.
Denote by $\PCont_+(M,\xi)$ the set of all positive contact isotopies in $\PCont_0(M,\xi)$. For brevity, set ${\rm G}\coloneqq\widetilde{\rm Cont}_0(M,\xi)$ and consider
\begin{equation} \label{dfn-G+}
{\rm G}_+\coloneqq\{[\psi]\in {\rm G}\mid \psi\in\PCont_+(M,\xi)\}.
\end{equation}
Fix any contact $1$-form $\alpha$ on $M$. For any two elements $[\varphi], [\psi]\in {\rm G}$, define $[\varphi]\succeq[\psi]$ if and only if $h_\alpha(\varphi)\geq h_\alpha(\psi)$ pointwise for some representatives (still denoted by $\varphi$ and $\psi$, without loss of generality) of $[\varphi]$ and $[\psi]$, respectively (Proposition~1.4.B in \cite{EP00}). Then ${\rm G}_+$ in (\ref{dfn-G+}) consists of the dominants of ${\rm G}$, where we call $a\in {\rm G}\setminus\{1\}$ a dominant if for every $b\in {\rm G}$ there exists some $k\in\bN$ such that $a^k\succeq b$. For $a\in {\rm G}_+$ and $b\in {\rm G}$, define 
\begin{align*}
\rho^+_{\succeq }(a,b)&\coloneqq\lim_{l\ra\infty}\frac{\inf\{k\in\bZ\mid a^k\succeq b^l\}}{l}\\
	&=\inf\left.\left\{\frac{k}{l}\in\bQ\,\right|\,(k\in\bZ, l\in\bN)(a^k\succeq b^l)\right\}.
\end{align*}
For the second equality above, see the proof of Theorem~2.4 in \cite{Sim10}. If $b\in {\rm G}_+$, then one similarly defines $\rho_{\succeq}^-(a,b)\coloneqq \rho_{\succeq}^+(b,a)$. The {\em relative growth rate between $a$ and $b$} is defined by
\[
\gamma_{\succeq}(a,b)\coloneqq\max\{|\rho_{\succeq}^+(a,b)|, |\rho_{\succeq}^-(a,b)|\}.
\]
Moreover, if $a,b\in {\rm G}_+$, then both $\rho_{\succeq}^+(a,b)$ and $\rho_{\succeq}^-(a,b)$ are nonzero, since inequality~(1.1.A) in \cite{EP00} implies that $\rho_{\succeq}^+(a,b)\rho_{\succeq}^-(a,b)\geq 1$. Define $d_{\gamma}\colon {\rm G}_+\times {\rm G}_+\ra[0,\infty)$ as follows (see Lemma~1.7.A in \cite{EP00} and (12) in \cite{RZ21}):
\begin{equation} \label{order-metric}
d_{\gamma}(a,b)\coloneqq\ln\gamma_{\succeq}(a,b)\,\,\,\,\mbox{for any $a, b \in {\rm G}_+$}
\end{equation}
and one easily verifies that $({\rm G}_+,d_{\gamma})$ is a pseudo-metric space. 

\medskip

We now focus on a special family of contact manifolds---prequantization spaces. Let $(M,\xi)$ be a prequantization space of a closed symplectic manifold $(X,\omega)$. By definition, there exists a principal $S^1$-bundle structure $p\colon M\ra X$ and an $S^1$-invariant $1$-form $\alpha\in\Omega^1(M)$ such that $d\alpha=p^*\omega$. This setting is particularly useful for linking symplectic Hamiltonian dynamics with contact Hamiltonian dynamics. Explicitly, given a Hamiltonian $H\colon [0,1]\times X\ra\bR$, there exists a contact isotopy $\psi_H=\{\psi_H^t\}_{t\in[0,1]}\in {\rm PCont}_0(M,\xi)$, generated by $p^*H\colon [0,1]\times M\ra\bR$. If $H$ is positive, then so is $p^*H$, which yields a positive contact isotopy $\psi_H \in {\rm PCont}_+(M,\xi)$ and a class $[\psi_H] \in {\rm G}_+$ (see (\ref{dfn-G+})). 

Let $\cH_+(X)\coloneqq\{\text{ smooth }H\colon  X\ra\bR_{+}\}$. Then, in the same spirit as Definition~\ref{dfn-T-type}, we can define another Thompson-type metric on $\cH_+(X)$ as follows:
\[
d_{\rm T, p}(H,G)\coloneqq\inf\left\{\ln C>0\,\left|\, \exists\,\varphi\in \ham(X,\omega)\,\, {\rm{s.t.}}\,\, \frac{1}{C}H\leq G\circ\varphi\leq CH\right\}\right.
\]
for any $H,G\in \cH_+(X)$, where the ``${\rm p}$'' in $d_{\rm T, p}$ indicates that we are working in the prequantization setting. The verification that $d_{\rm T, p}$ is a pseudo-metric follows in the same way as part~(2) of Proposition~\ref{prop-ln-metric}. With respect to the pseudo-metric $d_{\gamma}$ defined in (\ref{order-metric}), we then have the following comparison. 

\begin{thm}\label{thm-Thompson-rate}
	Given a prequantization $(M, \xi) \to (X, \omega)$, we have 
	\[ d_{\gamma}([\psi_H],[\psi_G])\leq d_{\rm T, p}(H,G)\]
	 for any $H,G\in\cH_+(X)$. 
\end{thm}

Note that Theorem \ref{thm-Thompson-rate} partially addresses Problem 1.2.B in \cite{EP00}—specifically in the prequantization setting—concerning the estimation of the relative growth rate in geometric or dynamical terms.

\begin{remark} Note that the value $d_{\gamma}(\cdot, \cdot)$ is not directly computable, as its inputs are homotopy classes. Therefore, it is nontrivial to obtain large-scale geometric properties of $({\rm G}_+, d_{\gamma})$ (and hence of $({\rm Ham}(X, \omega), d_{\rm T, p})$). For instance, a meaningful question is the following: for $(X, \omega) = (S^2, \omega_{\rm std})$ and $(M, \xi) = (S^3, \xi_{\rm std})$, how can $d_{\rm T, p}$ and $d_{\gamma}$ be estimated? We plan to address this question systematically somewhere else. \end{remark}

\subsection{Lifted topological entropy} We next turn to a different topic: investigating the relationship between the topological entropies of various dynamics under the tube construction (\ref{Usher-con}) with condition (\ref{H-con}). This will illustrate, from a quantitative perspective, how the dynamical systems involved change. 

\medskip

Recall that, on a fixed compact metric space $(X, d)$ and for any continuous flow $\varphi: \mathbb{R} \times X \to X$, its \emph{topological entropy} is defined by
\begin{equation} \label{dfn-top-entropy}
h_{\text{top}}(\varphi) = \lim_{\varepsilon \to 0} \limsup_{T \to \infty} \frac{1}{T} \ln N_{\text{sep}}(\varphi, T, \varepsilon)
\end{equation}
where $N_{\text{sep}}(\varphi, T, \varepsilon)$ denotes the maximal cardinality of subsets in which any two distinct points $x,y$ satisfy $\max_{0 \le t \le T} d(\varphi(t, x), \varphi(t, y)) \geq \varepsilon$. In other words, no two distinct points in such a subset have orbits under the flow $\varphi(t, \cdot)$ that remain $\varepsilon$-close up to time $T$. Note that, by Proposition~3.1.2 of \cite{KH95}, the topological entropies computed with respect to different metrics on $X$ coincide; hence we have suppressed the metric in the notation $h_{\mathrm{top}}(\varphi)$. 

There are two variants of the notation $h_{\rm top}(\varphi)$. (i) When $\varphi$ is a continuous map on $X$ (without time-parameter), one defines $h_{\rm top}(\varphi)$ in a parallel way to (\ref{dfn-top-entropy}), by simply replacing $\varphi(t, x)$ with the iterates $\varphi^k(x)$ and $\frac{1}{T}$ with $\frac{1}{k}$. The resulting value is still denoted by $h_{\rm top}(\varphi)$. (ii) For any flow $\varphi: \bR \times X \to X$, by rescaling the time-parameter one shows that $h_{\rm top}(\varphi) = h_{\rm top}(\varphi^1)$, where $\varphi^1\coloneqq \varphi(1, \cdot)$ is the time-1 map of the flow $\varphi$. For a comprehensive treatment of the basic properties of topological entropy, we refer the reader to the book \cite{KH95}. Although topological entropy has been studied extensively, in this paper we focus on the following settings. 

Given a Liouville domain $(W, \lambda)$, recall the set of Hamiltonian functions $\cH(W,\lambda)$ defined in (\ref{dfn-H-lambda}). On the one hand, one can consider the topological entropy of the time-1 map of the Hamiltonian flow $\varphi_H^t$ on $(W, \lambda)$, namely $h_{\rm top}(\varphi_H^1)$; on the other hand, under the tube construction (\ref{Usher-con}), one can consider 
\[  h_{\rm top}\left(\varphi^1_{\hat{\lambda}|_{\partial W_H}}\right) \,\,\,\,\mbox{i.e., topological entropy of the Reeb flow on $\partial (W_H, \hat{\lambda})$}. \]
Corollary~\ref{cor-formula} gives an explicit relation between the Hamiltonian flow on $(W,d\lambda)$ and the Reeb flow on $\partial(W_H,\hat{\lambda})$. Combining this with Proposition~\ref{prop-reparameterization}, we obtain the following comparison theorem.

\begin{thm}\label{thm-entropy}
Given a Liouville domain $(W, \lambda)$, for any $H\in\cH(W,\lambda)$, we have
\[
\left(\min_{S^1\times W}\tau_{\lambda,H}\right) h_{\rm top}\left(\varphi^1_{\hat{\lambda}|_{\partial W_H}}\right)\leq h_{\rm top}(\varphi_H^1)\leq \left(\max_{S^1\times W}\tau_{\lambda,H}\right) h_{\rm top}\left(\varphi^1_{\hat{\lambda}|_{\partial W_H}}\right)
\]
where $\tau_{\lambda, H}$ is the positive function defined in (\ref{H-con}). 
\end{thm}

An immediate consequence is the following. 

\begin{cor} \label{cor-pos-entropy} 
Let $(W,\lambda)$ be a Liouville domain	with $h_{\rm top}(\varphi^1_{\lambda|_{\partial W}})>0$, then for any  $H\in\cH(W,\lambda)$, we have 
\[ h_{\rm top}\left(\varphi^1_{\hat{\lambda}|_{\partial W_H}}\right)>0 \,\,\,\,\mbox{and then $h_{\rm top}(\varphi_H^1)>0$}.\] 
\end{cor}

\begin{proof} By the discussion in Section \ref{sec-Reeb-Ham}, we have 
\[ \varphi^1_{\hat{\lambda}|_{\partial W_H}}: \partial W \times \{0\} \to \partial W \times \{0\}, \]
where $\partial W \times \{0\}$ sits inside a distinguished component $W \times \{0\}$ of $\partial(W_H, \hat{\lambda})$. In other words, $\partial W \times \{0\}$ is an invariant subset of $\varphi^1_{\hat{\lambda}|_{\partial W_H}}$. Moreover, since $\hat{\lambda}|_{\partial W \times \{0\}} = \lambda|_{\partial W}$, we know that 
\[ \varphi^1_{\hat{\lambda}|_{\partial W_H}}\big|_{\partial W \times \{0\}}  = \varphi^1_{{\lambda}|_{\partial W}}. \]
Hence, by our hypothesis, 
\[  h_{\rm top}\left(\varphi^1_{\hat{\lambda}|_{\partial W_H}} \right) \geq h_{\rm top}\left(\varphi^1_{\hat{\lambda}|_{\partial W_H}}\big|_{\partial W \times \{0\}} \right) =  h_{\rm top}(\varphi^1_{\lambda|_{\partial W}}) >0. \]
Then $h_{\rm top}(\varphi_H^1)>0$ by Theorem~\ref{thm-entropy}, which completes the proof.
\end{proof}

\begin{remark} When $(W, \lambda)$ is of dimension 4 (so that $\partial(W, \lambda)$ is of dimension 3), the positive-entropy hypothesis can be ensured in a generic sense (see Theorem~1.5 in \cite{CDHR24}). 
\end{remark}
\begin{remark}
When $\dim W=2$, there is a classical way to obtain stability of the Hamiltonian entropy. Indeed, if $H_j\to H$ in $C^\infty(S^1\times W)$ with $\varphi^1_{H_j}\to \varphi^1_H$ in the $C^\infty$ topology, the map $f\mapsto h_{\mathrm{top}}(f)$ is continuous on the space of $C^\infty$ diffeomorphisms. More precisely, the upper semicontinuity in the $C^\infty$ case is obtained by Yomdin's theory via uniform control of local entropy, see Section~4.3.1 in \cite{Buz06}. On the other hand, Katok's horseshoe theorem for $C^{1+\alpha}$ surface diffeomorphisms gives lower semicontinuity: positive entropy can be approximated by uniformly hyperbolic horseshoes, and such horseshoes persist under $C^1$ perturbations \cite{Kat80}. Hence, for two-dimensional $W$, $H\mapsto h_{\mathrm{top}}(\varphi^1_H)$ is continuous with respect to the $C^\infty$ topology on Hamiltonians.
\end{remark}
By \cite{ADMP23}, for any closed contact $3$-manifold $(Y,\xi)$, there exists a $C^\infty$-generic open subset $\cV_0(Y,\xi)$ of contact form such that the topological entropy  is lower semi-continuous at every point of this subset  with respect to the $C^0$-distance. We can use this result to obtain the stability properties of topological entropy with respect to the Thompson metric. 
\begin{cor}
Let $(W,\lambda)$ be a two-dimensional Liouville domain and let $H_0\in\mathcal H^+(W,\lambda)$ satisfy $\tau_{\lambda,H_0}\equiv c_0>0$. If the associated contact form $\alpha_{H_0}=\hat\lambda|_{\partial W_{H_0}}\in\cV_0(\partial W_{H_0},\ker\alpha_{H_0})$ and a sequence of Hamiltonians $\{H_j\in\mathcal H^+(W,\lambda)\}_{j\in\bN}$ satisfies $d_{\rm T}(H_j,H_0)\to 0$ and $\|\tau_{\lambda,H_j}-c_0\|_{C^0}\ra 0$, then $h_{\rm top}(\varphi_{H_0}^1)\le\liminf_{j\to\infty}h_{\rm top}(\varphi_{H_j}^1)$.
\end{cor}
\begin{proof}
Set $Y_0\coloneqq\partial W_{H_0}, \alpha_0\coloneqq\alpha_{H_0}, \xi_0\coloneqq\ker\alpha_0$. Since $\dim W=2$, the manifold $Y_0$ is a closed three-dimensional contact manifold. By Theorem~1 of \cite{ADMP23}, the Reeb entropy functional $\alpha\mapsto h_{\rm top}(\varphi_\alpha^1)$ is lower semicontinuous at $\alpha_0$ with respect to their $d_{C^0}$-distance on $\mathcal R(Y_0,\xi_0)$, the space of contact forms on $(Y_0,\xi_0)$.

We now explain why this applies to the forms $\alpha_{H_j}$. Since we have $d_{\rm CBM}(\alpha_{H_j},\alpha_{H_0})\le d_{\rm T}(H_j,H_0)$, $d_{\rm CBM}(\alpha_{H_j},\alpha_{H_0})\to 0$.
Equivalently, after identifying $\partial W_{H_j}$ with $Y_0$ by the canonical contact identifications coming from the tube construction, and after composing with contactomorphisms isotopic to the identity, we may write $\widetilde\alpha_j=e^{f_j}\alpha_0$ on $(Y_0,\xi_0)$ with $\|f_j\|_{C^0}\to 0$. Thus $d_{C^0}(\widetilde\alpha_j,\alpha_0)\to 0$ in the sense of~\cite{ADMP23}. Since $\widetilde\alpha_j$ is obtained from $\alpha_{H_j}$ by pullback under a diffeomorphism/contactomorphism, their Reeb flows are conjugate. Therefore $h_{\rm top}(\varphi_{\widetilde\alpha_j}^1)=h_{\rm top}(\varphi_{\alpha_{H_j}}^1).$
Applying the lower semicontinuity theorem of~\cite{ADMP23} gives $h_{\rm top}(\varphi_{\alpha_{H_0}}^1)\le\liminf_{j\to\infty}h_{\rm top}(\varphi_{\alpha_{H_j}}^1)$.

It remains to transfer this inequality from the associated Reeb flows to the Hamiltonian time-one maps.
Since $\tau_{\lambda,H_0}\equiv c_0$, Theorem~\ref{thm-entropy} gives the equality $h_{\rm top}(\varphi_{H_0}^1)=c_0\,h_{\rm top}(\varphi_{\alpha_{H_0}}^1)$.
For the sequence $H_j$, the lower bound gives $h_{\rm top}(\varphi_{H_j}^1)\ge\left(\min_{S^1\times W}\tau_{\lambda,H_j}\right)h_{\rm top}(\varphi_{\alpha_{H_j}}^1)$.
Because $\|\tau_{\lambda,H_j}-c_0\|_{C^0}\to 0$, we have $\min_{S^1\times W}\tau_{\lambda,H_j}\to c_0$. Combining the previous inequalities, we obtain
\[
\begin{aligned}
h_{\rm top}(\varphi_{H_0}^1)
&=c_0\,h_{\rm top}(\varphi_{\alpha_{H_0}}^1) \le c_0\liminf_{j\to\infty}h_{\rm top}(\varphi_{\alpha_{H_j}}^1) \\
&=\liminf_{j\to\infty}\left(\min_{S^1\times W}\tau_{\lambda,H_j}\right)h_{\rm top}(\varphi_{\alpha_{H_j}}^1) \le\liminf_{j\to\infty}h_{\rm top}(\varphi_{H_j}^1).
\end{aligned}
\]
This proves the desired lower semicontinuity.
\end{proof}
\begin{remark}
If $W$ is a $2$-dimensional Liouville domain and the Reeb orbits of $\alpha_H$ on the binding $B=\partial W\times\{0\}$ are non-degenerate, then pages of $\partial W_H$ give a global surface of section for $R_{\alpha_H}$, whose binding consists of finitely many non-degenerate Reeb orbits.  Hence $\alpha_H\in\cV_0(\partial W_H,\ker\alpha_H)$.
\end{remark}

\begin{ex} \label{ex-ite-tube} Given a Liouville domain $(W, \lambda)$ and an autonomous $H \in \mathcal H(W, \lambda)$, set 
\[ \left((W_H)_0, \lambda_0\right) = (W, \lambda) \,\,\,\,\mbox{and}\,\,\,\, \left((W_H)_1, \lambda_1\right) = (W_H, \hat{\lambda}). \]
Then for $n\ge 2$, given Liouville domain $\left((W_H)_{n-1}, \lambda_{n-1}\right)$, define
\[
(W_H)_n\coloneqq \left\{((w, z_1, ..., z_{n-1}),z_n)\in (W_H)_{n-1}\times\mathbb{C}\, \bigg| \, \pi|z_n|^2\le H(w)-\pi\sum_{i=1}^{n-1}|z_i|^2\right\}
\]
and $\lambda_n \coloneqq \lambda_{n-1}+\frac{1}{2}r_n^2d\theta_n$, where $z_n=r_ne^{\sqrt{-1}\theta_n}$. Unfolding the recursion, we obtain 
\[
(W_H)_n = \Bigl\{(w,z_1,\dots,z_n)\in W\times\mathbb{C}^n\;\Big|\;
\pi\sum_{i=1}^n|z_i|^2\le H(w)\Bigr\}
\]
and $\lambda_n = \lambda + \frac{1}{2}\sum_{i=1}^n r_i^2\,d\theta_i$. For simplicity, denote by $(\alpha_H)_n\coloneqq\lambda_n|_{\partial (W_H)_n}$ the induced contact 1-form on the contact boundary $\partial((W_H)_n, \lambda_n)$. 
\end{ex}

By examining the Reeb dynamics on $\partial((W_H)_n, \lambda_n)$, we obtain the following conclusion, which is partially due to the fact that each ``lifting'' from $(W_H)_i$ to $(W_H)_{i+1}$ only produces, dynamically, a rotation in the newly added coordinate, and such a rotation has zero topological entropy.

\begin{cor} \label{cor-pos-entropy-2} Given a Liouville domain $(W, \lambda)$ and an autonomous $H \in \mathcal H(W, \lambda)$, for any $n\in\bN$, we have
	\[
\left(\min_{W}\tau_{\lambda,H}\right) h_{\rm top}(\varphi^1_{(\alpha_H)_n})\leq h_{\rm top}(\varphi_H^1)\leq \left(\max_{W}\tau_{\lambda,H}\right) h_{\rm top}(\varphi^1_{(\alpha_H)_n}).
	\]
	Here, $\varphi^1_{(\alpha_H)_n}$ is the time-1 map of the Reeb flow on $\partial((W_H)_n, \lambda_n)$. Moreover, if $h_{\rm top}(\varphi^1_{\lambda|_{\partial W}})>0$, then $h_{\rm top}(\varphi^1_{(\alpha_H)_n})>0$ for every $n\in\bN$.
 \end{cor}

Finally, for any two-dimensional $(W, \lambda)$, by \cite{Kat80}, positive entropy in two-dimensional dynamics implies the existence of transverse homoclinic points, which cannot arise in the integrable setting (which is essentially equivalent to being autonomous). Hence, for any autonomous Hamiltonian dynamics in this setting, the topological entropy is zero. As a direct implication of Corollary~\ref{cor-pos-entropy-2}, we obtain the following result, which is opposite to Corollary~\ref{cor-pos-entropy}, as well as the second conclusion of Corollary~\ref{cor-pos-entropy-2}. 
\begin{cor} \label{cor-2-dim-ite-entropy}
	Let $H\in\cH(W,\lambda)$ be an autonomous Hamiltonian on any two-dimensional Liouville domain $(W, \lambda)$. Then $h_{\rm top}(\varphi^1_{(\alpha_H)_n})=0$ for every $n\in\bN$. \end{cor}

To this end, recall that a flow $\varphi(t, \cdot)$ on a metric space is called \emph{equicontinuous} if, for every $\varepsilon>0$, there exists $\delta>0$ such that $d(x,y)<\delta$ implies $d(\varphi(t,x),\varphi(t,y))<\varepsilon$ for all $t\in\mathbb{R}$. It is straightforward to verify from the definition that any equicontinuous flow has zero topological entropy (for completeness, a proof will be provided in Proposition~\ref{prop-equiv-entropy} in Section~\ref{ssec-K-equi}). Although $h_{\rm top}(\varphi^1_{(\alpha_H)_n})=0$ for any autonomous $H\in\cH(W,\lambda)$ on a two-dimensional Liouville domain and any $n\in\bN$, there exists an example in which the Reeb flow of $(\alpha_H)_n$ is not equicontinuous.

\begin{ex}\label{ex-non-equicontinuous}
Let $(W,\lambda)=(\overline{B}^{\,2}(\pi),\lambda_{\rm std}=\frac{1}{2}r^2d\varphi)$ and let $H(z)=e-e^{|z|^2}$, or equivalently $H(r)=e-e^{r^2}$. Then $H\in \cH(W,\lambda_{\rm std})$, and the Reeb flow of $(\alpha_H)_n$ is not equicontinuous for any $n\geq 1$. 
\end{ex} 

The verification of Example \ref{ex-non-equicontinuous} will be postponed to Section~\ref{ssec-K-equi}.

\subsection{$C^0$-Liouville domains}

Let $(W,\lambda)$ be a Liouville domain. Define the following $C^0$-strong class of Hamiltonian functions on $(W, \lambda)$:
\begin{align*}
	\cH^0( W,\lambda)\coloneqq\left\{ H\in C^0(S^1\times W,[0,\infty))\;\middle|\;H|_{S^1\times \partial W}=0,\;H>0 \text{ on } S^1\times {\rm Int}(W),\right.\\
\left.\text{and for every }\varepsilon>0,\ c_H(\varepsilon)\coloneqq\min_{(t,w)\in S^1\times W}(H(t,\sL_\lambda^{-\varepsilon}(w))-e^{-\varepsilon}H(t,w))>0\right\}.
\end{align*}
This is the natural time-dependent $C^0$-analogue of the positivity condition $\tau_{\lambda,H}>0$. In particular, every smooth $H\in \cH(W,\lambda)$ belongs to $\cH^0( W,\lambda)$ (see Proposition \ref{prop-smooth-into-strong}). For $H\in \cH^0(W,\lambda)$, similarly to the tube construction in (\ref{Usher-con}), set
\[
W_H\coloneqq\left\{(w,z)\in W\times \bC\;\middle|\;z=re^{2\pi i t},\ \pi |z|^2\le H(t,w)\right\}
\]
and let $U_H\coloneqq {\rm Int}(W_H)$. As before, $U_H$ carries the exact symplectic form $d\hat\lambda$, where $\hat\lambda\coloneqq \lambda+\pi r^2\,dt$. It is then natural to introduce a symplectic Banach--Mazur distance for (open) interior of Liouville domains.

\begin{dfn}\label{def-open-sbm}
Let $H,G\in \cH^0( W,\lambda)$. Define
\[
d_{\rm SBM}\left(W_H,W_G\right)\coloneqq\inf\left\{\ln C>0\,\left|\,
\begin{array}{l}
	\exists i_{H,G}\colon\frac{1}{C}U_H{\hookrightarrow}U_G,{\rm{s.t.}}\,\frac{1}{C^2}U_G\subset i_{H,G}\left(\frac{1}{C}U_H\right)\subset U_G\\
	\exists i_{G,H}\colon \frac{1}{C}U_G{\hookrightarrow}U_H, {\rm{s.t.}}\,\frac{1}{C^2}U_H\subset i_{G,H}\left(\frac{1}{C}U_G\right)\subset U_H
\end{array}\right\}\right.. \vspace{2mm}
\]
Here, for $C\ge 1$, we write $\frac{1}{C}U_H\coloneqq \sL_{\hat\lambda}^{-\ln C}(U_H).$
\end{dfn}
For a Liouville domain $(W,\lambda_W)$, one can associate a barcode $\cB(W,\lambda_W)$ via $S^1$-equivariant symplectic homology. Usher proved that, for any two Liouville domains $(W,\lambda_W)$ and $(V,\lambda_V)$, $d_{\rm bot}(\cB(W,\lambda_W),\cB(V,\lambda_V))\le d_{\rm SBM}((W,\lambda_W),(V,\lambda_V)).$
We next extend this estimate to the $C^0$-setting.

\begin{thm}\label{thm-barcode-extension-c0}
Let $(W,\lambda)$ be a Liouville domain. Then the assignment 
\[
\cH(W,\lambda)\to \{\text{barcodes}\}\quad \mbox{by} \quad H\mapsto \cB(H)\coloneqq \cB(W_H,\hat\lambda)
\]
extends uniquely to a map $\cB:\cH^0(W,\lambda)\to \{\text{barcodes}\}$ such that $d_{\rm bot}(\cB(H),\cB(G))\le d_{\rm SBM}(W_H,W_G)$ for all $H,G\in \cH^0(W,\lambda).$ In particular, the extended barcode map is continuous with respect to $d_{\rm SBM}$, and hence $C^0$-continuous.
\end{thm}
\begin{remark}
Theorem~\ref{thm-barcode-extension-c0} may be viewed as an application of Usher's theory of barcodes for open Liouville domains \cite{Ush22}. The Liouville vector field of $\hat\lambda=\lambda+\frac12r^2d\theta$ is $\widehat Y=Y_\lambda+\frac12r\partial_r$, and therefore $\sL_{\hat\lambda}^{-s}(w,z)=(\sL_\lambda^{-s}(w),e^{-s/2}z)$. The condition $c_H(s)>0$ implies $\sL_{\hat\lambda}^{-s}(U_H)\subset U_H$ and, for $s>0$, gives a uniform collar separating $\sL_{\hat\lambda}^{-s}(U_H)$ from $\partial U_H$. Thus $(U_H,\hat\lambda)$ is an open Liouville domain in Usher's sense. The barcode $\cB(H)$ is therefore the barcode of $(U_H,\hat\lambda)$, and for smooth $H$ it agrees with the barcode of the compact Liouville domain $(W_H,\hat\lambda)$. The inequality $d_{\rm bot}(\cB(H),\cB(G))\le d_{\rm SBM}(W_H,W_G)$ is exactly Usher's stability estimate for the fine symplectic Banach--Mazur distance.
\end{remark}

\subsection{Capacity under stabilization} \label{ssec-capacity-stab} The next topic concerns capacity. Recall that a symplectic capacity is a functional $c$ from the space of symplectic manifolds to the non-negative real numbers, satisfying certain axioms. In particular, the (ball-)normalization axiom states that
\[ c(B^{2n}(a)) = c(Z^{2n}(a)) = a \]
where, throughout the paper, 
\[
B^{2n}(a)\coloneqq\left\{z\in\bC^n\;\middle|\;\pi\sum_{j=1}^n|z_j|^2<a\right\},\qquad\overline{B}^{\,2n}(a)\coloneqq\left\{z\in\bC^n\;\middle|\;\pi\sum_{j=1}^n|z_j|^2\le a\right\},
\]
and $Z^{2n}(a)\coloneqq B^2(a)\times \bC^{n-1}$.  This axiom excludes volume from being a symplectic capacity. Moreover, the existence of such a (ball-normalized) symplectic capacity is equivalent to Gromov's celebrated non-squeezing theorem \cite{Gro85}. This leads to the following (first) symplectic capacity:
\begin{equation} \label{dfn-wG}
 w_G(M, \omega) \coloneqq  \sup\{a\,| \, \mbox{$\exists$ a symplectic embedding $(B^{2n}(a), \omega_{\rm std}) \hookrightarrow (M, \omega)$}\}. 
 \end{equation}
commonly known as the Gromov width. 

In what follows, we investigate how $w_G$ behaves under the tube construction (\ref{Usher-con}), via input Hamiltonian functions in the set (\ref{dfn-H-lambda}). 

\begin{lemma}\label{lem:Gwidth-continuity-c0-time-dependent}
For any Liouville domain $(W,\lambda)$, the map
\[
\cH^0(W,\lambda)\to [0,\infty),\qquad
H\mapsto w_G(W_H,d\hat\lambda)
\]
is continuous with respect to the $C^0$-topology.
\end{lemma}

\begin{remark} Although the target $(M, \omega)$ in \eqref{dfn-wG} is meant to be a symplectic manifold, the definition can be modified so that $(B^{2n}(a), \omega_{\rm std}) \hookrightarrow ({\rm Int}(M), \omega)$ when $M$ has boundary, or even when its boundary is not smooth. For convenience, we still denote $w_G({\rm Int}(M), \omega)$ by $w_G(M, \omega)$. \end{remark}
In particular, this implies that, for any fixed $(W, \lambda)$ and $H \in \mathcal H(W, \lambda)$, $w_G(W_{aH}, d\hat{\lambda})$ tends to zero as $a \to 0$. As another main result of this paper, the following theorem establishes certain asymptotic behaviors of $w_G$. 

\begin{thm} \label{thm-capacity} 
Let $(W,\lambda)$ be a Liouville domain, and let $H\in \cH(W,\lambda)$ be possibly time-dependent. Set
$M\coloneqq \max_{(t,w)\in S^1\times W} H(t,w),\overline H(w)\coloneqq \int_0^1 H(t,w)\,dt,\overline M\coloneqq \max_{w\in W}\overline H(w)$.
Then we have the following asymptotic behaviors:
\begin{itemize}
\item[(i)]  $\overline M\le \liminf_{a\to0^+}\frac{w_G(W_{aH},d\hat\lambda)}{a}\le \limsup_{a\to0^+}\frac{w_G(W_{aH},d\hat\lambda)}{a}\le M$. 
\item[(ii)] $\lim_{a\to\infty} w_G(W_{aH}, d\hat{\lambda}) = w_G(W\times \mathbb C, \omega_{\rm prod})$. 
\end{itemize}
In particular, if $H$ is autonomous, then $\overline H=H$, hence
$\lim_{a\to0^+}\frac{w_G(W_{aH},d\hat\lambda)}{a}=\max_{w\in W}H(w)$.
\end{thm}

Symplectic manifolds of the form $W \times \bC$ have been studied before. When $W$ is a standard (four-dimensional) toric domain, $W \times \bC$ (or more generally $W \times \bC^{N}$) provides the so-called {\it stabilized structure}, and the stabilized embedding problem---the existence of symplectic embeddings between stabilized domains---is a rich subject (see \cite{CGH23}). Here, part~(ii) of Theorem~\ref{thm-capacity} addresses a different problem: the source is a higher-dimensional symplectic ball rather than a stabilized ball. When $(M, \omega)$ is a closed symplectic manifold---in particular, when $M = \mathbb T^2$---earlier results on estimating the Gromov width of $M \times \mathbb C$ appear in \cite{Gut08}. Interestingly, within the class of symplectic manifolds (with or without boundary), the value of $w_G$ may be either finite or infinite, depending on the input. We illustrate this via the following two propositions. 

Recall that  $c_Z(W, d \lambda)$ denotes the cylindrical capacity of  the Liouville domain $(W, d\lambda)$, measuring the smallest symplectic cylinder $Z^{2n}(a)$ into which $(W, d\lambda)$ can be symplectically embedded. 

\begin{prop} \label{prop-wG-finite} The following list collects several cases of Liouville domains $(W, \lambda)$ whose stabilizations have finite Gromov width.
\begin{itemize}
\item[(i)] If $(W, \lambda)$ satisfies $w_G(W, d\lambda) = c_Z(W, d\lambda)$, then $w_G(W \times \mathbb C, \omega_{\rm prod}) = w_G(W, d \lambda)$. 
\item[(ii)] If $(W, \lambda) = (S^2\setminus \bigcup_{j=1}^n{\rm Int}(D_j^2),\lambda)$, where $D_1^2,\cdots, D_n^2\subset S^2$ are pairwise disjoint smoothly embedded closed disks and $n\ge1$, then $w_G(W \times \bC, \omega_{\rm prod}) = \int_W \omega$. 
\item[(iii)] If $(W, \lambda) = (D_g^*N, \lambda_{\rm can})$ where $N$ is a connected closed surface which is not an oriented surface of genus at least $2$, then $w_G(W \times \bC, \omega_{\rm prod})< \infty$. 
\end{itemize}
\end{prop}

\begin{ex} By \cite{CGH23}, standard examples satisfying part~(i) of Proposition~\ref{prop-wG-finite} are the monotone toric domains in $(\bR^{2n}, \omega_{\rm std})$. \end{ex}

\begin{remark} 	(1) The inequalities $w_G(W,d\lambda)\leq w_G(W\times\bC,\omega_{\rm prod})$ and $c_Z(W,d\lambda)\geq c_Z(W\times\bC,\omega_{\rm prod})$ hold by definition. Hence $w_G(W,d\lambda)=c_Z(W,d\lambda)$ implies $w_G(W\times\bC,\omega_{\rm prod})=c_Z(W\times\bC,\omega_{\rm prod})$. (2) Conclusion~(ii) of Proposition~\ref{prop-wG-finite} also holds for $S^2$ (without boundary). (3) A symplectically uniruled symplectic $4$-manifold (see Section~\ref{ssec-capacity}) admits {\rm no} closed Lagrangian submanifold of genus $g\ge 2$; see Theorem~2.2 of \cite{Wel07} and Theorem~7.3 of \cite{Wel18}. Thus part~(iii) of Proposition~\ref{prop-wG-finite} is consistent with the expectation that oriented surfaces of genus at least $2$ should exhibit different behavior.
 \end{remark}

\begin{prop}\label{prop-lift-capacity}
Let $(X,\omega_X)$ be a symplectic manifold such that
$
w_G(X,\omega_X)=\infty.
$
Let $(\Sigma,\omega_\Sigma)$ be a symplectic surface of positive genus (with or without boundary). Then
$
w_G(\Sigma\times X,\omega_\Sigma\oplus\omega_X)=\infty.
$
\end{prop}

As an immediate consequence, we obtain the following.
\begin{cor}\label{cor-torus}
Let $W=\Sigma_1\times\cdots \times \Sigma_n$, where $\Sigma_i$ is a two dimensional surface of positive genus (with or without boundary), then $w_G(W\times\bC, \omega_{\rm prod})=\infty$.	
\end{cor}

\subsection{Large-scale geometry} \label{ssec-LSG} 

In this subsection, we show that the Thompson-type metric on Hamiltonians already exhibits nontrivial large-scale geometry, even on the two-dimensional disk. The key input comes from \cite{CGH23a}, where the authors constructed a family of smooth concave toric domains in $\bC^2$ whose Banach--Mazur distances grow linearly with the $\ell^\infty$-distance of the parameter.

Our task is to transfer that construction to the Hamiltonian setting considered in this paper. The key observation is that the toric domains from \cite{CGH23a} can be expressed as graph-type domains over a fixed interval. This makes it possible, after a simple rescaling, to realize them as tube domains associated to explicit radial Hamiltonians on the 2-disk $\overline{B}^{\,2}(a)$. Once this is done, Theorem~\ref{thm-estimation} converts the Banach--Mazur lower bound for those domains into a lower bound for the Thompson-type metric $d_{\rm T}$ from Definition~\ref{dfn-T-type}.

We now recall the precise geometric input that will be used.

\begin{prop}[Section 2.4 in \cite{CGH23a}] \label{prop-CGH}
Fix $N \ge 1$. Then there exist constants  $c_N>0$ and a family of smooth concave toric domains $\{X_v\}_{v\in\bR^N}\subset \bC^2$ such that 
\[ d_{\rm BM}(X_v,X_w)\ge c_N\|v-w\|_\infty\]
for any $v, w \in \bR^N$. In particular, for each $v \in \bR^N$, there exists a smooth concave function $f_v:[0,1]\to[0,\infty)$ such that $f_v(1)=0$ and
$X_v=\{(u,z)\in\bC^2 \mid \pi|u|^2\le 1,\ \pi|z|^2\le f_v(\pi|u|^2)\}$.
\end{prop}

Based on the family of domains constructed in Proposition~\ref{prop-CGH}, we are able to construct explicit Hamiltonians on the disk whose pairwise $d_{\rm T}$-distances are bounded below by the $\ell^\infty$-distance of the parameters.
\begin{thm}\label{thm-large-scale-dT} 
For any fixed $a>0$ and $N\ge 1$, there exist a map $\Phi_{N,a}:\bR^N \to \mathcal H^+(\overline{B}^{\,2}(a),\lambda_{\mathrm{std}})$, $v \mapsto \Phi_{N,a}(v)$ and a constant $c_N>0$, depending only on $N$, such that
\[ d_{\rm T}(\Phi_{N,a}(v),\Phi_{N,a}(w))\ge c_N\|v-w\|_\infty\]
for all $v,w\in\bR^N$. 
\end{thm}
The goal in this subsection is to import the large-scale geometry of the toric domains from Proposition~\ref{prop-CGH} into the Hamiltonian space on a fixed disk. For each parameter $v\in\bR^N$, the concave toric domain $X_v$ is determined by a profile function $f_v$. After rescaling this profile to $[0,a]$, we define a radial Hamiltonian $H_v$ on $\overline{B}^{\,2}(a)$ by $H_v(u)\coloneqq \widetilde f_v(\pi|u|^2)$. Then the associated tube domain is exactly the Liouville rescaling of the toric domain, namely ${\rm U}(H_v)=aX_v$. Thus the Banach--Mazur lower bound for the family $\{X_v\}$ becomes a lower bound for the symplectic Banach--Mazur distance between the tube domains. Finally, Theorem~\ref{thm-estimation} converts this domain-level lower bound into the desired lower bound for the Thompson-type metric $d_{\rm T}$ on Hamiltonians.

Note that, although the function space $\mathcal H^+(\overline{B}^{\,2}(a),\lambda_{\mathrm{std}})$ tends to be large itself, due to the group action of ${\rm Symp}(\overline{B}^{\,2}(a), \lambda_{\rm std})$ in Definition~\ref{dfn-T-type}, it is not immediately clear whether $d_{\rm T}$ retains enough freedom. Theorem~\ref{thm-large-scale-dT} confirms that $(\mathcal H^+(\overline{B}^{\,2}(a),\lambda_{\mathrm{std}}), d_{\rm T})$ is at least as large as any high-dimensional Euclidean space $(\bR^N, \|\cdot\|_{\infty})$. The construction (and the estimate of $d_{\rm BM}$) in Proposition~\ref{prop-CGH} is based on ECH capacities and their Weyl law, and is therefore restricted to four-dimensional domains. This explains why, in Theorem~\ref{thm-large-scale-dT}, we can establish the large-scale geometry only on $\overline{B}^{\,2}(a)$. It would be interesting to explore whether the conclusion of Theorem~\ref{thm-large-scale-dT} extends to higher-dimensional base Liouville domains $(W, \lambda)$. 

\subsection{Systolic ratio estimation}
Given a Liouville domain $(W, \lambda)$, denote by $\Sys(\partial W)$ the systole of the contact boundary $(\partial W,\lambda|_{\partial W})$, that is, the minimal period of a closed Reeb orbit of $\lambda|_{\partial W}$. Similarly, for $H \in \mathcal H(W, \lambda)$, denote by $\Sys(\partial W_H)$ the systole of $(\partial W_H,\alpha_H)$. Recall the defining property $\tau_{\lambda,H}(= \iota_{X_H}\lambda + H) >0$ from (\ref{H-con}), and set $\tau_{\min}(H)\coloneqq \min_{S^1\times W}\tau_{\lambda,H}$. Recall also that the normalized contact volume of $(\partial W_H,\alpha_H)$ is defined by
\[
\Vol(\partial W_H,\alpha_H)\coloneqq\frac{1}{n!}\int_{\partial W_H}\alpha_H\wedge(d\alpha_H)^n.
\]
The systolic ratio of $W_H$ is defined as
\[
\rho(W_H)\coloneqq\frac{\Sys(\partial W_H)^{n+1}}{\Vol(\partial W_H,\alpha_H)}.
\]
 We then have the following estimate, whose proof is given here, as it is fairly short and can be regarded as a direct consequence of Corollary~\ref{cor-Reeb-Ham}.

\begin{prop} \label{prop-sys-ratio} 
Let $(W^{2n}, \lambda)$ be a $2n$-dimensional Liouville domain. Then, for any $H \in \mathcal H(W, \lambda)$, 
	\[
\frac{n!\,\min\{\Sys(\partial W),\tau_{\min}(H)\}^{n+1}}{(n+1)\Cal(H)}\le\rho(W_H)\le\frac{n!\,\Sys(\partial W)^{n+1}}{(n+1)\Cal(H)} .
\]
\end{prop}

\begin{proof}
By Corollary~\ref{cor-Reeb-Ham}, the closed Reeb orbits on $\partial W_H$ are of two kinds. First, the invariant subset $\partial W\times\{0\}\subset \partial W_H$ carries exactly the original Reeb flow, so the closed Reeb orbits there are precisely those on $\partial W$; in particular, their periods are bounded below by $\Sys(\partial W)$. Second, the Reeb orbits arising from integer-periodic Hamiltonian orbits of $H$ have periods $T_{t_0,w_0,k}=\int_0^k\tau_{\lambda,H}(t_0+\sigma,\varphi_H^\sigma(t_0,w_0))\,d\sigma.$ Hence these periods are bounded below by $\tau_{\min}(H)$. Therefore $\Sys(\partial W_H)\ge \min\{\Sys(\partial W),\tau_{\min}(H)\}$. On the other hand, since $\partial W\times\{0\}$ is an invariant subset of $\partial W_H$, we also have  $\Sys(\partial W_H)\le \Sys(\partial W)$.

It remains to compute the normalized contact volume. By Stokes' theorem, $\int_{\partial W_H}\alpha_H\wedge(d\alpha_H)^n=\int_{W_H}(d\hat\lambda)^{n+1}$. Therefore $\Vol(\partial W_H,\alpha_H)=\frac{1}{n!}\int_{W_H}(d\hat\lambda)^{n+1}=(n+1)\Vol(W_H)$. Using the identity $\Vol(W_H)=\frac{1}{n!}\Cal(H)$, we obtain $\Vol(\partial W_H,\alpha_H)=\frac{n+1}{n!}\Cal(H)$.
Substituting this identity into the two estimates for $\Sys(\partial W_H)$ gives
\[
\frac{n!\,\min\{\Sys(\partial W),\tau_{\min}(H)\}^{n+1}}{(n+1)\Cal(H)}\le\rho(W_H)\le\frac{n!\,\Sys(\partial W)^{n+1}}{(n+1)\Cal(H)},
\]
as requested.
\end{proof}

Set $\tau_{\max}(H) \coloneqq \max_{S^1\times W}\tau_{\lambda,H}$, then a similar argument as in Proposition \ref{prop-sys-ratio} yields the following result.  

\begin{cor}\label{cor-sys-ratio-euler}
Let $(W^{2n},\lambda)$ be a Liouville domain with $\chi(W)\neq 0$. Then for any $H\in\mathcal H(W,\lambda)$ we have $\Sys(\partial W_H)\leq \min\{\Sys(\partial W),\tau_{\max}(H)\}$. Consequently,
\[
\frac{n!\,\min\{\Sys(\partial W),\tau_{\min}(H)\}^{n+1}}{(n+1)\Cal(H)}\leq\rho(W_H)\leq
\frac{n!\,\min\{\Sys(\partial W),\tau_{\max}(H)\}^{n+1}}
{(n+1)\Cal(H)}.
\]
\end{cor}

\begin{proof}
Since the time-one map $\varphi_H^1\colon W\to W$ is Hamiltonian isotopic, and hence homotopic, to the identity, its Lefschetz number satisfies $L(\varphi_H^1)=L(\id_W)=\chi(W)\neq 0$. By the Lefschetz fixed point theorem, $\varphi_H^1$ has a fixed point $w_0\in W$; see \cite[Theorem~2C.3]{Hat02}.

If $w_0\in {\rm Int}(W)$, then Corollary~\ref{cor-Reeb-Ham}, applied with $k=1$, gives a closed Reeb orbit on $\partial W_H$ of period $T_{t_0,w_0,1}=\int_0^1\tau_{\lambda,H}\bigl(\sigma,\varphi_H^\sigma(t_0,w_0)\bigr)\,d\sigma\leq \tau_{\max}(H)$. So, $\Sys(\partial W_H)\leq \tau_{\max}(H)$.

If $w_0\in\partial W$, then, since $H$ is autonomous near $\partial W$ and $H|_{\partial W}\equiv 0$, the Hamiltonian vector field restricts to a reparametrization of the Reeb vector field of $(\partial W,\lambda|_{\partial W})$. Thus the fixed point of $\varphi_H^1$ gives a closed Reeb orbit on the binding $\partial W\times\{0\}\subset\partial W_H$ with period bounded above by $\tau_{\max}(H)$. Again, $\Sys(\partial W_H)\leq \tau_{\max}(H)$. On the other hand, the binding $\partial W\times\{0\}$ carries the original Reeb flow, so $\Sys(\partial W_H)\leq \Sys(\partial W)$. Therefore $\Sys(\partial W_H)\leq \min\{\Sys(\partial W),\tau_{\max}(H)\}$.
\end{proof}

The upper bound in Proposition~\ref{prop-sys-ratio} is not particularly useful, since whenever $\Cal(H) \to \infty$, we obtain $\rho(W_H) \to 0$, which is essentially obvious. The lower bound in the same proposition is more useful: if a sequence of Hamiltonians $\{H_n\}_{n \in \mathbb{N}} \subset \mathcal H(W, \lambda)$ satisfies $\min_{S^1 \times W}\tau_{\lambda, H_n}\geq \kappa$ for some positive $\kappa$ independent of $n$, while $\Cal(H_n) \to 0$, then we obtain a sequence of Liouville domains $\{W_{H_n}\}_{n \in \mathbb{N}}$ with $\rho(W_{H_n}) \to \infty$. Note that, unlike the Anosov--Katok conjugation construction in \cite{ABHS18}, where the input Hamiltonians may take negative values, our input Hamiltonians $H$ are by definition always non-negative; hence the volume condition $\Cal(H) \to 0$ is in fact a rather strong constraint. This makes it more difficult to construct a sequence of Hamiltonians $\{H_n\}_{n \in \mathbb{N}}$ yielding divergent $\{\rho(W_{H_n})\}_{n \in \mathbb{N}}$. For instance, when $(W, \lambda_{\rm std}) \subset (\bC^n, \lambda_{\rm std})$ is a Liouville domain in $\bC^n$, it seems impossible to construct such a sequence using only radial Hamiltonian functions $H$, i.e., $H(t,z) = H(\pi|z|^2)$, for which $\tau_{\lambda_{\rm std}, H}$ admits a clear geometric interpretation (see the proof of Theorem~\ref{thm-large-scale-dT}). In fact, for radial input Hamiltonians, the condition $\Cal(H_n) \to 0$ implies that $\min_{S^1 \times W}\tau_{\lambda, H_n} \to 0$ as $n \to \infty$. 

\begin{remark} From the proof of Proposition~\ref{prop-sys-ratio}, one can sharpen the lower bound in the resulting estimate by replacing $\tau_{\rm min}$ by $\min T_{t_0, w_0, k}$, i.e., the minimal period of integer-periodic Hamiltonian orbits of $H$. Still, it is a nontrivial task to control $\min T_{t_0, w_0, k}$ so that it remains relatively large (even as $\Cal(H)$ approaches $0$). 
\end{remark}

\subsection{Categorification}

Generalizing the tube construction above, given a Liouville domain $(W, \lambda)$, consider two input Hamiltonian functions $H_-, H_+ \in \mathcal H(W, \lambda)$---that is, satisfying condition (\ref{H-con})---and further required to satisfy $H_- \leq H_+$ pointwise. One can then consider 
\begin{equation} \label{tunnel-cons}
W_{H_-, H_+} \coloneqq \left\{(w,z) \in W \times \bC \, |\, z = re^{2\pi i t} \,\,\mbox{and}\,\, H_-(t, w) \leq \pi|z|^2 \leq H_+(t,w) \right\}.
\end{equation}
Suppose in addition that $H_-(t,w)<H_+(t,w)$ for $(t,w)\in S^1\times {\rm Int}(W)$. Then $W_{H_-,H_+}$ is a manifold with corners that can be expressed as $W_{H_-, H_+} = \overline{W_{H_+} \backslash W_{H_-}}$, where 
\[ \partial W_{H_-, H_+} = \partial W_{H_-} \cup_{\partial W \times \{0\}} \partial W_{H_+}. \]
We call $W_{H_-, H_+}$ a {\it tunnel construction}. When $H_-= 0$ (even though $0 \notin \mathcal H(W, \lambda)$), the tunnel construction reduces to the tube construction $W_{H_+}$ in (\ref{Usher-con}). In the other extreme case $H_- = H_+ (= H)$, the tunnel $W_{H_-, H_+}$ reduces to the contact manifold $(\partial W_{H}, \hat{\lambda}|_{\partial W_H})$. 

This picture is useful for keeping track of the fiberwise nature of the construction. For each base point $w\in W$, the fiber $(W_{H_1,H_2})_w$ is the annular region in the $\bC$-coordinate cut out by the two radii determined by $H_1$ and $H_2$; over the boundary $\partial W$, this annular fiber collapses to the common corner $\partial W\times\{0\}$. Thus the tunnel should be viewed as a family of annuli over the base $W$, interpolating between the lower and upper tube boundaries $\partial W_{H_1}$ and $\partial W_{H_2}$.

\begin{figure}[ht]
\centering
\includegraphics[width=0.70\textwidth]{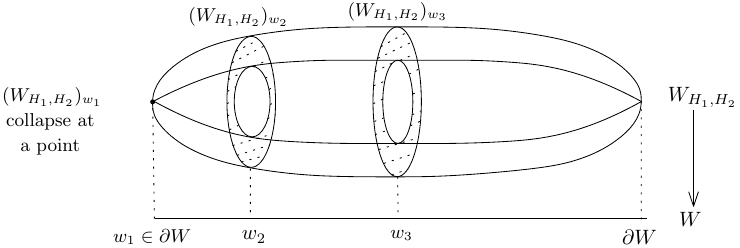}
\caption{A schematic picture of the tunnel construction $W_{H_1,H_2}$}
\label{fig:tunnel-categorification}
\end{figure}
\medskip

The tunnel construction provides a geometric realization of the Thompson-type metric introduced in Definition~\ref{dfn-T-type}. To elaborate, we first introduce the following quantity associated to a tunnel construction. 

\begin{dfn} \label{dfn-thickness}
Let $(W, \lambda)$ be a Liouville domain. For $H_-, H_+ \in \mathcal H(W, \lambda)$ with $H_- \leq H_+$, the \emph{thickness} of the tunnel $W_{H_-,H_+}$ defined in (\ref{tunnel-cons}) is given by
\begin{equation} \label{equ-thickness}
\ell(W_{H_-,H_+})\coloneqq \ln\left(\sup_{(t,w)\in S^1\times {\rm Int}(W)}\frac{H_+(t,w)}{H_-(t,w)}\right)\in [0,\infty].
\end{equation}
In the degenerate case $H_- = H_+$ (where $W_{H_-, H_+}$ is a contact manifold), set $\ell(W_{H_-,H_+}) =0$. 
\end{dfn}

The following lemma shows how the Thompson-type metric $d_{\rm T}$ can be identified with the thickness of tunnels. 
\begin{lemma} \label{lemma-dT-thick}
Let $(W,\lambda)$ be a Liouville domain. For any $H,G\in \mathcal H(W,\lambda)$, the Thompson-type metric $d_{\rm T}(\cdot,\cdot)$ from Definition~\ref{dfn-T-type} and the thickness $\ell(W_{\cdot,\cdot})$ of a tunnel from Definition~\ref{dfn-thickness} satisfy the following relation:
\begin{equation} \label{dT-thickness}
d_{\rm T}(H,G)=\inf\left\{\ell(W_{H_-,H_+}) \,\bigg|\, \begin{array}{l} H_- \leq \min\{H,G\circ\psi\}\leq \max\{H,G\circ\psi\}\leq H_+,\\ H_-,H_+\in\mathcal H(W,\lambda),\ \psi\in{\rm Symp}(W,\lambda)\end{array}\right\}.
\end{equation}
Here, $(G\circ\psi)(t,w)\coloneqq G(t,\psi(w))$. In other words, $d_{\rm T}(H,G)$ measures the smallest thickness of a tunnel containing, up to symplectomorphism, both $\partial W_H$ and $\partial W_G$ as interior level sets.
\end{lemma}

\begin{remark} When the input Hamiltonian functions $H, G \in \mathcal H(W, \lambda)$ already satisfy $H \leq G$, taking $\psi = {\rm id}$ in Lemma~\ref{lemma-dT-thick} yields $d_{\rm T}(H, G) \leq \ell(W_{H, G})$, where $W_{H, G}$ is well-defined. \end{remark}

Geometrically, the tunnel construction admits a {\it composition} in the following sense: if $H_1 \leq H_2 \leq H_3 \in \mathcal H(W, \lambda)$, then 
\[ W_{H_1, H_3} \simeq W_{H_1, H_2} \cup_{\partial W_{H_2}} W_{H_2, H_3}, \]
where the gluing is carried out trivially along the contact manifold $\partial W_{H_2}$, and ``$\simeq$'' denotes isomorphism as Liouville cobordisms. Meanwhile, it is readily verified that the thickness is subadditive: $\ell(W_{H_1, H_3}) \leq \ell(W_{H_1, H_2}) + \ell(W_{H_2, H_3})$. This suggests that the Thompson-type metric $d_{\rm T}$ can be formulated categorically, with quantitative measurement given by the thickness function. In fact, such a ``metric'' category was introduced---as a categorification of a metric space---by Lawvere in \cite{Law73}. In short, a {\it Lawvere metric category} is a category\footnote{In general, a {\it normed category} means associating a norm to every hom-set of the category. Here, a Lawvere metric category is simply a normed category in which every hom-set contains exactly one element.} enriched in the monoidal poset $([0, \infty], \geq, +, 0)$. The following result, one of the main results of this paper, explicitly lists the axioms of such a category. 

\begin{thm} \label{thm-cat} Let $(W, \lambda)$ be a Liouville domain. Then the function space $\mathcal H(W, \lambda)$ from (\ref{dfn-H-lambda}) can be equipped with a Lawvere metric category structure, denoted by ${\rm Tun}(W, \lambda)$. Explicitly, ${\rm Ob}({\rm Tun}(W,\lambda)) \coloneqq  \mathcal H(W, \lambda)$, and for $H, G \in {\rm Ob}({\rm Tun}(W,\lambda))$,
\begin{equation} \label{hom-tun-cat}
{\rm hom}_{{\rm Tun}(W,\lambda)}(H,G)\coloneqq \inf_{\psi\in{\rm Symp}(W,\lambda)} \max\left\{\ln\left(\sup_{(t,w)\in S^1\times{\rm Int}(W)} \frac{H(t,w)}{G(t,\psi(w))}\right),0\right\}.
\end{equation}
Then the following two properties hold:
\begin{itemize}
\item[(i)] ${\rm hom}_{{\rm Tun}(W, \lambda)}(H, H) =0$ for any $H \in {\rm Ob}({\rm Tun}(W,\lambda))$;
\item[(ii)] ${\rm hom}_{{\rm Tun}(W, \lambda)}(H_1, H_3) \leq {\rm hom}_{{\rm Tun}(W, \lambda)}(H_1, H_2) + {\rm hom}_{{\rm Tun}(W, \lambda)}(H_2, H_3)$ for any objects $H_1, H_2, H_3 \in  {\rm Ob}({\rm Tun}(W,\lambda))$. 
\end{itemize}
Moreover, one can symmetrize ${\rm hom}_{{\rm Tun}(W, \lambda)}$ by defining
\[ {\rm hom}_{{\rm Tun}(W, \lambda)}^{\rm sym}(H, G) \coloneqq \inf_{\psi\in{\rm Symp}(W,\lambda)} \ln \left(\sup_{(t,w)\in S^1\times{\rm Int}(W)} \max\left\{\frac{H(t,w)}{G(t,\psi(w))}, \frac{G(t,\psi(w))}{H(t,w)}\right\}\right). \]
Then ${\rm hom}_{{\rm Tun}(W, \lambda)}^{\rm sym}(H, G) = d_{\rm T}(H, G)$, the Thompson-type metric in Definition~\ref{dfn-T-type}.

\end{thm}

By Lemma~\ref{lemma-dT-thick} and the last conclusion of Theorem~\ref{thm-cat}, the symmetrization ${\rm hom}_{{\rm Tun}(W, \lambda)}^{\rm sym}$ in Theorem~\ref{thm-cat} admits a geometric interpretation in terms of the tunnel construction in (\ref{tunnel-cons}) and the thickness function in (\ref{equ-thickness}).

\begin{remark} In a similar fashion, one can formulate the pseudo-metric spaces discussed above, $(\sD_{2n}, d_{\rm SBM})$ and $(\{\mbox{contact 1-forms}\}, d_{\rm CBM})$, as Lawvere metric categories. Theorems~\ref{thm-estimation} and~\ref{thm-SBM-CBM-ln} then show that the maps $H \in \mathcal H^+(W, \lambda) \mapsto U(H)$ and $H \in \mathcal H^+(W, \lambda) \mapsto \alpha_H$ are ``1-Lipschitz'' functors with respect to these Lawvere metric structures. We leave the verification of the details to the interested reader. \end{remark}

\subsection{Floer homology of the tube} \label{ssec-Floer-tube} The tube construction yields a Liouville domain $W_H$. Recall that its Liouville completion $\widehat{W_H}$, obtained by gluing a cylindrical end $[1, \infty) \times \partial W_H$ to $W_H$, also admits a Floer theory called symplectic homology, with an equivariant version denoted by $\SH_*(W_H)$. By definition, the generators of $\SH_*(W_H)$ can be identified with the closed Reeb orbits on $\partial W_H$. Slowing down the picture, one can view ${\rm SH}_*(W_H)$ as a limit of Hamiltonian Floer homologies with respect to Viterbo-type (admissible) functions $K_N$, where $K_N$ is $C^2$-small inside ${\rm int}(W_H)$ and linear with admissible slope $N$ on a sufficiently large end of $[1, \infty) \times \partial W_H$. As usual, this Floer homology is denoted by ${\rm HF}_*(K_N)$. By Corollary~\ref{cor-Reeb-Ham}, the generators of the corresponding chain complex ${\rm CF}_*(K_N, J)$ (well-defined for a generic choice of almost complex structure $J$) can be classified into the following three subclasses:
 \begin{equation} \label{orbit-subclass}
  \mathcal P(K_N)=\mathcal P_N^{\mathrm{cst}}\sqcup \mathcal P_N^D\sqcup \coprod_{1\leq k\leq \kappa(N)}\mathcal P_N^{(k)}.
\end{equation}
Here $\kappa(N)$ denotes the maximal winding number that can occur among the Hamiltonian orbits of $K_N$ with admissible slope $N$, $D$ denotes $W \times \{0\}$, a divisor (with boundary) in $W_H$, and for an orbit $\gamma$ not contained in $D$, its {\it winding number} is set to be $k(\gamma)\coloneqq (2\pi)^{-1}\int_\gamma d\theta _{\bC}$; that is, the rotation number around the $z$-axis in the $\bC$-component of the tube construction. In (\ref{orbit-subclass}), $\mathcal P_N^{\mathrm{cst}}$ denotes the off-divisor constant orbits in a compact region, $\mathcal P_N^D$ denotes the orbits contained in $D$, and $\mathcal P_N^{(k)}$ denotes the integer-periodic Hamiltonian orbits (of the Hamiltonian flow generated by $H$) with winding number $k\geq 1$.  

We declare the constant orbits in $\mathcal P_N^{\mathrm{cst}}$ and the divisor orbits in $\mathcal P_N^D$ to have winding number zero. Then, as $\bK$-modules, we have the following decomposition:
\begin{equation} \label{equ-orbit-decomp}
{\rm CF}_*(K_N, J) = \bK\langle \mathcal P_N^{\mathrm{cst}}\sqcup \mathcal P_N^D\rangle\oplus  \bigoplus_{1 \leq k \leq \kappa(N)}  \bK\langle \mathcal P_N^{(k)}\rangle =:  C_{N,*}^{(0)}\oplus \bigoplus_{1\leq k\leq \kappa(N)} C_{N,*}^{(k)}.
\end{equation}
We call the Floer datum $(K_N,J)$ \emph{$D$-compatible} if $X_{K_N}|_D\subset TD$ and $JTD=TD$, and \emph{regular} if the moduli spaces used to define the differential are cut out transversely. A local argument, due to the dimension reason from $D$, shows that such Floer data is generic. 

Upgrading to a chain complex, for any two such generators $\gamma_-, \, \gamma_+$ in $\mathcal P(K_N)$, one investigates the standard Floer cylinders from $\gamma_-$ to $\gamma_+$ in $\widehat{W_H}$ perturbed by the Hamiltonian vector field of $K_N$. The following lemma provides a topological constraint on these Floer cylinders. 

\begin{lemma} 
\label{prop:winding-monotonicity}
Let $u\colon \bR\times S^1\ra \widehat{W_H}$ be a finite-energy Floer cylinder for a regular $D$-compatible datum $(K_N,J)$, with asymptotics $\lim_{s\to-\infty}u(s,\cdot)=\gamma_-$ and $\lim_{s\to+\infty}u(s,\cdot)=\gamma_+$ in $\mathcal P(K_N)$.  Suppose that neither asymptotic orbit is contained in $D$. Then $$u\cdot D=k(\gamma_+)-k(\gamma_-)\geq 0.$$ Moreover, $u\cdot D=0$ if and only if $u\cap D=\varnothing$. In particular, when restricted to Floer cylinders connecting orbits in $\coprod_{1\leq k\leq \kappa(N)}\mathcal P_N^{(k)}$, the Floer differential $\partial_{K_N, J}$ for the Floer chain complex ${\rm CF}_*(K_N, J)$ {\rm cannot} decrease the winding number.
\end{lemma}
The relation in Lemma~\ref{prop:winding-monotonicity} is illustrated schematically in Figure~\ref{fig:winding-monotonicity}. The left-hand panel depicts a Floer cylinder disjoint from the divisor; in this case the intersection number vanishes and the two asymptotic orbits have the same winding number. The right-hand panel depicts a cylinder with positive intersections with $D$; by positivity of intersections, each crossing contributes $+1$, and the total contribution is exactly the increase of the winding number from $\gamma_-$ to $\gamma_+$.
\begin{figure}[ht]
\centering
\includegraphics[width=0.90\textwidth]{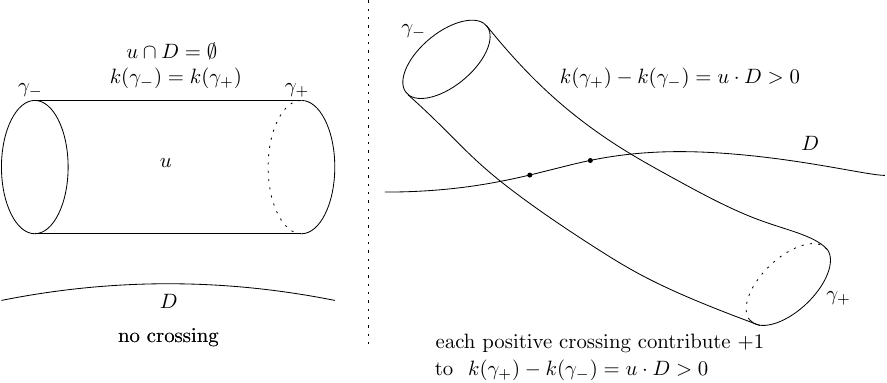}
\caption{Floer cylinders relative to the divisor $D$}
\label{fig:winding-monotonicity}
\end{figure}
Importantly, note that Lemma~\ref{prop:winding-monotonicity} does {\it not} exclude the possibility that a Floer cylinder connecting orbits in $\coprod_{1\leq k\leq \kappa(N)}\mathcal P_N^{(k)}$ may intersect $D$. Therefore, without invoking further constraints, we can at present only conclude that, with respect to the decomposition in (\ref{equ-orbit-decomp}), the differential $\partial_{K_N, J}$ is merely lower triangular: 
\begin{equation}
\label{equ-matrix-decomp}
\partial_{K_N,J}\big|_{\bigoplus_{1\leq k\leq \kappa(N)} C_{N,*}^{(k)}}
=
\begin{pmatrix}
\partial_N^{(1)} & 0 & 0 & \cdots & 0\\
\ast & \partial_N^{(2)} & 0 & \cdots & 0\\
\ast & \ast & \ddots & \ddots & \vdots\\
\vdots & \vdots & \ddots & \partial_N^{(\kappa(N)-1)} & 0\\
\ast & \ast & \cdots & \ast & \partial_N^{(\kappa(N))}
\end{pmatrix}.
\end{equation}
In particular, the entries ``$\ast$'' in the matrix (\ref{equ-matrix-decomp}) precisely capture those Floer cylinders connecting $\gamma_- \in \mathcal P_N^{(k)}$ and $\gamma_+ \in \mathcal P_N^{(\ell)}$ with different winding numbers $k \neq \ell$ (so that these cylinders must pass through the divisor $D$).

\medskip

Back to ${\rm CF}_*(K_N, J)$ in (\ref{equ-orbit-decomp}), let us denote 
\begin{align*} \label{equ-orbit-decomp-2}
{\rm CF}_*(K_{\infty}, J) :& = \varinjlim_{N \in \bN} {\rm CF}_*(K_{N}, J)\\
& =\varinjlim_{N \in \bN}\left( \bK\langle \mathcal P_N^{\mathrm{cst}} \sqcup \mathcal P_N^D\rangle\oplus  \bigoplus_{1 \leq k \leq \kappa(N)}  \bK\langle \mathcal P_N^{(k)}\rangle\right)\\
& = \underbrace{\left(\varinjlim_{N \in \bN}\bK\langle \mathcal P_N^{\mathrm{cst}}\rangle \right)}_{\text{cst-orbits}}
\oplus \underbrace{\left(\varinjlim_{N \in \bN} \bK\langle\cP_N^D\rangle \right)}_{\text{$D$-orbits}} \oplus \underbrace{\left( \varinjlim_{N \in \bN} \bigoplus_{1 \leq k \leq \kappa(N)}  \bK\langle \mathcal P_N^{(k)}\rangle\right)}_{\text{Ham-orbits}}
\end{align*} 
where the direct limit in the first line is taken over inclusions ${\rm CF}_*(K_{N}, J) \to {\rm CF}_*(K_{N+1}, J)$. To simplify the setting, by working on strictly positive filtration levels, one can ignore the off-divisor constant orbits, the part labelled by ``cst-orbits''. Then the remaining generators consist of all the closed Reeb orbits on the contact boundary $\partial W$ (which come from the part labelled by ``$D$-orbits'') and all the integer-periodic Hamiltonian orbits of $H$ in the interior of $W$ (which come from the part labelled by ``Ham-orbits''). 

Again, when restricting on the Hamiltonian orbits part, the differential, now naturally denoted by $\partial_{K_{\infty}, J}|_{\tiny{\left<\mbox{Ham-orbits}\right>}}$, is in the same form as in (\ref{equ-matrix-decomp}). However, the following Example \ref{ex-ell-Floer-diff} shows that Reeb orbits in divisor $D$ and Hamiltonian orbits must interacted. In other words, $\partial_{K_{\infty}, J}$ will mix different types of Reeb orbits on the contact boundary $\partial W_H$, and the classification via winding numbers as in the defining equation in (\ref{equ-orbit-decomp}) does {\it not} necessarily split the differential $\partial_{K_{\infty}, J}$.

\begin{ex}\label{ex-ell-Floer-diff} 
In this example, we revisit the standard four-dimensional ellipsoid \(E(1,a)\) from the perspective of the tube construction in (\ref{Usher-con}). We take the standard two-dimensional disk \((\overline{B}^2(1), \lambda_{\rm std}) \subset (\mathbb{C}, \lambda_{\rm std})\) as the initial Liouville domain and denote it by \((W,\lambda)\) for notational consistency. Recall that the symbol `\(1\)' in \(\overline{B}^2(1)\) indicates that its standard Euclidean area is equal to \(1\) (see the convention at the beginning of Section \ref{ssec-capacity-stab}). Fix an irrational scalar \(a>1\), and consider the following radially symmetric Hamiltonian function on \(\overline{B}^2(1)\), written in coordinates \(w=(r,\theta)\):
\[
H(r,\theta)=a(1-\pi r^2)=a(1-\pi |w|^2), \qquad r\in[0,1].
\]
By (\ref{Usher-con}), we obtain \(W_H = \{(w,z) \,| \, \pi|z|^2 \leq a(1- \pi|w|^2) \}\), which is precisely \(E(1,a)\). 

For the generators of \(\mathrm{CF}_*(K_\infty, J)\), apart from the constant ones, we have the following:
\begin{itemize}
\item[(i)] \(D\)-orbits: iterates of the full \(1\)-rotation along the boundary \(\partial\overline{B}^2(1)\), denoted by \(\gamma_D^k\) for \(k\in\mathbb{N}\), with periods \(\mathcal A(\gamma_D^k)=k\).
\item[(ii)] Ham-orbits: iterates of the simple orbit arising from the unique critical point of \(H(r,\theta)\) at the origin of \(\overline{B}^2(1)\), denoted by \(\gamma_H^\ell\) for \(\ell\in\mathbb{N}\), with periods \(\mathcal A(\gamma_H^{\ell})=\ell a\).
\end{itemize}
It is standard that their Conley–Zehnder indices are given by
\begin{equation} \label{ellipsoid-indices}
\mathrm{CZ}(\gamma_D^k)=2k+1+2\lfloor k/a\rfloor,\qquad
\mathrm{CZ}(\gamma_H^\ell)=2\ell+1+2\lfloor \ell a\rfloor.
\end{equation} 
These formulas can be recovered in the tube setting as in Proposition 4.3 and Proposition 4.5 of \cite{Ush22}. Meanwhile, the winding number of \(\gamma_H^{\ell}\) is equal to \(\ell\), hence in particular non-zero. Therefore, Lemma \ref{prop:winding-monotonicity} is trivial in this case, since the winding number grows proportionally to the action of the orbits \(\gamma_H^{\ell}\), plus that the positive energy of any connecting Floer cylinder is non-negative. 

Moreover, the restriction of the differential \(\partial_{K_{\infty}, J}\) (counting only index 1 Floer cylinders) to the subspace generated by the Ham-orbits satisfies
\begin{equation*}
\partial_{K_{\infty}, J}|_{\left<\mathrm{Ham\text{-}orbits}\right>} = 0
\end{equation*}
since the collection of Ham-orbits consists precisely of the \(\gamma_H^{\ell}\), and by the index formula (\ref{ellipsoid-indices}) above, the Conley–Zehnder indices of any two distinct such orbits differ by at least \(2\). Consequently, if the Ham-orbits alone formed a chain complex, it would be simply an infinite direct sum of copies of the ground field \(\mathbb{K}\). On the other hand, it is a standard fact that the positive symplectic homology of \(E(1,a)\) is one-dimensional, i.e. \({\rm SH}_+(E(1,a))\cong \mathbb{K}\). This forces infinitely many homological cancellations to occur between the \(\gamma_H^\ell\) and the \(\gamma_D^k\).
\end{ex}

\jznote{Notwithstanding the complexity of the differential $\partial_{K_{\infty}, J}$, one may pursue the following two approaches. First, one can reduce the orbit pool in the construction of (positive) symplectic homology by imposing homotopy constraints (see either \cite{PS16} or \cite{SZ21}), since Floer cylinders preserve the homotopy type of orbits. This requires that the initial Liouville domain $(W, \lambda)$ possess sufficient topological complexity, e.g., a high-genus surface with disks removed. Second, one can investigate more deeply the mixed-type Floer differential relations between $D$-orbits and Ham-orbits via symplectic field theory (SFT). More concretely, one may consider a neighborhood of the divisor $D$ inside $W_H$, perform neck-stretching for a Floer cylinder, and seek quantitative restrictions arising from SFT-compactness (cf.~\cite{DL19} in the setting of closed ambient symplectic manifolds). These directions will be investigated in detail in the future work \cite{FZ-Floer}.}

\section*{Acknowledgement}
This work grew out of the China-Israel joint research project on \emph{Symplectic Topology and Chaotic Dynamics}, supported by NSFC Grant No.~12361141812. Throughout the writing process, we have been grateful to the Chinese team members---Jinxin Xue, Huagui Duan, and Jian Wang---for their steadfast support and insightful discussions. We thank Leonid Polterovich, representing the Israeli team, for his generous guidance. \jznote{We are also grateful to Leonid Polterovich and Michael Usher for their valuable feedback on an earlier draft of this paper. Finally, we thank the organizer, Wentian Kuang, of the Workshop on Celestial Mechanics and Hamiltonian Systems (June~2026), for the opportunity to present parts of this work as it was in
preparation.} The second author is also partially supported by the National Key R\&D Program of China (Grant No.~2023YFA1010500) and NSFC Grants No.~12301081 and No.~12511540054.

\section{Background}

 In this section, we briefly recall the dynamics on Liouville domains and study the properties of the domain $(W_H, \hat{\lambda})$ constructed in \eqref{Usher-con}. In addition, we establish a few preparatory propositions that will be used in the proofs of our main theorems.
 
\subsection{Tube construction revisited} A {\em Liouville domain} $(W,\lambda)$ is a compact exact symplectic manifold with boundary, where the Liouville vector field $Y_{\lambda}$, defined by $\iota_{Y_\lambda}d\lambda=\lambda$, points outwards along $\partial W$. Therefore for all $t<0$ the time-$t$ Liouville flow  $\sL_{\lambda}^t\colon W\ra W$ is well-defined.  The {\em skeleton} of $(W,\lambda)$ is defined by
\[
{\rm Sk}(W)\coloneqq\bigcap_{t> 0}\sL_{\lambda}^{-t}(W).
\]

Given a compactly supported time-dependent Hamiltonian $H\colon S^1\times W^{2n}\ra\bR$, let $\varphi_H^t$ denote its Hamiltonian flow, obtained by  integrating  the Hamiltonian vector field $X_{H_t}$  determined by  $\iota_{X_{H_t}}d\lambda=d_wH_t$. For $\psi\coloneqq\varphi_H^1$, the {\rm Calabi invariant} is defined by
\[
\Cal(\psi)\coloneqq\int_{S^1\times W}Hdt\wedge (d\lambda)^n
\]
By Lemma~10.3.4 in \cite{MS17}, the above integral is independent of the choice of compactly supported Hamiltonian $H$ generating $\psi$. A curve $\gamma\colon[0,1]\ra W$ is called a $1$-periodic orbit of $H$ if $\dot{\gamma}(t)=X_{H_t}(\gamma(t))$ and $\gamma(0)=\gamma(1)$. Its action  is given by
\[
\cA_{H}(\gamma)\coloneqq\int_{\gamma}\lambda+\int_0^1H(t,\gamma(t))dt.
\]

 The restriction  $\alpha\coloneqq\lambda|_{\partial W}$ defines a contact form on $\partial W$. Recall that the Reeb vector field $R_\alpha$ of the contact manifold $(\partial W, \ker(\alpha))$ is uniquely determined by $d\alpha(R_{\alpha},\cdot)=0$ and $\alpha(R_\alpha)=1$. A  {\em closed Reeb orbit} in $(\partial W, \ker(\alpha))$ is a curve $\gamma\colon [0,T]\ra \partial W$ satisfying $\dot{\gamma}(t)=R_{\alpha}(\gamma(t))$ and $\gamma(0) = \gamma(T)$. 

\medskip

Given a Liouville domain $(W,\lambda)$, recall from (\ref{dfn-H-lambda}) that $\cH(W,\lambda)$ denotes the set of all Hamiltonian functions $H\colon S^1\times W\ra\bR$ satisfying condition (\ref{H-con}). Also, recall that for any $H\in\cH(W,\lambda)$ we associate a domain $W_H \subset W \times \bC$ constructed in (\ref{Usher-con}), equipped with $1$-form $\hat\lambda(= \lambda+\pi r^2dt) \in\Omega^1(W\times\bC)$. The following propositions are inspired by Proposition~2.1 in \cite{ABE25} and Proposition~4.1 in \cite{Ush22}.
\begin{prop}\label{prop-Louville}
For any $H\in\cH(W^{2n},\lambda)$, the domain $(W_H,\hat{\lambda})$ is a Liouville domain with the volume ${\rm Vol}(W_H)=\frac{1}{n!}\,\Cal(H).$	
\end{prop} 
\begin{proof}
It is immediate that $d\hat{\lambda}=d\lambda+dx\wedge dy$ is symplectic on $W\times \bC$, hence also on $W_H$. Denote by $Y_{\lambda}$ the Liouville vector field of $(W,\lambda)$, so that $\iota_{Y_\lambda}d\lambda=\lambda$. Then the Liouville vector field of $(W\times \bC,\hat{\lambda})$ is $\widehat Y=Y_\lambda+\frac12(x\partial_x+y\partial_y),$ since $\iota_{\widehat Y}d\hat\lambda=\hat\lambda$.

On $W\times (\bC\setminus\{0\})$, the angular coordinate $t=\theta/2\pi\in S^1$ is well-defined and smooth, so we may define a smooth function $f_H(w,z)=\pi |z|^2-H(t,w)$. Then 
\[ W_H\cap (W\times (\bC\setminus\{0\}))=f_H^{-1}((-\infty,0]).\]
We compute $df_H(\widehat{Y})$. Since $x\partial_x+y\partial_y=r\partial_r$ and $dt(r\partial_r)=0$, we have
\begin{align*}
df_H(\widehat Y)
&=d(\pi |z|^2)(\widehat Y)-dH(\widehat Y) \\
&=\pi |z|^2-dH_t(Y_\lambda)
-\partial_tH\cdot dt\!\left(\frac12(x\partial_x+y\partial_y)\right) \\
&=\pi |z|^2-dH_t(Y_\lambda).
\end{align*}
Using $\iota_{X_{H_t}}d\lambda=d_wH_t$ and
$\iota_{Y_\lambda}d\lambda=\lambda$, we obtain $d_wH_t(Y_\lambda)=d\lambda(X_{H_t},Y_\lambda)=-\lambda(X_{H_t}).$ Hence, we have 
\[ df_H(\widehat Y)=\pi |z|^2+\lambda(X_{H_t})=f_H+H+\lambda(X_{H_t})=f_H+\tau_{\lambda,H}.\] Therefore, on $f_H^{-1}(0)$ we have $df_H(\widehat Y)=\tau_{\lambda,H}>0$. It follows that $0$ is a regular value of $f_H$, and  $\widehat Y$  is outward transverse to the boundary along $f_H^{-1}(0)\subset W\times (\bC\setminus\{0\})$.

  Since $H$ is autonomous in a neighborhood of $\partial W$, there exists an open neighborhood $U\subset W$ of $\partial W$ and a smooth function $h\colon U\to \bR$ such that $H(t,w)=h(w)$ for all $(t,w)\in S^1\times U$. On $U\times \bC$, define $g_H(w,z)=\pi |z|^2-h(w)$.
Then $g_H$ is smooth on all of $U\times \bC$, and by definition, $W_H\cap (U\times \bC)=g_H^{-1}((-\infty,0])$. We again compute
\begin{align*}
dg_H(\widehat Y)
&=d(\pi |z|^2)(\widehat Y)-dh(Y_\lambda) \\
&=\pi |z|^2-dH_t(Y_\lambda) \\
&=\pi |z|^2+\lambda(X_{H_t}) \\
&=g_H+h+\lambda(X_{H_t}) =g_H+\tau_{\lambda,H}.
\end{align*}
Thus, on $g_H^{-1}(0)$ we have $dg_H(\widehat Y)=\tau_{\lambda,H}>0$. Hence $0$ is a regular value of $g_H$, and $\widehat Y$ is outward transverse to the boundary along $g_H^{-1}(0)$, which contains $\partial W\times\{0\}$.

Since every boundary point of $W_H$ is covered by one of the two cases above, it follows that $W_H$ is a smooth manifold with boundary $\partial W_H$, and that $\widehat Y$ points outward everywhere along $\partial W_H$. Finally, $W_H$ is compact because $W$ is compact and $\pi |z|^2\le H(t,w)\le \max_{S^1\times W}H$. Therefore $(W_H,\hat{\lambda})$ is a Liouville domain.

The volume of $W_H$ is
		\begin{align*}
		{\rm Vol}(W_H)&=\frac{1}{(n+1)!}\int_{W_H}(d\hat{\lambda})^{n+1}\\
		& =\frac{n+1}{(n+1)!}\int_{W_H}d\left(\frac{1}{2}r^2d\theta\right)\wedge(d\lambda)^n\\
		&=\frac{1}{n!}\int_{\partial W_H}\frac{1}{2}r^2d\theta\wedge(d\lambda)^n\\
		& =\frac{1}{n!}\int_{[0,2\pi]\times W}\frac{H(\theta/2\pi,w)}{2\pi}d\theta\wedge(d\lambda)^n =\frac{1}{n!}\int_{[0,1]\times W}Hdt\wedge(d\lambda)^n.
	\end{align*}
Here we use $\partial W_H=\{(w,re^{i\theta})\in W\times \bC\mid \pi r^2=H(\theta/2\pi,w)\}$.
\end{proof}

An immediate consequence is:
\begin{cor}\label{cor-calabi}
	For any $H,G\in\cH(W,\lambda)$, we have
	\[
	\frac{1}{n+1}|\ln\Cal(H)-\ln\Cal(G)|\leq d_{\rm BM}((W_H,\hat{\lambda}),(W_G,\hat{\lambda}))\leq d_{\rm SBM}((W_H,\hat{\lambda}),(W_G,\hat{\lambda})).
	\]
	\end{cor}
\begin{proof}
Let $(W,\lambda_W)$ and $(V,\lambda_V)$ be two $2n$-dimensional Liouville domains. For any $\varepsilon>0$, there exists $C$ with $\ln C<d_{\rm BM}((W,\lambda_W),(V,\lambda_V))+\varepsilon$ and Liouville embeddings $i_{W,V} \colon \tfrac{1}{C}W \xL V$ and $i_{V,W} \colon \tfrac{1}{C}V \xL W$.  Recall that for a Liouville domain $(W,\lambda_W)$ with Liouville flow $\sL_{\lambda_W}^t$, the scaled domain $aW$ for $0<a\leq 1$ is defined by $aW=\sL_{\lambda_W}^{\ln a}(W)$, and its volume satisfies $\Vol(aW)=a^n\Vol(W)$. In particular, for $a=1/C$, we have 
\[ \Vol(\tfrac{1}{C} W) = (\tfrac{1}{C})^n \Vol(W).\]

Since $i_{W,V}$ is a Liouville embedding,  the image $i_{W,V}(\tfrac{1}{C}W)$ has the same volume as $\tfrac{1}{C}W$, and it is contained in $V$. Therefore, $\left(\frac{1}{C}\right)^n\Vol(W)=\Vol\left(\frac{1}{C}W\right) \leq \Vol(V)$. Similarly, from the embedding $i_{V,W}$ we obtain $\left(\frac{1}{C}\right)^n\Vol(V)=\Vol\left(\frac{1}{C}V\right)\leq \Vol(W)$. Combining these two inequalities yields 
\[ \frac{1}{C^n}\Vol(V) \leq \Vol(W) \leq C^n\Vol(V).\]
 Taking logarithms we obtain  $\frac{1}{n}\bigl|\ln\Vol(W) - \ln\Vol(V)\bigr| \leq \ln C$. Letting $\varepsilon \to 0$, we conclude that $d_{\mathrm{BM}}((W,\lambda_W),(V,\lambda_V)) \geq \frac{1}{n} \bigl|\ln\Vol(W) - \ln\Vol(V)\bigr|$. Taking $(W,\lambda_W)=(W_H,\hat{\lambda})$ and $(V,\lambda_V)=(W_G,\hat{\lambda})$, then we obtain the desired conclusion with constant being $\frac{1}{n+1}$.
\end{proof}

\begin{remark}[Section~4 in \cite{Ush22}]\label{rmk-positive}
The  condition $\tau_{\lambda,H}>0$ implies that $H > 0$ on ${\rm Int}\,W$. Otherwise, $H$ would attain a non-positive global minimum in ${\rm Int}\,W$. Then at such a point for the corresponding fixed time $t$, one has
$X_{H_t}=0$, hence $\tau_{\lambda,H}(t,w)=H(t,w)<0$, which contradicts the positivity of $\tau_{\lambda,H}$. This condition also implies that $\lambda(X_{H_t})|_{\partial W}>0$ and that $d\lambda(X_{H_t},\cdot)|_{T\partial W}\equiv 0$. Hence $X_{H_t}|_{\partial W}$ is a positive, possibly nonconstant, multiple of the Reeb field $R_{\lambda|_{\partial W}}$ of $\lambda|_{\partial W}$.
\end{remark}

The following proposition provides a sufficient condition where two tubes, constructed from different input Hamiltonian functions, can be related by a Liouville diffeomorphism. 

\begin{prop}\label{prop-invariant}
Let $\{H^{s}\}_{s\in[0,1]}\subset \cH(W,\lambda)$ be a smooth family. Assume that the time-one map $\varphi_{H^{s}}^{1}$ is independent of $s$, and that there exists a neighborhood $U$ of $\partial W$ such that $H^{s}=H^{0} \text{ on } S^{1}\times U$, for all $s\in[0,1]$. Then there exists a Hamiltonian diffeomorphism $\psi\colon W\times \bC\to W\times \bC$ such that $\psi(W_{H^{0}})=W_{H^{1}}$. In particular, the restriction $\psi|_{W_{H^{0}}}\colon (W_{H^{0}},\hat\lambda)\to (W_{H^{1}},\hat\lambda)$ is a Liouville diffeomorphism.
\end{prop}
\begin{proof}
Write $\phi_s^t\coloneqq \varphi_{H^s}^t$. After shrinking $U$, we may assume that $U$ is invariant under the common Hamiltonian flow near $\partial W$. Indeed, on $U$ the Hamiltonians $H^s$ are all equal to $H^0$, which is autonomous near $\partial W$, and the corresponding Hamiltonian vector field preserves the level sets of $H^0$. Hence, for all $s\in[0,1]$, the flows $\phi_s^t$ agree on $U$.

Consider the $s$-variation of the Hamiltonian isotopy $Y_s^t\coloneqq \partial_s\phi_s^t\circ(\phi_s^t)^{-1}$. Since the endpoints $\phi_s^0=\mathrm{id}$ and $\phi_s^1$ are independent of $s$, we have $Y_s^0=Y_s^1=0$. Moreover, by the argument above, $Y_s^t=0$ on $U$. By Banyaga’s Lemma (Proposition~3.1.5 in \cite{Ban97}), there exists a smooth family of Hamiltonians $F_s\colon S^1\times W\to\mathbb R$ with support contained in $W\setminus U$, such that $Y_s^t=X_{F_s(t,\cdot)}$ and
\begin{equation}\label{eq-variation-formula}
\partial_s H^s=\partial_t F_s-dH_t^s(X_{F_s}).
\end{equation}
Here $X_{F_s}$ denotes the Hamiltonian vector field of $F_s(t,\cdot)$ with respect to $d\lambda$. The normalization $F_s=0$ on $U$ also makes $F_s$ a smooth function on $S^1\times W$, since $Y_s^0=Y_s^1=0$.

 Let $\widetilde\Omega\coloneqq\mathbb R\times S^1\times W$ with coordinates $(\rho,t,w)$ and symplectic form $\omega_{\widetilde\Omega}\coloneqq d\lambda+d\rho\wedge dt$. For each $s$, define $f_s(\rho,t,w)\coloneqq\rho-H^s(t,w)$. Let $K_s(\rho,t,w)\coloneqq F_s(t,w)$. Since $K_s$ is independent of $\rho$, its Hamiltonian vector field with respect to $\omega_{\widetilde\Omega}$ is $X_{K_s}=X_{F_s}+(\partial_tF_s)\partial_\rho$. Therefore $X_{K_s}(f_s)=\partial_tF_s-dH_t^s(X_{F_s})$. Using \eqref{eq-variation-formula}, we obtain $(\partial_s+X_{K_s})f_s=0$. Thus, if $\widetilde\Psi_s$ denotes the flow of $X_{K_s}$, then $\frac{d}{ds}(f_s\circ \widetilde\Psi_s)=0$. Consequently, $f_s\circ\widetilde\Psi_s=f_0$, and $\widetilde\Psi_s$ sends the hypersurface $\Gamma_0\coloneqq\{f_0=0\}$ to $\Gamma_s\coloneqq\{f_s=0\}$.

It remains to descend this construction to $W\times\mathbb C$. Put $\rho=\pi r^2$ and $z=re^{2\pi it}$. The map $\Phi\colon (0,\infty)\times S^1\times W\to W\times\mathbb C^*$ defined by \[ \Phi(\rho,t,w)=\left(w,\sqrt{\rho/\pi}\,e^{2\pi it}\right)\]
satisfies $\Phi^*\hat\lambda=\lambda+\rho\,dt$ and $\Phi^*d\hat\lambda=d\lambda+d\rho\wedge dt=\omega_{\widetilde\Omega}$. Since $F_s$ is supported away from $U$, and $H^s>0$ on ${\rm Int}W$, there exists $\delta>0$ such that $H^s(t,w)\geq 2\delta$ on the support of $F_s$, for all $s\in[0,1]$. Choose a smooth cutoff function $\chi_s(\rho,t,w)$ such that $\chi_s=1$ in a neighborhood of $\Gamma_s\cap{\rm supp}F_s$, while $\chi_s=0$ for $\rho\leq \delta$ and also near $\partial W$. Replacing $K_s$ by $\chi_sK_s$ does not change the above flow near the hypersurfaces $\Gamma_s$, but
makes the Hamiltonian vanish for $\rho$ sufficiently small and near
$\partial W$.

Hence, the function
\[
\overline K_s(w,z)\coloneqq
\begin{cases}
(\chi_sK_s)(\pi|z|^2,\frac{\arg z}{2\pi},w),& z\neq 0,\\
0, & z=0,
\end{cases}
\]
defines a smooth compactly supported Hamiltonian on $W\times\mathbb C$. Let $\psi_s$ be its Hamiltonian flow with respect to $d\hat\lambda$. By construction, $\psi_s$ agrees, near the smooth graph part of $\partial W_{H^0}$, with the flow obtained from $\widetilde\Psi_s$ under the map $\Phi$. Therefore, $\psi_1$ sends the graph part of $\partial W_{H^0}$ to the graph part of $\partial W_{H^1}$. Near the binding $\partial W\times\{0\}$, the
Hamiltonian $\overline K_s$ vanishes, and the two domains agree because $H^s=H^0$ near $\partial W$. Therefore $\psi_1$ fixes the binding and sends $\partial W_{H^0}$ onto $\partial W_{H^1}$. Moreover, the identity $(\partial_s+X_{K_s})f_s=0$ shows that the inward side of $\partial W_{H^0}$ is mapped to the inward side of $\partial W_{H^1}$. Hence $\psi_1(W_{H^0})=W_{H^1}$.

Finally, since $\psi_s$ is a Hamiltonian isotopy of the exact
symplectic manifold $(W\times\mathbb C,d\hat\lambda)$, we have
\[
\frac{d}{ds}\psi_s^*\hat\lambda=\psi_s^*\mathcal L_{X_{\overline K_s}}\hat\lambda=d\left(\psi_s^*(\overline K_s+\hat\lambda(X_{\overline K_s}))\right).
\]
It follows that $\psi_1^*\hat\lambda-\hat\lambda$ is exact. Therefore the restriction $\psi_1|_{W_{H^0}}\colon (W_{H^0},\hat\lambda)\to (W_{H^1},\hat\lambda)$
is a Liouville diffeomorphism.
\end{proof}

This subsection ends with a cute lemma which will be helpful in the proof of Theorem \ref{thm-estimation}. 

\begin{lemma}\label{lemma-reparametrization}
	If $\psi\colon (W,\lambda_W)\ra (V,\lambda_V)$ is a strict Liouville diffeomorphism, then it induces a canonical bijection between $\cH(W,\lambda_W)$ and $\cH(V,\lambda_V)$ via $H\mapsto H\circ\psi^{-1}$ and
	 a strict Liouville diffeomorphism
	\[
	\psi_H\colon (W_H,\hat{\lambda}_W)\ra (V_{H\circ \psi^{-1}},\hat{\lambda}_V) \quad \mbox{by} \quad (w,z)\mapsto (\psi(w),z).
	\]
	In particular, $(W_H,\hat{\lambda})\cong(W_{H\circ\psi},\hat{\lambda})$ for any $\psi\in{\rm Symp}(W,\lambda)$.
\end{lemma}
\begin{proof}
Let $K \coloneqq H \circ \psi^{-1} \colon S^1 \times V \to \mathbb{R}$.
We first show that $K \in \mathcal{H}(V,\lambda_V)$. Since $\psi \colon W \to V$ is a diffeomorphism of manifolds with boundary, it maps $\partial W$ onto $\partial V$. Hence, because $H$ is autonomous in a neighborhood of $\partial W$ and satisfies $H|_{\partial W} \equiv 0$, the same is true for $K$ in a neighborhood of $\partial V$. Thus it remains to verify that $\tau_{\lambda_V,K}(t,v)=\lambda_V(X_{K_t}(v))+K(t,v)>0.$

To do so, we first identify the Hamiltonian vector field of $K$. Since $\psi$ is a strict Liouville diffeomorphism, we have $\psi^*\lambda_V=\lambda_W$, hence $\psi^*d\lambda_V=d\lambda_W$. Therefore $\psi$ is a symplectomorphism. For each fixed $t\in S^1$, we have $K_t=H_t\circ\psi^{-1}$. Since $\psi$ is a strict Liouville diffeomorphism, it is symplectic. Hence, 
\[ \iota_{\psi_*X_{H_t}}d\lambda_V=(\psi^{-1})^*(\iota_{X_{H_t}}d\lambda_W)=(\psi^{-1})^*d_wH_t=d_vK_t.\]
Therefore $X_{K_t}=\psi_*X_{H_t}$. Then, using $\psi^*\lambda_V=\lambda_W$, we obtain $\psi^*\tau_{\lambda_V,K}(t,\cdot)=\psi^*(\lambda_V(X_{K_t})+K_t)=\lambda_W(X_{H_t})+H_t=\tau_{\lambda_W,H}(t,\cdot).$ Since $\tau_{\lambda_W,H}>0$ by assumption, we conclude that $\tau_{\lambda_V,K}>0$. Thus $K=H\circ\psi^{-1}\in \mathcal{H}(V,\lambda_V)$.

Next, define $\psi_H \colon W_H \to V_{H\circ\psi^{-1}}$ by $(w,z)\mapsto (\psi(w),z)$.
We check that this map is well defined. If $(w,z)\in W_H$, writing $z=re^{i\theta}$ with $\theta=2\pi t$, we have $\pi |z|^2 \le H(t,w)=(H\circ\psi^{-1})(t,\psi(w))$, so $(\psi(w),z)\in V_{H\circ\psi^{-1}}$. Since the inverse map is simply $(v,z)\mapsto (\psi^{-1}(v),z)$,
the map $\psi_H$ is a diffeomorphism.

Finally, we compute the pullback of the Liouville form, $\psi_H^*\hat{\lambda}_V=\psi_H^*\left(\lambda_V+\frac12 r^2\,d\theta\right)=\psi^*\lambda_V+\frac12 r^2\,d\theta =\lambda_W+\frac12 r^2\,d\theta =\hat{\lambda}_W$.
Therefore $\psi_H$ is a strict Liouville diffeomorphism
$\psi_H\colon (W_H,\hat{\lambda}_W)\to (V_{H\circ\psi^{-1}},\hat{\lambda}_V)$.

This proves that the assignment $H\mapsto H\circ\psi^{-1}$ gives a bijection from $\mathcal{H}(W,\lambda_W)$ to $\mathcal{H}(V,\lambda_V)$, with inverse $K\mapsto K\circ\psi$.
\end{proof}
 
 \subsection{Reeb dynamics on the lift}\label{sec-Reeb-Ham}
We now study the dynamics on $\partial W_H$. When $H$ is autonomous, the following results recover those from Section~4 in \cite{Ush22}.  We begin by computing the Reeb vector field.
\begin{prop}\label{prop-construction}
Let $H\in\cH(W,\lambda)$ and set $\alpha_H\coloneqq\hat{\lambda}|_{\partial W_H}$. Then, on $\partial W_H\cap\{r>0\}$, the Reeb vector field of $\alpha_H$ is given by
\[
R_{\alpha_H}=\frac{1}{\tau_{\lambda,H}}\left(X_{H_t}+\frac{\partial_tH}{2\pi r}\partial_r+2\pi\partial_\theta\right).
\]
Equivalently, since $\theta=2\pi t$, one may write
\[
R_{\alpha_H}=\frac{1}{\tau_{\lambda,H}}\left(X_{H_t}+\frac{\partial_tH}{2\pi r}\partial_r+\partial_t\right).
\]
Moreover, this vector field extends smoothly to the entire $\partial W_H$.
\end{prop}
\begin{proof}
Set $F(w,r,t)\coloneqq \pi r^2-H(t,w),$ so that $\partial W_H=\{F=0\}$. Since $\theta=2\pi t$, we have $\hat{\lambda}=\lambda+\pi r^2dt$ and $d\hat{\lambda}=d\lambda+2\pi r\,dr\wedge dt$. On $\partial W_H\cap\{r>0\}$, write the Reeb vector field in the form $R=aX_{H_t}+b\partial_r+c\partial_t$, where $a,b,c$ are functions to be determined.

The condition $\iota_R d\alpha_H=0$ is equivalent to saying that, along $\partial W_H$, the $1$-form $\iota_R d\hat{\lambda}$ is proportional to $dF$. Thus there exists a function $g$ on $\partial W_H$ such that $\iota_Rd\hat{\lambda}=g\,dF$. Using $\iota_{X_{H_t}}d\lambda=d_wH_t$, we compute $\iota_Rd\hat{\lambda}=a\,d_wH_t+2\pi r\,b\,dt-2\pi r\,c\,dr$. On the other hand, $dF=2\pi r\,dr-d_wH_t-(\partial_tH)\,dt$. Comparing the coefficients of $d_wH_t, dt$ and $dr$, we obtain  
\[ a=-g,\,\, b=-\frac{g\,\partial_tH}{2\pi r},\,\,\mbox{and}\,\, c=-g.\]
It remains to impose the normalization condition $\alpha_H(R)=1$. Since $\pi r^2=H(t,w)$ on $\partial W_H$, we have 
\[ 1=\hat{\lambda}(R)=a\,\lambda(X_{H_t})+\pi r^2c=-g(\lambda(X_{H_t})+H(t,w))=-g\,\tau_{\lambda,H}.\]
Hence, $g=-\frac{1}{\tau_{\lambda,H}}$, and therefore $a=\frac{1}{\tau_{\lambda,H}}, b=\frac{\partial_tH}{2\pi r\,\tau_{\lambda,H}}, c=\frac{1}{\tau_{\lambda,H}}$. Substituting these coefficients into the expression for $R$ gives
\[
R_{\alpha_H}=\frac{1}{\tau_{\lambda,H}}\left(X_{H_t}+\frac{\partial_tH}{2\pi r}\partial_r+\partial_t\right)
=\frac{1}{\tau_{\lambda,H}}\left(X_{H_t}+\frac{\partial_tH}{2\pi r}\partial_r+2\pi\partial_\theta\right).
\]

Finally, by the defining assumptions on $H$, the Hamiltonian is autonomous in a neighborhood of $\partial W$. Hence $\partial_tH=0$ near the locus $r=0\subset\partial W_H$. Therefore $\frac{\partial_tH}{2\pi r}\partial_r$ vanishes near $r=0$. The remaining term $\partial_t=2\pi\partial_\theta$ is the angular vector field in the $\mathbb C$-factor and extends smoothly across $r=0$, where it vanishes. Thus the above formula defines a smooth vector field on all of $\partial W_H$.
\end{proof}
 
\begin{remark}\label{rmk-rho}
Equivalently, one may use the coordinate $\rho=\frac{1}{2}r^2$ instead of $r$. In this coordinate, we have $\hat{\lambda}=\lambda+\rho d\theta$. The formula for the Reeb vector field of $\alpha_H=\hat{\lambda}|_{\partial W_H}$ on $\partial W_H$ becomes 
\[
R_{\alpha_H}= \frac{1}{\tau_{\lambda,H}}
\left(
X_{H_t}+\frac{1}{2\pi}\partial_tH\,\partial_\rho
+2\pi\partial_\theta
\right).
\]
\end{remark}

Next, we give an explicit formula for the Reeb flow on $\partial W_H$.

\begin{cor}\label{cor-formula}
In the setting of Proposition~\ref{prop-construction}, let $\{\psi^s:\partial W_H\to\partial W_H\}_{s\in\bR}$ denote the flow of the Reeb vector field $R_{\alpha_H}$. For $(t,w)\in S^1\times W$, let $\sigma\mapsto\varphi_H^\sigma(t,w)$ be the non-autonomous Hamiltonian trajectory starting from $w$ at time $t$, namely
\begin{equation}\label{eq-Ham-vector}
\frac{d}{d\sigma}\varphi_H^\sigma(t,w)=X_{H_{t+\sigma}}(\varphi_H^\sigma(t,w)),
\qquad
\varphi_H^0(t,w)=w
\end{equation}
where $t+\sigma$ is understood modulo $1$. Since $\tau_{\lambda,H}>0$, for each $(t,w)\in S^1\times W$, the map $u\mapsto \int_0^u \tau_{\lambda,H}(t+\sigma,\varphi_H^\sigma(t,w))\,d\sigma$ is a strictly increasing diffeomorphism $\bR\to\bR$. Let $h\colon\bR\times S^1\times W\to\bR$ denote its inverse, so that
\begin{equation}\label{eq-h}
\int_0^{h(s,t,w)} \tau_{\lambda,H}(t+\sigma,\varphi_H^\sigma(t,w))\,d\sigma=s.
\end{equation}
Then for $(w,z)\in\partial W_H$, written as $z=re^{2\pi it}$ and $\pi r^2=H(t,w)$, one has
\[
\psi^s(w,z)=\left(\varphi_H^{h(s,t,w)}(t,w),\ \sqrt{\frac{H(t+h(s,t,w),\varphi_H^{h(s,t,w)}(t,w))}{\pi}}\,e^{2\pi i(t+h(s,t,w))}\right).
\]
Moreover, this formula extends smoothly to the binding $\partial W\times\{0\}$.
\end{cor}
\begin{proof}
Fix $(w,z)\in \partial W_H$, and write $z=re^{2\pi i t}$ with $\pi r^2=H(t,w)$. Define 
\[
\eta(s)\coloneqq h(s,t,w), \,\, w(s)\coloneqq \varphi_H^{\eta(s)}(t,w),\,\, t(s)\coloneqq t+\eta(s), \text{ and } r(s)\coloneqq \sqrt{H(t(s),w(s))/\pi}.
\] Then $\Psi^s(w,z)\coloneqq (w(s),r(s)e^{2\pi i t(s)})$ lies in $\partial W_H$ for all $s$, since $\pi r(s)^2=H(t(s),w(s))$. We claim that $\Psi^s=\psi^s$. Indeed, by the definition of $h$, we have $h(0,t,w)=0$. Differentiating \eqref{eq-h} with respect to $s$ gives $\dot{\eta}(s) = \tau_{\lambda,H}(t(s),w(s))^{-1}$. Therefore, $\dot{t}(s) = \dot{\eta}(s) =\tau_{\lambda,H}(t(s),w(s))^{-1}.$
Moreover, since $\sigma\mapsto \varphi_H^\sigma(t,w)$ is generated by $X_{H_{t+\sigma}}$, we have 
\[
\dot{w}(s)=\dot{\eta}(s)X_{H_{t(s)}}(w(s))=\frac{ X_{H_{t(s)}}(w(s)) }{ \tau_{\lambda,H}(t(s),w(s)) }.
\]

Next, differentiating the identity $\pi r(s)^2=H(t(s),w(s))$ yields 
\[
2\pi r(s)\dot{r}(s)=d_wH(t(s),w(s))(\dot{w}(s))+\partial_tH(t(s),w(s))\dot{t}(s).
\]
 Substituting the formulas for $\dot{w}(s)$ and $\dot{t}(s)$, we get
\[
2\pi r(s)\dot{r}(s)=\frac{ d_wH(t(s),w(s))(X_{H_{t(s)}}(w(s))) + \partial_tH(t(s),w(s)) }{ \tau_{\lambda,H}(t(s),w(s)) }.
\]
Since $\iota_{X_{H_t}}d\lambda=d_wH_t$, we have $d_wH_t(X_{H_t})=d\lambda(X_{H_t},X_{H_t})=0$. Hence, \[ \dot{r}(s) = \frac{ \partial_tH(t(s),w(s)) }{ 2\pi r(s)\tau_{\lambda,H}(t(s),w(s)) }. \]

Combining these identities, we find that $\Psi^s$ satisfies 
\[
\frac{d}{ds}\Psi^s=\frac{1}{\tau_{\lambda,H}}\left(X_{H_t}+\frac{\partial_tH}{2\pi r}\partial_r+\partial_t\right)
\]
which is exactly the Reeb equation from Proposition~\ref{prop-construction}. Since $\Psi^0(w,z)=(w,z)$, uniqueness of solutions implies $\Psi^s(w,z)=\psi^s(w,z)$. This proves the stated formula on $\partial W_H\cap\{r>0\}$.

Finally, because $H$ is autonomous in a neighborhood of $\partial W$, the expression above is independent of the choice of angular coordinate near the binding and extends smoothly across $r=0$. Equivalently, on $\partial W\times\{0\}$, the extended Reeb flow agrees with the Reeb flow of $\lambda|_{\partial W}$. This completes the proof.
\end{proof}

\begin{remark}\label{rmk-bdy}
Near $\partial W$, the Hamiltonian is autonomous, so $X_{H_t}$ is independent of $t$ there. Thus on the binding, $R_{\alpha_H}|_{\partial W\times\{0\}}=R_{\lambda}|_{\partial W}$. Hence, the orbits of $R_{\alpha_H}$ are each either contained in or disjoint from $\partial W\times\{0\}$. 
\end{remark}

There exists a natural correspondence between Reeb orbits in $(\partial W_H\setminus (\partial W\times\{0\}),\alpha_H)$ and integer-periodic Hamiltonian orbits of $H$ on $(W,d\lambda)$. For a time-dependent Hamiltonian $H$, we write $\varphi_H^\sigma(t_0,w_0)\coloneqq\varphi_H^{t_0+\sigma,t_0}(w_0)$, where $\varphi_H^{t_1,t_0}$ denotes the Hamiltonian flow from time $t_0$ to time $t_1$.
\begin{cor}\label{cor-Reeb-Ham}
The periodic Reeb orbits of $(\partial W_H,\alpha_H)$, up to time translation, are of the following two types.

\begin{enumerate}
\item
For each periodic Reeb orbit $c:\bR/T\bZ\to \partial W$ of $\lambda|_{\partial W}$, the curve $\gamma_c:\bR/T\bZ\to \partial W_H$, $\gamma_c(s)\coloneqq (c(s),0)$, is a periodic Reeb orbit of $\alpha_H$.

\item
Let $(t_0,w_0,k)\in S^1\times {\rm Int}(W)\times \bN$ satisfy $H(t_0,w_0)>0$ and $\varphi_H^k(t_0,w_0)=w_0$. Set
$z_0\coloneqq \sqrt{H(t_0,w_0)/\pi}\,e^{2\pi i t_0}$ and
$T_{t_0,w_0,k}\coloneqq \int_0^k \tau_{\lambda,H}(t_0+\sigma,\varphi_H^\sigma(t_0,w_0))\,d\sigma$.
Then $(w_0,z_0)\in \partial W_H$ and $h(T_{t_0,w_0,k},t_0,w_0)=k$. Thus $\gamma_{t_0,w_0,k}(s)\coloneqq \psi^s(w_0,z_0)$ is a periodic Reeb orbit of period $T_{t_0,w_0,k}$. More explicitly,
\[
\gamma_{t_0,w_0,k}(s)=\left(\varphi_H^{h(s,t_0,w_0)}(t_0,w_0),\sqrt{\frac{H(t_0+h(s,t_0,w_0),\varphi_H^{h(s,t_0,w_0)}(t_0,w_0))}{\pi}}\,e^{2\pi i(t_0+h(s,t_0,w_0))}\right).
\]
Moreover,  $\Gamma(s)\coloneqq \varphi_H^{ks}(t_0,w_0)$ is a $1$-periodic orbit of the Hamiltonian $H^{(k)}_{t_0}(t,w)\coloneqq kH(t_0+kt,w)$, and its action satisfies $T_{t_0,w_0,k}=\cA_{H^{(k)}_{t_0}}(\Gamma).$
\end{enumerate}

Conversely, every periodic Reeb orbit of $(\partial W_H,\alpha_H)$ is of one of the two types above.
\end{cor}

\begin{proof}
Since $\alpha_H|_{\partial W\times\{0\}}=\lambda|_{\partial W}$, and  $R_{\alpha_H}$ is tangent to $\partial W\times\{0\}$, the restriction of $R_{\alpha_H}$ to $\partial W\times\{0\}$ coincides with the Reeb vector field of $\lambda|_{\partial W}$. Therefore, the closed Reeb orbits contained in $\partial W\times\{0\}$ are precisely the closed Reeb orbits of $(\partial W,\lambda|_{\partial W})$. This proves (1).

Now let $\widetilde\gamma(s)=\psi^s(w_0,z_0)$ be a periodic Reeb orbit not contained in $\partial W\times\{0\}$. By Remark~\ref{rmk-bdy}, it is disjoint from $\partial W\times\{0\}$. After a time translation, we may write $z_0=r_0e^{2\pi i t_0}\neq 0$, where $r_0=\sqrt{H(t_0,w_0)/\pi}$. Let $T>0$ be a period of $\widetilde{\gamma}$. By Corollary~\ref{cor-formula},
\[
\widetilde\gamma(s)=\left(\varphi_H^{h(s,t_0,w_0)}(t_0,w_0),\sqrt{\frac{H(t_0+h(s,t_0,w_0),\varphi_H^{h(s,t_0,w_0)}(t_0,w_0))}{\pi}}\,e^{2\pi i(t_0+h(s,t_0,w_0))}\right).
\]
Since $\widetilde\gamma(T)=\widetilde\gamma(0)$ and $z_0\neq 0$, the angular coordinate satisfies $t(T)=t_0+h(T,t_0,w_0)=t_0+k$ for some $k\in \bZ$. Moreover, \[
\partial_s h(s,t_0,w_0)=\frac{1}{\tau_{\lambda,H}(t_0+h(s,t_0,w_0),\varphi_H^{h(s,t_0,w_0)}(t_0,w_0))}>0
\]
 so $h(\,\cdot\,,t_0,w_0)$ is strictly increasing. Since $T>0$ and $h(0,t_0,w_0)=0$, we have  $k=h(T,t_0,w_0)\in\bN$. The $W$-component of the equality $\widetilde{\gamma}(T)=\widetilde{\gamma}(0)$ gives $\varphi_H^k(t_0,w_0)=w_0$. Finally, the defining equation of $h$ yields
\[
T=\int_0^k \tau_{\lambda,H}(t_0+\sigma,\varphi_H^\sigma(t_0,w_0))\,d\sigma.
\]
Thus every periodic Reeb orbit disjoint from the binding determines a triple $(t_0,w_0,k) \in S^1\times \operatorname{Int}(W)\times\bN$ as in (2).

Conversely, suppose that $(t_0,w_0,k)\in S^1\times {\rm Int}(W)\times \bN$ satisfies $H(t_0,w_0)>0$ and $\varphi_H^k(t_0,w_0)=w_0$. Define $z_0\coloneqq \sqrt{H(t_0,w_0)/\pi}\,e^{2\pi i t_0}$ and let $T_{t_0,w_0,k}$ be as above. Then the defining property of $h$ gives $h(T_{t_0,w_0,k},t_0,w_0)=k$. Applying Corollary~\ref{cor-formula} at time $T_{t_0,w_0,k}$, we obtain
\[
\psi^{T_{t_0,w_0,k}}(w_0,z_0)=\left(\varphi_H^k(t_0,w_0),\sqrt{\frac{H(t_0+k,\varphi_H^k(t_0,w_0))}{\pi}}\,e^{2\pi i(t_0+k)}\right).
\]
Since $H$ is $1$-periodic in time and $\varphi_H^k(t_0,w_0)=w_0$, this equals $(w_0,z_0).$
Hence $\gamma_{t_0,w_0,k}(s)\coloneqq \psi^s(w_0,z_0)$ is a periodic Reeb orbit of period $T_{t_0,w_0,k}$.

It remains to verify the action identity. Define $\Gamma(s)\coloneqq \varphi_H^{ks}(t_0,w_0)$. Since $\varphi_H^k(t_0,w_0)=w_0$, the curve $\Gamma$ is $1$-periodic. Moreover, $\dot\Gamma(s)=kX_{H_{t_0+ks}}(\Gamma(s))=X_{H^{(k)}_{t_0,s}}(\Gamma(s))$, where $H^{(k)}_{t_0,s}(w)=kH(t_0+ks,w)$. Thus $\Gamma$ is a $1$-periodic Hamiltonian orbit of $H^{(k)}_{t_0}$. Its action is $\cA_{H^{(k)}_{t_0}}(\Gamma)=\int_0^1 (\lambda(\dot\Gamma(s))+H^{(k)}_{t_0}(s,\Gamma(s)))\,ds$.
Using $\dot\Gamma(s)=kX_{H_{t_0+ks}}(\Gamma(s))$, $H^{(k)}_{t_0}(s,\Gamma(s))=kH(t_0+ks,\Gamma(s))$, and $\tau_{\lambda,H}=\lambda(X_{H_t})+H_t$, we obtain
\[
\cA_{H^{(k)}_{t_0}}(\Gamma)=\int_0^1 k\tau_{\lambda,H}(t_0+ks,\Gamma(s))\,ds=\int_0^k \tau_{\lambda,H}(t_0+\sigma,\varphi_H^\sigma(t_0,w_0))\,d\sigma=T_{t_0,w_0,k},
\]
where in the last step we used the change of variables $\sigma=ks$. This completes the proof.
\end{proof}

\subsection{Floer theory}
We briefly recall the Floer-theoretic setup used throughout the paper, in particular, on results in Section \ref{ssec-Floer-tube}.  Let $(W,\lambda_W)$ be a Liouville domain and denote by $\widehat W=W\cup_{\partial W}([1,\infty)_r\times \partial W)$ its Liouville completion. On the cylindrical end we write  $\widehat\lambda_W=r\alpha_W$, where $\alpha_W=\lambda_W|_{\partial W}$. 

A Hamiltonian $K_N:S^1\times \widehat W\to \mathbb R$ is called admissible if it is $C^2$-small on the compact part $W$ and, on the cylindrical end, agrees with a function of the radial coordinate which is linear of slope $N>0$ for $r\gg 1$. We always choose $N$ outside the Reeb spectrum of $(\partial W,\alpha_W)$. This condition ensures that no new $1$-periodic Hamiltonian orbits appear on the region where $K_N$ is exactly linear. After a small perturbation in the transition region, the set $\cP(K_N)$ of $1$-periodic orbits of $K_N$ is non-degenerate. Note that the orbits in the transition region correspond, in the usual way, to Reeb orbits on $(\partial W,\alpha_W)$ with period controlled by the slope $N$.

Let $J_N$ be a compatible almost complex structure which is cylindrical on the end. In particular, on the cylindrical end $J_N$ is invariant in the radial direction and is adapted to the contact form $\alpha_W$. We define the Floer chain complex over $\mathbb Q$ by $\CF_*(K_N,J_N)\coloneqq \bQ\langle \mathcal P(K_N)\rangle$. The differential is defined by counting finite-energy solutions $u:\mathbb R\times S^1\to \widehat W$ of the Floer equation $\partial_su+J_N(\partial_tu-X_{K_N}(u))=0$ with  asymptotics $\lim_{s\to-\infty}u(s,\cdot)=\gamma_-,  \lim_{s\to+\infty}u(s,\cdot)=\gamma_+$, where $\gamma_\pm\in\mathcal P(K_N)$. The standard maximum principle for admissible Hamiltonians and cylindrical almost complex structures implies that these Floer cylinders remain in a compact subset of $\widehat W$. Hence the usual compactness and gluing arguments apply, and for a generic choice of $J_N$ we obtain a well-defined chain complex $(\CF_*(K_N,J_N),\partial_{K_N,J_N})$. We denote its homology by $\HF_*(K_N,J_N)$. 

If $N\leq N'$, we choose a monotone homotopy from $K_N$ to $K_{N'}$. Counting the corresponding parametrized Floer solutions gives a continuation map $\HF_*(K_N,J_N)\to \HF_*(K_{N'},J_{N'})$. These maps are compatible under composition. Taking the direct limit over any cofinal sequence of admissible slopes gives the symplectic homology
\[
{\rm SH}_*(W)\coloneqq \varinjlim_N {\rm HF}_*(K_N,J_N).
\]
The resulting group is independent of all auxiliary choices, including the cofinal family $\{K_N\}$, the perturbations, and the almost complex structures.

We can also apply the filtered or the $S^1$-equivariant versions of this construction. The action functional, with the above Hamiltonian convention, is $\mathcal A_K(\gamma) = \int_{S^1} (\gamma^*\widehat\lambda_W+K_t(\gamma(t))\,dt)$. It induces an action filtration on the Floer complex, and the continuation maps are compatible with this filtration in the standard monotone sense. Applying the same cofinal family of admissible Hamiltonians to the $S^1$-equivariant Floer complex gives $\SH_*^{S^1}(W,\lambda_W)$. In the quantitative applications below, we use the filtered positive $S^1$-equivariant symplectic homology and the associated persistence module/barcode.

\subsection{Quantitative homological tools}\label{sec-bottleneck}
In this subsection, we will use the barcode of $S^1$-equivariant symplectic homology to estimate the symplectic Banach--Mazur distance. We refer \cite{PRSZ20, Ush22}  for more details. Now we begin with recalling the following definition. 

	An {\em $\bR_+$-persistence module} $\bV$ consists of a collection of vector spaces $V_s$, indexed by $s\in\bR_+$, together with linear maps (the ``structure maps'') $\sigma^{\bV}_{t,s}\colon V_s\ra V_t$ for $s\leq t$, such that $\sigma^{\bV}_{s,s}=\operatorname{id}_{V_s}$ and $\sigma^{\bV}_{u,s}=\sigma^{\bV}_{u,t}\circ\sigma^{\bV}_{t,s}\text{ whenever }s\leq t\leq u$.
If $\bV=\{V_s,\sigma^{\bV}_{t,s}\}$ and $\bW=\{W_s,\sigma^{\bW}_{t,s}\}$ are $\bR_+$-persistence modules, a {\em morphism} $F\colon\bV\ra\bW$ is a collection of linear maps $F_s\colon V_s\ra W_s$ satisfying $F_t\circ\sigma^{\bV}_{t,s}=\sigma^{\bW}_{t,s}\circ F_s\text{ whenever }s\leq t$.

\begin{remark}\label{rmk-ln}
There is a natural correspondence between $\bR_+$-persistence module and $\bR$-persistence module. Given an $\bR_+$-persistence module $\bV=\{\{V_t\}_{t\in\bR_+},\{\sigma_{t,s}\}_{s\leq t}\}$, its logarithmic version $\bV'$ is the $\bR$-persistence module defined by 
	\[
	\bV'=\left\{ \{V'_{t}\coloneqq V_{e^t}\}_{t\in\bR},\{\sigma'_{t,s}=\sigma_{e^t,e^s}\}_{s\leq t}\right\}.
	\]
\end{remark}
For any Liouville domain $(W,\lambda_W)$, one can define its $S^1$-equivariant symplectic homology. According to \cite{GH18}, the filtered positive $S^1$-equivariant symplectic homology defines an $\bR_+$-persistence module
\[
\bC\bH(W,\lambda_W)=\left\{\{{\rm CH}^L(W,\lambda_W)\}_{L\in \bR_+},\{\iota_{L_2,L_1}\}_{L_2, L_1\in \bR_+, L_1\leq L_2}\right\}
\]
where ${\rm CH}^L(W,\lambda_W)$ is a $\bQ$-vector space and $\iota_{L_2,L_1}\colon {\rm CH}^{L_1}(W,\lambda_W)\ra {\rm CH}^{L_2}(W,\lambda_W)$ is a $\bQ$-linear map. By Remark~\ref{rmk-ln}, we can obtain a logarithmic version of $S^1$-equivariant persistence module $\bC\bH'(W,\lambda_W)$. Moreover, we can define the distance between the $\bR$-persistence modules by introducing the concept of interleaving.
Let $\mathbb{V}=\{V_s,\sigma_{t,s}\}$ and $\mathbb{W}=\{W_s,\tau_{t,s}\}$ be two $\mathbb{R}$-persistence modules and let $\delta> 0$.  A {\em $\delta$-interleaving} between $\bV$ and $\bW$ consists of linear maps $\varphi_t\colon V_t\ra W_{t+\delta}\text{ and }\psi_t\colon W_t\ra V_{t+\delta}$, for all $t\in\bR$,
which commute with all structure maps: for every $s\leq t$,
\[
\tau_{t+\delta,s+\delta}\circ\varphi_s
=\varphi_t\circ\sigma_{t,s},
\qquad
\sigma_{t+\delta,s+\delta}\circ\psi_s
=\psi_t\circ\tau_{t,s}.
\]
They are also required to satisfy, for every $t\in\bR$,
\[
\psi_{t+\delta}\circ\varphi_t=\sigma_{t+2\delta,t}
\qquad\text{and}\qquad
\varphi_{t+\delta}\circ\psi_t=\tau_{t+2\delta,t}.
\]
The {\em interleaving distance} between $\bV$ and $\bW$ is
\[
d_{\rm int}(\bV,\bW)\coloneqq
\inf\{\delta\geq0\mid \bV\text{ and }\bW\text{ admit a $\delta$-interleaving}\}.
\]

By the Normal Form Theorem, when the persistence module $\bV$ is pointwise finite dimensional, we can decompose it into a direct sum of ``interval type'' persistence modules of the form $\bQ_{I_i}$ supported only in the interval $I_i \subset \bR$ with pointwise dimension 1. We collect these intervals $I_i$  with multiplicities $m_i$ as $\cB(\bV)$, the {\em barcode} of $\bV$.

We next define a distance between barcodes. For an interval $I$, let $I^\delta$ denote the interval obtained by expanding $I$ outward by $\delta$ at both finite endpoints. For a barcode $\cB$ and $\varepsilon>0$, let $\cB_\varepsilon$ be the sub-multiset of bars in $\cB$ whose lengths are greater than $\varepsilon$. A {\em matching} between two barcodes $\cB$ and $\cC$ is a bijection $\mu\colon\cB'\ra\cC'$ between sub-multisets $\cB'\subset\cB$ and $\cC'\subset\cC$. Thus $\operatorname{dom}\mu=\cB'$ and $\operatorname{im}\mu=\cC'$; bars outside these sub-multisets are left unmatched.

A {\em $\delta$-matching} between $\cB$ and $\cC$ is such a partial matching $\mu\colon\cB'\ra\cC'$ satisfying
	\begin{enumerate}
		\item $\cB_{2\delta}\subset {\rm dom}\mu.$,
		\item $\cC_{2\delta}\subset \im\mu$,
		\item If $\mu(I)=J$, then $I\subset J^{\delta}$, $J\subset I^{\delta}$.
	\end{enumerate}
	The {\em bottleneck distance} between $\cB$ and $\cC$ is
	\[
	d_{\rm bot}(\cB,\cC):=\inf\{\delta>0\mid \cB \text{ and }\cC \text{ admit a $\delta$-matching}\}.
	\]

\begin{theorem}[Isometry Theorem]\label{thm-isometry-pm}
	For any two persistence modules $\bV$ and $\bW$, we have
	\[
	d_{\rm int}(\bV,\bW)=d_{\rm bot}(\cB(\bV),\cB(\bW)).
	\]
\end{theorem}
Denote the barcode of $\bC\bH'(W,\lambda_W)$ and $\bC\bH'(V,\lambda_V)$ by $\cB(W,\lambda_W)$ and $\cB(V,\lambda_V)$. Then we can use the bottleneck distance to estimate the symplectic Banach--Mazur distance.
\begin{prop}[Theorem 1.8 in \cite{SZ21} or Proposition~3.1 in \cite{Ush22}]\label{prop-interleaving}
	For any two Liouville domains $(W,\lambda_W)$ and $(V,\lambda_V)$, we have
	\[
	d_{\rm bot}(\cB(W,\lambda_W),\cB(V,\lambda_V))\leq d_{\rm SBM}((W,\lambda_W),(V,\lambda_V)).
	\]
\end{prop}

\subsection{Equicontinuity and entropy} \label{ssec-K-equi}

\begin{prop} \label{prop-equiv-entropy}
Let $\varphi^t\colon X\ra X$ be an equicontinuous flow. Then $h_{\rm top}(\varphi)=0$.
\end{prop}
\begin{proof}
Let $g$ be any Riemannian metric on $X$ and denote by $d$ the induced Riemannian distance, which generates the topology of $X$.  We define a new $\varphi$-invariant metric on $X$ by
    \[
    d'(x,y)\coloneqq \sup_{t\in\bR}d(\varphi^t(x),\varphi^t(y)), \text{ for any }x,y\in X.
    \]
    Equicontinuity of $\varphi^t$ guarantees that $d'$ is finite and induces the same topology as $d$. Moreover, $d'$ is invariant under the flow: $d'(\varphi^t(x),\varphi^t(y))=d'(x,y)$ for all $t$. Because of this invariance, two points $x,y$ are $(T,\varepsilon)$-separated with respect to $d'$ if and only if $d'(x,y)\ge\varepsilon$. Consequently the maximal cardinality $N_{\mathrm{sep}}(\varphi,T,\varepsilon,d')$ equals the maximal cardinality of an $\varepsilon$-separated set in the compact metric space $(X,d')$.  This number is finite and independent of $T$.   Hence $\limsup_{T\to\infty}\frac{1}{T}\ln N_{\mathrm{sep}}(\varphi,T,\varepsilon,d')=0$ for every $\varepsilon>0$, and therefore $h_{\mathrm{top}}(\varphi)=0$.
\end{proof}

The next part is the verification of Example \ref{ex-non-equicontinuous}, which shows that the tube construction can produce contact forms whose Reeb flow is {\it not} equicontinuous.
\begin{proof} [Verification of Example \ref{ex-non-equicontinuous}] 
We use coordinates $(r,\varphi,r_1,\theta_1,\ldots,r_n,\theta_n)$ on $\overline{B}^{\,2}(\pi)\times \bC^n$. On $\partial W_H$, the contact form is
    \[
   \hat{\lambda}|_{\partial W_H} = \frac{1}{2} r^2 d\varphi + \frac{1}{2} r_1^2 d\theta_1+\cdots+\frac{1}{2} r_n^2 d\theta_n.  
    \]
    Since $r^2=|z|^2$ is smooth at the origin, so is $H=e-e^{r^2}$. Moreover, with the convention $\iota_{X_H}d\lambda_{\rm std}=dH$, we have $X_H=2e^{r^2}\partial_\varphi$  and hence 
    \[
    \tau(r)=\lambda_{\rm std}(X_H)+H=e-e^{r^2}+r^2e^{r^2}=e+(r^2-1)e^{r^2}.
    \]
    Since $\tau(0)=e-1$ and $\tau'(r)=2r^3e^{r^2}\geq0$ for $r\in[0,1]$, we have $\tau(r)>0$ on the whole disk. Together with $H|_{r=1}=0$, this also verifies directly that $H\in\cH(W,\lambda_{\rm std})$.

    A direct computation now gives the Reeb vector field
    \[
    R = \frac{2e^{r^2}}{\tau(r)} \partial_\varphi
    + \frac{2\pi}{\tau(r)} (\partial_{\theta_1} +\cdots+\partial_{\theta_n}).
    \]
    In particular, there is no singular factor $1/r$. The Reeb flow preserves
	the radial coordinates $r,r_1,\ldots,r_n$, while the angular coordinates
	evolve linearly    
    \[
    \dot{\varphi}=\Omega_\varphi(r)\coloneqq\frac{2e^{r^2}}{\tau(r)}, \qquad
    \dot{\theta}_1=\cdots=\dot{\theta}_n=
    \Omega_\theta(r)\coloneqq\frac{2\pi}{\tau(r)}.
    \]
    Furthermore,
    \[
    \Omega_\varphi'(r)
    =\frac{4re^{r^2}(e-e^{r^2})}{\tau(r)^2}>0
    \qquad\text{for }0<r<1.
    \]
	We shall show that this strict dependence of $\Omega_\varphi$ on $r$ forces failure of equicontinuity. 
	
	Fix $r_0=\frac{1}{2}$. For every sufficiently small $s>0$, the strict positivity above gives $\Omega_\varphi(r_0+s)\neq \Omega_\varphi(r_0).$ For such an $s$, choose two points $x_s,y_s\in \partial W_H$ as follows. They have the same angular coordinates $\varphi=\theta_1=\cdots=\theta_n=0$. Their $r$-coordinates are $r_0$ and $r_0+s$, respectively. We set
	\[ r_1(x_s)=\sqrt{\frac{H(r_0)}{\pi}}, r_1(y_s)=\sqrt{\frac{H(r_0+s)}{\pi}},	r_2=\cdots=r_n=0.\] Thus both points lie on $\partial W_H$. Moreover, as $s\to0^+$, the points $x_s$ and $y_s$ converge to the same point of $\partial W_H$. In particular, $d(x_s,y_s)\to 0$ for any fixed auxiliary Riemannian metric $d$ on $\partial W_H$.
	
	Let $\phi^t$ denote the Reeb flow. Since the $r$-coordinate is preserved, the difference of the $\varphi$-coordinates of $\phi^t(x_s)$ and $\phi^t(y_s)$ is $t(\Omega_\varphi(r_0+s)-\Omega_\varphi(r_0)) \text{ mod }2\pi $. Choose $t_s=\frac{\pi}{	\left|\Omega_\varphi(r_0+s)-\Omega_\varphi(r_0)\right|}$. Then the two points $\phi^{t_s}(x_s)$ and $\phi^{t_s}(y_s)$ have $\varphi$-coordinates differing by $\pi$. Then we obtain a uniform separation: the first $\overline{B}^{\,2}(\pi)$-component of $\phi^{t_s}(x_s)$ has polar radius $r_0$, while that of $\phi^{t_s}(y_s)$ has polar radius $r_0+s$, and their angular coordinates differ by $\pi$. Therefore, in the Euclidean metric on the $\overline{B}^{\,2}(\pi)$-factor, the distance between these two projections is $r_0+(r_0+s)\geq 2r_0=1$. Consequently, after comparing the chosen Riemannian metric on 	$\partial W_H$ with the ambient product metric on a fixed compact neighborhood, there exists a constant $\varepsilon_0>0$, independent of $s$, such that $d(\phi^{t_s}(x_s),\phi^{t_s}(y_s))\geq \varepsilon_0.$

Finally, given any $\delta>0$, choose $s>0$ so small that $d(x_s,y_s)<\delta$. The time $t_s$ chosen above then satisfies $d(\phi^{t_s}(x_s),\phi^{t_s}(y_s))\geq \varepsilon_0$.
	This contradicts equicontinuity of the Reeb flow. Hence the Reeb flow of $(\alpha_H)_n$ is not equicontinuous.
	\end{proof}


\subsection{Estimation of capacity} \label{ssec-capacity} The proof of Proposition \ref{prop-wG-finite} relies on a technique from symplectically uniruled manifolds. Recall a closed symplectic manifold $(M,\omega)$ is called \emph{symplectically uniruled} if there exists a nonzero genus-zero Gromov--Witten invariant of the form
$\langle \mathrm{pt},a_2,\dots,a_k\rangle_{k,A}^M,$
where $k\ge 1$, $0\neq A\in H_2(M;\mathbb Z)$, and $a_i\in H^*(M;\mathbb Q)$. The following proposition is essentially due to \cite{Lu06}.

\begin{prop}\label{prop:uniruled-finite}
If $(M,\omega)$ is symplectically uniruled, then
$w_G(M\times \mathbb C)<\infty$. In particular, if $(W,\omega)$ can be symplectically embedded into a closed symplectically uniruled manifold, then $w_G(W\times\bC)<\infty$.
\end{prop}

\begin{proof}
By assumption, there exist a homology class $A\neq 0$ and classes $a_2,\dots,a_k$ such that $\langle \mathrm{pt},a_2,\dots,a_k\rangle_{k,A}^M\neq 0.$
Since this invariant is nonzero and $A\neq0$, we have $\omega(A)>0$, so $A$ is non-torsion. Hence there is a class $d\in H^2(M;\mathbb Q)$ such that $\langle d,A\rangle\neq0$. Set $r\coloneqq\max\{0,3-k\}$ and $m\coloneqq k+r$, and let $b_2,\dots,b_m$ be the list $a_2,\dots,a_k$ followed by $r$ copies of $d$. The divisor axiom gives $\langle \mathrm{pt},b_2,\dots,b_m\rangle_{m,A}^M=\langle d,A\rangle^r\langle \mathrm{pt},a_2,\dots,a_k\rangle_{k,A}^M\neq0$. In particular, $m\ge3$. 

For $a>0$, let $S^2(a)=(S^2,\omega_a)$ denote the symplectic two-sphere of total area $\int_{S^2}\omega_a=a$. Then we  consider
\[ X_a\coloneqq M\times S^2(a),\,\, \Omega_a\coloneqq\omega\oplus \omega_a,\,\, \mbox{and}\,\, B\coloneqq(A,0)\in H_2(X_a;\mathbb Z).\]For $i=2,\dots,m$, define $\widetilde b_i\coloneqq {\rm pr}_1^*b_i\in H^*(X_a;\mathbb Q)$. Writing $\mathrm{pt}_{X_a}={\rm pr}_1^*(\mathrm{pt}_M)\cup{\rm pr}_2^*(\mathrm{pt}_{S^2(a)})$, observe that every genus-zero stable map to $X_a$ in class $B=(A,0)$ is constant in the $S^2(a)$-factor. Thus the corresponding virtual moduli cycle is the product of the virtual moduli cycle for maps to $M$ in class $A$ with $[S^2(a)]$. Pairing the latter factor with $\mathrm{pt}_{S^2(a)}$ gives, equivalently by the genus-zero product formula,
\[
\langle \mathrm{pt}_{X_a},\widetilde b_2,\dots,\widetilde b_m\rangle^{X_a}_{m,B}=\langle\mathrm{pt},b_2,\dots,b_m\rangle^M_{m,A}\neq0.
\]

Since $B\neq 0$ and $m\ge 3$, Theorem~4.1 of \cite{HLS21} applies and yields 
\[ w_G(X_a)\le \Omega_a(B)=\omega(A).\] 
Let \(c>0\) and suppose that there exists a symplectic embedding $\varphi\colon B^{2n+2}(c)\hookrightarrow M\times\mathbb C$. For any \(0<c'<c\), the set \(\varphi(\overline{B}^{\,2n+2}(c'))\) is compact. Hence its projection to the \(\mathbb C\)-factor is contained in \(B^2(a)\) for some \(a>0\). Therefore the restriction of \(\varphi\) gives $B^{2n+2}(c')\hookrightarrow M\times B^2(a) \hookrightarrow M\times S^2(a)$. It follows that $c'\le w_G(M\times S^2(a))\le \omega(A)$. Letting \(c'\ra c\), we obtain \(c\le\omega(A)\). Taking the supremum over all such \(c\) proves that $w_G(M\times\mathbb C)\le\omega(A)<\infty$.
\end{proof}

\begin{ex}  Every Hamiltonian $S^1$-manifold is symplectically uniruled by \cite{McD09a}.\end{ex}

The following result is another elementary result that will be useful to the proof of Proposition \ref{prop-wG-finite}. 

\begin{lemma}\label{lem:surface-width-area}
Let $(\Sigma,\omega)$ be a connected compact symplectic surface, possibly with boundary. Then
$w_G(\Sigma,\omega)=\int_\Sigma \omega.$
\end{lemma}

\begin{proof}
Suppose that there exists a symplectic embedding $\varphi\colon B^2(r)\hookrightarrow \Sigma$. Then we have $ r=\int_{B^2(r)}\omega_{\mathrm{std}}=\int_{\varphi(B^2(r))}\omega\le\int_\Sigma \omega$, hence $w_G(\Sigma,\omega)\le \int_\Sigma \omega.$

For the reverse inequality, fix $\varepsilon>0$. By the classification of compact connected surfaces, $\Sigma$ can be represented as a quotient of a planar polygon $P$ by pairwise edge identifications, possibly after removing finitely many disjoint open disks to account for the boundary. Let $q\colon P\to \Sigma$ be the quotient map. Then $q$ is injective on the interior of $P$.

Choose a smoothly bounded domain $D_\varepsilon$, compactly supported in ${\rm Int}(P)$ and diffeomorphic to a disk, such that $\int_{D_\varepsilon} q^*\omega > \int_\Sigma \omega - \varepsilon$. By Moser's theorem in dimension two, the symplectic disk $(D_\varepsilon,q^*\omega)$ is symplectomorphic to the standard disk $\overline{B}^{\,2}(r_\varepsilon)$ with $r_\varepsilon=\int_{D_\varepsilon} q^*\omega$. Since $q$ is injective on $D_\varepsilon$, this gives a symplectic embedding $\overline{B}^{\,2}(r_\varepsilon)\hookrightarrow \Sigma$, and therefore $w_G(\Sigma,\omega)\ge r_\varepsilon > \int_\Sigma \omega-\varepsilon$. Letting $\varepsilon\to 0$ proves the lemma.
\end{proof}



\section{Proof of Theorem~\ref{thm-estimation}}\label{sec-thm-A}
Before proving Theorem~\ref{thm-estimation}, we first show some basic properties of the Thompson-type distance $d_{\rm T}$ on $\cH(W,\lambda)$.
\begin{prop}\label{prop-ln-metric}
\begin{enumerate}[leftmargin=*]
\item We have the following equivalent expression,
    \[
    d_{\rm T}(H,G)=\inf_{\varphi\in{\rm Symp}(W,\lambda)}\sup_{ S^1\times {\rm Int}(W)} \bigl|\ln H-\ln (G\circ\varphi)\bigr| = \inf_{\varphi\in{\rm Symp}(W,\lambda)}\|\ln H-\ln (G\circ\varphi)\|_{\infty}.
    \]
\item The Thompson-type distance $d_{\rm T}$ satisfies the basic metric properties:
\begin{itemize}
\item  $d_{\rm T}(H,G)=d_{\rm T}(G,H)$.
\item  $d_{\rm T}(H,G)\ge 0$.
\item For any $H,G,F\in\cH(W,\lambda)$,
        $d_{\rm T}(H,F)\leq d_{\rm T}(H,G)+d_{\rm T}(G,F)$.
\end{itemize}
Consequently, $d_{\rm T}$ is an extended pseudo-metric on $\mathcal{H}$ (i.e., it may take the value $+\infty$).
\item The Thompson-type distance $d_{\rm T}$ satisfies the following invariance properties:
    \begin{itemize}
    	\item For any $\varphi,\psi\in{\rm Symp}(W,\lambda)$ and $H,G\in\cH(W,\lambda)$, we have
    	\[
    	d_{\rm T}(H,G)=d_{\rm T}(H\circ \varphi, G\circ\psi). 
    	\]
        \item  For any constant $a>0$, $d_{\rm T}(aH, aG)=d_{\rm T}(H,G)$.
        \item  Define $H_{t_0}(t,x)\coloneqq H(t+t_0,x)$ for any $t_0\in S^1$. Then $d_{\rm T}(H_{t_0},G_{t_0})=d_{\rm T}(H,G)$.
    \end{itemize}
\end{enumerate}
\end{prop}
\begin{proof}
	The verification of (1) and (3) is straightforward and therefore omitted.
	
	For (2), the first bullet comes from the following equivalence,
	\[
	\frac{1}{C}H\leq G\circ\varphi\leq CH\Longleftrightarrow \frac{1}{C}G\leq H\circ\varphi^{-1}\leq CG.
	\]
	The second bullet is trivial. For the third bullet, if for some $C,D\geq 1$ and $\varphi,\psi\in{\rm Symp}(W,\lambda)$, we have
	\[
	\frac{1}{C}H\leq G\circ\varphi\leq CH, \frac{1}{D}G\leq F \circ\psi\leq DG
	\] 
	then
	\[
	\frac{1}{CD}H\leq \frac{1}{D}G\circ\varphi\leq F\circ\psi\circ\varphi\leq D\,G\circ \varphi\leq CD\, H.\qedhere
	\]	
\end{proof}

For the next proposition, recall that $\cH^+(W,\lambda)$ is a subset of $\cH(W, \lambda)$ with an extra condition that $\lambda(X_{H_t})\ge0$. 

\begin{prop}\label{prop-up-bound}
    Let $H,G\in\cH^+(W,\lambda)$ and $C\ge 1$.
    \begin{enumerate}
        \item If $H\ge \frac{1}{C}G$, then there exists a Liouville embedding $i_C\colon \frac{1}{C}W_G\hookrightarrow W_H$.
        \item If $H\le CG$, then there exists a Liouville embedding $i'_C\colon \frac{1}{C}W_H\hookrightarrow W_G$. 
        \item If $\frac{1}{C}G\le H\le CG$, then both embeddings exist (by (1) and (2) above) and satisfy
        \[
        \frac{1}{C^{2}}W_H\subset i_C\!\left(\frac{1}{C}W_G\right)\subset W_H,
        \qquad 
        \frac{1}{C^{2}}W_G\subset i'_C\!\left(\frac{1}{C}W_H\right)\subset W_G .
        \]
    \end{enumerate}
\end{prop}

\begin{proof}
    Recall that for a function $F\in\cH^+(W,\lambda)$, the domain $W_F$ is defined as
    \[
    W_F=\left\{(w,z)\in W\times\mathbb C\ \middle|\ z=re^{2\pi it},\ \pi|z|^2\le F(t,w)\right\}.
    \]
    and the Liouville form on $W\times\mathbb{C}$ is $\hat{\lambda}=\lambda+\pi r^2dt$.

    (1) Assume $H\ge\frac{1}{C}G$.  Consider the map
    \[
    i_C\colon \frac{1}{C}W_G\rightarrow W_H,\qquad (w,z)\mapsto(w,z).
    \]
    Here, the rescaled domain $\frac{1}{C}W_G$ is given by
    \begin{align*}
    	 \frac{1}{C}W_G
    &=\bigl\{(\sL_{\lambda}^{\ln(1/C)}(w),\,z/\sqrt{C})\in W\times\mathbb{C}
    \;\big|\; (w,z)\in W_G\bigr\}\\
    & =\Bigl\{(w,z)\in \sL_{\lambda}^{\ln(1/C)}(W)\times\mathbb{C}
    \;\Big|\; \pi\bigl|\sqrt{C}\cdot z\bigr|^{2}\le G(t,\sL_{\lambda}^{\ln C}(w))\Bigr\}
    \end{align*}
    where $\sL_{\lambda}^{\ln(1/C)}$ is the time-$\ln(1/C) $ Liouville flow on $W$.
     For  $(w,z)\in\frac{1}{C}W_G$, we have
    \[
    \pi|z|^{2}\le\frac{1}{C}\,G(t,\sL_{\lambda}^{\ln C}(w))\le H(t,\sL_{\lambda}^{\ln C}(w)).
    \]
    
    Let $Y_\lambda$ be the Liouville vector field on $W$. For each fixed $t\in S^1$, the function $s\mapsto H(t,\sL_\lambda^s(w))$ is non-increasing, because
    \[
    \frac{d}{ds}H(t,\sL_\lambda^s(w))=d_wH_t(Y_\lambda)(\sL_\lambda^s(w))=d\lambda(X_{H_t},Y_\lambda)(\sL_\lambda^s(w))=-\lambda(X_{H_t})(\sL_\lambda^s(w))\le 0. 
    \]
    Hence, $H(t,\sL_\lambda^{\ln C}(w))\le H(t,w)$. Therefore $\pi|z|^2\le H(t,w),$ which means $(w,z)\in W_H$. Thus $i_C$ is well-defined.
    Since $i_C$ is simply the inclusion, it satisfies $i_C^*\hat\lambda=\hat\lambda$; consequently it is a Liouville embedding.

    (2) The second embedding is completely analogous: if $H\le CG$ we set
    \[
    i'_C\colon \frac{1}{C}W_H\rightarrow W_G,\qquad (w,z)\mapsto(w,z),
    \]
    and the same argument shows that it is a Liouville embedding.

   (3) Assume now $\frac{1}{C}G\le H\le CG$.  Then both (1) and (2) apply.
    To prove the first inclusion in (3), observe that
    \[
    \frac{1}{C^{2}}W_H=\Bigl\{(w,z)\in\sL_{\lambda}^{2\ln(1/C)}(W)\times\mathbb{C}\;\Big|\; \pi\bigl|Cz\bigr|^{2}\le H(t,\sL_{\lambda}^{2\ln C}(w))\Bigr\}.
    \]
    On the other hand,
    \[
    i_C\!\left(\frac{1}{C}W_G\right)
    =\Bigl\{(w,z)\in\sL_{\lambda}^{\ln(1/C)}(W)\times\mathbb{C}\;\Big|\; \pi\bigl|\sqrt{C}\,z\bigr|^{2}\le G(t,\sL_{\lambda}^{\ln C}(w))\Bigr\}.
    \]
    For $(w,z)\in\frac{1}{C^{2}}W_H$ we have
    \[
    \pi\bigl|\sqrt{C}\,z\bigr|^{2}=\frac{1}{C}\,\pi\bigl|Cz\bigr|^{2}    \le \frac{1}{C}\,H(t,\sL_{\lambda}^{2\ln C}(w))\le\frac{1}{C}\,G(t,\sL_{\lambda}^{2\ln C}(w))
    \]
    where the last inequality uses $H\le CG$ and the fact that $G$ is also non-increasing along the forward Liouville flow. Since $2\ln C\ge \ln C$, we have
    \[
   G(t,\sL_{\lambda}^{2\ln C}(w))\le G(t,\sL_{\lambda}^{\ln C}(w)).
    \]
    hence
    \[
    \pi\bigl|\sqrt{C}\,z\bigr|^{2}\le \frac{1}{C}\,G(t,\sL_{\lambda}^{\ln C}(w))    \le G(t,\sL_{\lambda}^{\ln C}(w)).
    \]
    This shows $(w,z)\in i_C(\frac{1}{C}W_G)$, establishing $\frac{1}{C^{2}}W_H\subset i_C(\frac{1}{C}W_G)$.

    The second inclusion in (3) is proved in the same way, exchanging the roles of $H$ and $G$ and using the embedding $i'_C$.
\end{proof}  
\begin{proof}[Proof of Theorem~\ref{thm-estimation}]
By Proposition~\ref{prop-up-bound}, for any $H,G\in\cH^+(W,\lambda)$, we have 
\[
d_{\rm SBM}({\rm U}(H),{\rm U}(G))\leq\inf\left\{\ln C\geq 0\,\left|\,\frac{1}{C}H\leq G\leq C H\right\}\right..
\]
By Lemma~\ref{lemma-reparametrization}, we have $(W_{G},\hat\lambda)\cong(W_{G\circ \varphi},\hat\lambda)$ for any $\varphi\in{\rm Symp}(W,\lambda)$. Then we have
\begin{align*}
	d_{\rm SBM}({\rm U}(H),{\rm U}(G))&\leq\inf\left\{\ln C\,\left|\,\exists\,\varphi\in {\rm Symp}(W,\lambda),\text{ s.t. } \frac{1}{C}H\leq G\circ\varphi \leq C H\right\}\right.\\
	&=d_{\rm T}(H,G).
\end{align*}
This completes the proof.  \end{proof}

\section{Proof of Theorem~\ref{thm-SBM-CBM-ln}}
We begin this section by constructing the canonical diffeomorphism $\varphi_{H,G}$ mentioned right below Corollary \ref{cor-estimation-class}, relating contact boundaries of tubes that are constructed from two input Hamiltonian functions $H$ and $G$.
 \begin{prop}\label{prop-contact-bdy}
	For any $H,G\in \cH(W,\lambda)$, the Liouville flow defines a diffeomorphism $\varphi_{H,G}\colon \partial W_H\ra \partial W_G,	p\mapsto \sL_{\hat{\lambda}}^{\,f_{H,G}(p)}(p)$ such that $\varphi_{H,G}^*(\alpha_G)=e^{f_{H,G}}\alpha_H$. Here, $f_{H,G}(p)$ is the unique time such that	$\sL_{\hat{\lambda}}^{\,f_{H,G}(p)}(p)\in\partial W_G$ 	for every $p\in\partial W_H$. Moreover, for any $H,G,F\in\cH(W,\lambda)$, we have
 $\varphi_{H,F}=\varphi_{G,F}\circ\varphi_{H,G}$.
\end{prop}
\begin{proof}
Let $\widehat Y$ be the Liouville vector field of $\hat\lambda$. Since $\hat\lambda=\lambda+\frac12r^2d\theta$, we have $\widehat Y=Y_\lambda+\frac12r\partial_r,$ and hence $\sL_{\hat\lambda}^s(w,z)=(\sL_\lambda^s(w),e^{s/2}z).$

We first define the ``hitting time'' from $\partial W_H$ to $\partial W_G$. Fix $p=(w,z)\in \partial W_H$.
If $z=0$, then $p\in \partial W\times\{0\}$, because $p\in \partial W_H$ implies $\pi|z|^2=H(t,w)=0$ (hence $w\in \partial W$). Since $G|_{\partial W}=0$, we also have $p\in \partial W_G$, so in this case we set $f_{H,G}(p)$ to be $0$.

Now assume $z\neq 0$, and write $z=re^{2\pi i t_0}$. Along the Liouville flow, the angular coordinate is constant, so $\sL_{\hat\lambda}^s(p)\in \partial W_G$ is equivalent to $\pi e^s|z|^2=G(t_0,\sL_\lambda^s(w))$.
Equivalently, if we define $$g_p(s)\coloneqq e^{-s}(\pi e^s|z|^2-G(t_0,\sL_\lambda^s(w)))=\pi|z|^2-e^{-s}G(t_0,\sL_\lambda^s(w)),$$ then $g_p(s)=0$ if and only if $\sL_{\hat\lambda}^s(p)\in \partial W_G$. We claim that $g_p$ is strictly increasing. Indeed, using $d_wG_{t_0}=\iota_{X_{G_{t_0}}}d\lambda$ and $\iota_{Y_\lambda}d\lambda=\lambda$, we get $d_wG_{t_0}(Y_\lambda)=d\lambda(X_{G_{t_0}},Y_\lambda)=-\lambda(X_{G_{t_0}})$.  Therefore
\begin{align*}
\frac{d}{ds}(e^{-s}G(t_0,\sL_\lambda^s(w)))&=e^{-s}(d_wG_{t_0}(Y_\lambda)-G_{t_0})(\sL_\lambda^s(w))\\
& =-e^{-s}\tau_{\lambda,G}(t_0,\sL_\lambda^s(w))<0.
\end{align*}
Hence,
$g_p'(s)=e^{-s}\tau_{\lambda,G}(t_0,\sL_\lambda^s(w))>0,$
and then $g_p$ is strictly increasing.

Next, we prove that $g_p$ has a unique zero. Since $z\neq0$ and $p\in\partial W_H$, we have $H(t_0,w)=\pi|z|^2>0$, and hence $w\in{\rm Int}(W)$. Let $s_+(w)\in(0,\infty]$ be the maximal forward time for which the Liouville trajectory through $w$ remains in $W$. If $s_+(w)<\infty$, then $\sL_\lambda^s(w)$ approaches $\partial W$ as $s\ra s_+(w)$. Consequently, $G(t_0,\sL_\lambda^s(w))\to0$ and $g_p(s)\to\pi|z|^2>0$. If $s_+(w)=\infty$, then $G$ is bounded on the compact manifold $S^1\times W$, so $e^{-s}G(t_0,\sL_\lambda^s(w))\ra 0$ and hence $\quad g_p(s)\ra\pi|z|^2>0$ as $s\to\infty$. 
	Thus $g_p$ is positive at some forward time in either case. 
	
	On the other hand, the backward Liouville orbit of $w$ eventually lies in a compact neighborhood of ${\rm Sk}(W)$ contained in ${\rm Int}(W)$. Indeed, the compact sets $\sL_\lambda^{-T}(W)$ decrease to ${\rm Sk}(W)$ as $T\to\infty$, so they are eventually contained in any prescribed neighborhood of the skeleton. Since $G(t_0,\cdot)>0$ on ${\rm Int}(W)$, there exist $c>0$ and $s_0<0$ such that $G(t_0,\sL_\lambda^s(w))\ge c$ for all $s\le s_0$. Therefore $g_p(s)\le \pi|z|^2-ce^{-s}\ra-\infty$ as $s\to-\infty$.	By strict monotonicity, there exists a unique $t_p\in\bR$ such that $g_p(t_p)=0$, equivalently $\sL_{\hat\lambda}^{t_p}(p)\in \partial W_G$.

We now define $f_{H,G}(p)\coloneqq t_p$ and $\varphi_{H,G}(p)\coloneqq \sL_{\hat\lambda}^{\,f_{H,G}(p)}(p)$. Then we show that $f_{H,G}$ is smooth. On $\partial W_H\setminus (\partial W\times\{0\})$, one may choose the angle $t_0$ smoothly on a neighborhood, and then the equation $g_p(s)=0$ together with $\partial_s g_p(s)>0$ implies by the implicit function theorem that $f_{H,G}$ is smooth there.

It remains to treat points of $\partial W\times\{0\}$. Let $p_0\in \partial W\times\{0\}$. Since $\widehat Y$ is transverse to both $\partial W_H$ and $\partial W_G$, by the flow-box theorem there exist a neighborhood $U$ of $p_0$, a manifold $V$, and coordinates $(q,s)\in V\times (-\varepsilon,\varepsilon)$ on $U$ such that $\widehat Y=\partial_s$. In these coordinates, $\partial W_H\cap U$ and $\partial W_G\cap U$ are graphs $s=u_H(q)$ and $s=u_G(q)$ of smooth functions. Hence, for $p=(q,u_H(q))\in \partial W_H\cap U$, we have $f_{H,G}(p)=u_G(q)-u_H(q)$, so $f_{H,G}$ is smooth near $p_0$. Therefore $f_{H,G}$ is smooth on all of $\partial W_H$, and so is $\varphi_{H,G}$.

Constructing similarly $f_{G,H}$ and $\varphi_{G,H}$, uniqueness of the hitting time gives
$
\varphi_{G,H}\circ \varphi_{H,G}=\id_{\partial W_H}, \varphi_{H,G}\circ \varphi_{G,H}=\id_{\partial W_G}.
$
Thus $\varphi_{H,G}$ is a diffeomorphism with inverse $\varphi_{G,H}$.

We now prove the conformal relation. Since $(\sL_{\hat\lambda}^s)^*\hat\lambda=e^s\hat\lambda$ and $\hat\lambda(\widehat Y)=0$, for any $v\in T_p\partial W_H$ we have
$
d\varphi_{H,G}(v)=d(\sL_{\hat\lambda}^{\,f_{H,G}(p)})_p(v)+df_{H,G}(v)\,\widehat Y_{\varphi_{H,G}(p)}.
$
Therefore
\[
\varphi_{H,G}^*(\alpha_G)(v)= \hat\lambda_{\varphi_{H,G}(p)}\!\left(d(\sL_{\hat\lambda}^{\,f_{H,G}(p)})_p(v)\right) = e^{f_{H,G}(p)}\hat\lambda_p(v)=e^{f_{H,G}(p)}\alpha_H(v).
\]
Hence, $\varphi_{H,G}^*(\alpha_G)=e^{f_{H,G}}\alpha_H$. Finally, for $p\in \partial W_H$, both $\varphi_{H,F}(p)$ and $\varphi_{G,F}(\varphi_{H,G}(p))$ lie on the same Liouville orbit through $p$, and both belong to $\partial W_F$. Then the uniqueness of the hitting time with $\partial W_F$ implies that $\varphi_{H,F}=\varphi_{G,F}\circ \varphi_{H,G}$.
\end{proof}

Next, we establish several metric properties of  $d_{\rm CBM}(\alpha_H,\alpha_G)$.
\begin{prop}\label{prop-CBM}
The distance $d_{\rm CBM}$ satisfies the following properties. For any $H,G,K\in\cH(W,\lambda)$, any $\varphi\in{\rm Cont}_0(\partial W_H,\ker\alpha_H)$ and $\psi\in{\rm Cont}_0(\partial W_G,\ker\alpha_G)$,
	\begin{enumerate}
		\item $d_{\rm CBM}(\alpha_H,\alpha_G)\geq 0$ and $d_{\rm CBM}(\alpha_H,\alpha_H)= 0$;
		 \item $d_{\rm CBM}(\varphi^*(\alpha_H),\psi^*(\alpha_G))=d_{\rm CBM}(\alpha_H,\alpha_G)$. 
		\item $d_{\rm CBM}(\alpha_H,\alpha_G)=d_{\rm CBM}(\alpha_G,\alpha_H)$;
		\item $d_{\rm CBM}(\alpha_G,\alpha_K)=d_{\rm CBM}(\varphi_{H,G}^*(\alpha_G),\varphi_{H,K}^*(\alpha_K))$;
		\item $d_{\rm CBM}(\alpha_H,\alpha_K)\leq d_{\rm CBM}(\alpha_H,\alpha_G)+d_{\rm CBM}(\alpha_G,\alpha_K)$;
			\end{enumerate}
\end{prop}
\begin{proof}
	The item (1) is trivial. 
	
	For (2), for any   $\varphi, \varphi_0\in{\rm Cont}_0(\partial W_H,\ker\alpha_H)$, if 
	$
	\varphi^*_0\varphi_{H,G}^*(\alpha_G)=e^f\varphi^*\alpha_H=\varphi^*(e^{f\circ\varphi^{-1}}\alpha_H),
	$
	then
	$
	(\varphi^{-1})^*\varphi^*_0\varphi_{H,G}^*(\alpha_G)=e^{f\circ\varphi^{-1}}\alpha_H.
	$
	By definition,
	\begin{align*}
		d_{\rm CBM}(\varphi^*\alpha_H,\alpha_G)
		&=\inf\left\{\ln C\left| 
	\begin{array}{l}
		\exists\varphi_0\in{\rm Cont}_0(\partial W_{H},\ker\alpha_H), \text{ s.t.}\\
		\varphi^*_0\varphi_{H,G}^*(\alpha_G)=e^f\varphi^*\alpha_H,|f|\leq \ln C
	\end{array}
	\right\}\right.\\
	&=\inf\left\{\ln C\left| 
	\begin{array}{l}
		\exists\varphi_1=\varphi_0\circ \varphi\in{\rm Cont}_0(\partial W_{H},\ker\alpha_H), \text{ s.t.}\\
		\varphi_1^*\varphi_{H,G}^*(\alpha_G)=e^{f\circ\varphi^{-1}}\alpha_H,|f\circ\varphi^{-1}|\leq \ln C
	\end{array}
	\right\}\right.\\
	&=d_{\rm CBM}(\alpha_H,\alpha_G).
	\end{align*}
For any $\psi\in{\rm Cont}_0(\partial W_G,\ker\alpha_G)$, since $\varphi_{G,H}\circ\varphi_{H,G}=\id_{\partial W_H}$, we can define  $\varphi\coloneqq\varphi_{G,H}\circ \psi\circ\varphi_{H,G}\in{\rm Cont}_0(\partial W_H,\ker\alpha_H).$ Equivalently, $\psi=\varphi_{H,G}\circ\varphi\circ\varphi_{G,H}$, and the pullbacks satisfy $\varphi^*\varphi_{H,G}^*=\varphi_{H,G}^*\psi^*.$ This gives an identification of ${\rm Cont}_0(\partial W_H,\ker\alpha_H)$ and ${\rm Cont}_0(\partial W_G,\ker\alpha_G)$.
By definition,
\begin{align*}
	d_{\rm CBM}(\alpha_H,\psi^*\alpha_G)
		&=\inf\left\{\ln C\left| 
	\begin{array}{l}
		\exists\varphi_0\in{\rm Cont}_0(\partial W_{H},\ker\alpha_H), \text{ s.t.}\\
		\varphi^*_0\varphi_{H,G}^*\psi^*(\alpha_G)=e^f\alpha_H,|f|\leq \ln C
	\end{array}
	\right\}\right.\\
	&=\inf\left\{\ln C\left| 
	\begin{array}{l}
		\exists\varphi_1=\varphi\circ\varphi_0\in{\rm Cont}_0(\partial W_{H},\ker\alpha_H), \text{ s.t.}\\
		\varphi^*_1\varphi_{H,G}^*(\alpha_G)=e^f\alpha_H,|f|\leq \ln C
	\end{array}
	\right\}\right.\\
	&=d_{\rm CBM}(\alpha_H,\alpha_G).
\end{align*}
For (3), by definition, 
\begin{align*}
	d_{\rm CBM}\left(\alpha_H,\alpha_G\right)
		&=\inf\left\{\ln C\left| 
	\begin{array}{l}
		\exists\varphi\in{\rm Cont}_0(\partial W_{H},\ker\alpha_H), \text{ s.t.}\\
		\varphi^*\varphi_{H,G}^*(\alpha_G)=e^f\alpha_H,|f|\leq \ln C
	\end{array}
	\right\}\right. \\
	&=\inf\left\{\ln C\left| 
	\begin{array}{l}
		\exists\psi\in{\rm Cont}_0(\partial W_{G},\ker\alpha_G), \text{ s.t.}\\
		\varphi_{H,G}^*\psi^*(\alpha_G)=e^f\alpha_H,|f|\leq \ln C
	\end{array}
	\right\}\right.	\\
	&=\inf\left\{\ln C\left| 
	\begin{array}{l}
		\exists\psi\in{\rm Cont}_0(\partial W_{G},\ker\alpha_G), \text{ s.t.}\\
		\psi^*(\alpha_G)=e^{f\circ\varphi_{G,H}}\varphi_{G,H}^*\alpha_H,|f\circ\varphi_{G,H}|\leq \ln C
	\end{array}
	\right\}\right.	\\
	&=\inf\left\{\ln C\left| 
	\begin{array}{l}
		\exists\psi\in{\rm Cont}_0(\partial W_{G},\ker\alpha_G), \text{ s.t.}\\
		(\psi^{-1})^*\varphi_{G,H}^*\alpha_H=e^{-f\circ\varphi_{G,H}\circ\psi^{-1}}\alpha_G,\\
		|f\circ\varphi_{G,H}\circ\psi^{-1}|\leq \ln C
	\end{array}
	\right\}\right.	\\
	&=d_{\rm CBM}(\alpha_G,\alpha_H).
\end{align*}
The second equality uses the conjugation identity $\varphi^*\varphi_{H,G}^*=\varphi_{H,G}^*\psi^*$ above. The third and fourth equalities follow by pulling back with $\varphi_{G,H}$ and $\psi^{-1}$, respectively.

For (4), by definition,
\begin{align*}
	d_{\rm CBM}\left(\varphi_{H,G}^*(\alpha_G),\varphi_{H,K}^*(\alpha_K)\right)
		=&\inf\left\{\ln C\left| 
	\begin{array}{l}
		\exists\varphi\in{\rm Cont}_0(\partial W_{H},\ker\alpha_H), \text{ s.t.}\\
		\varphi^*\varphi_{H,G}^*(\alpha_G)=e^f\varphi_{H,K}^*(\alpha_K),|f|\leq \ln C
	\end{array}
	\right\}\right. \\
	=&\inf\left\{\ln C\left| 
	\begin{array}{l}
		\exists\psi\in{\rm Cont}_0(\partial W_{G},\ker\alpha_G), \text{ s.t.}\\
		\varphi_{H,G}^*\psi^*(\alpha_G)=e^f\varphi_{H,K}^*(\alpha_K),|f|\leq \ln C
	\end{array}
	\right\}\right.	\\
	=&\inf\left\{\ln C\left| 
	\begin{array}{l}
		\exists\psi\in{\rm Cont}_0(\partial W_{G},\ker\alpha_G), \text{ s.t.}\\
		\psi^*(\alpha_G)=e^{f\circ\varphi_{G,H}}\varphi_{G,K}^*\alpha_K,\\
		|f\circ\varphi_{G,H}|\leq \ln C
	\end{array}
	\right\}\right.	\\
	=&d_{\rm CBM}(\alpha_G,\alpha_K).
\end{align*}
The second equality uses the same conjugation identity. For the third equality, pull back by $\varphi_{G,H}$ and use
	$\varphi_{G,K}=\varphi_{H,K}\circ\varphi_{G,H}$. 

For (5), by (3) in Lemma~7 in \cite{RZ21}, the triangle inequality of $d_{\rm CBM}$ on $(\partial W_H,\ker\alpha_H)$,
\begin{align*}
	d_{\rm CBM}(\alpha_H,\alpha_K)
	&=d_{\rm CBM}(\alpha_H,\varphi_{H,K}^*\alpha_K)\\
	&\leq d_{\rm CBM}(\alpha_H,\varphi_{H,G}^*\alpha_G)+d_{\rm CBM}(\varphi^*_{H,G}\alpha_G,\varphi^*_{H,K}\alpha_K)\\
	&=d_{\rm CBM}(\alpha_H,\alpha_G)+d_{\rm CBM}(\alpha_G,\alpha_K).\qedhere
\end{align*}
\end{proof}

Now, we can give the proof of Theorem~\ref{thm-SBM-CBM-ln}

\begin{proof}[Proof of Theorem~\ref{thm-SBM-CBM-ln}] The proof will be divided into two parts. 
First, we will show that for any $H,G\in\cH(W,\lambda)$, we have
	\[
	d_{\rm SBM}({\rm U}(H),{\rm U}(G))\leq d_{\rm CBM}(\alpha_H,\alpha_G).
	\]
	Recall some notations from \cite{RZ21}. Let $(M,\xi=\ker\alpha_0)$ be a closed and Liouville-fillable. Consider the associated Liouville domain $(W,\lambda,L)$ such that  $\partial W=M$, $\lambda|_{M}=\alpha_0$ and the flow of $L$ is complete for $t<0$. Let $\alpha\in O_{(M,\xi)}(\alpha_0)$, then $\alpha=e^f\alpha_0$ for some $f\colon M\ra\bR$. Define
$
W^\alpha\coloneqq \{(s,x)\in\hat{W}\mid s\leq e^{f(x)}\}.
$ 
By Lemma~8 in \cite{RZ21}, for any $\alpha_1,\alpha_2\in O_{(M,\xi)}(\alpha_0)$, we have
\begin{equation}\label{eq-SBM-CBM}
	d_{\rm SBM}(W^{\alpha_1},W^{\alpha_2})\leq d_{\rm CBM}(\alpha_1,\alpha_2).
\end{equation}
Take $(M,\alpha_0)=(\partial W_H,\alpha_H)$ and $(W,\lambda,L)=(W_H,\hat{\lambda},\widehat Y)$, where $\widehat Y$ is the Liouville vector field of $\hat\lambda$. Then $\varphi_{H,G}^*\alpha_G=e^{f_{H,G}}\alpha_H$. Under the natural identification induced by $\varphi_{H,G}$, the domain
$W^{\varphi_{H,G}^*(\alpha_G)}$ is Liouville isomorphic to $W_G$. Apply \eqref{eq-SBM-CBM}, we have
\[
d_{\rm SBM}(W^{\alpha_H},W^{\varphi_{H,G}^*(\alpha_G)})\leq d_{\rm CBM}(\alpha_H,\varphi^*_{H,G}\alpha_G)= d_{\rm CBM}(\alpha_H,\alpha_G).
\]

Then we will show that for any $H,G\in\cH^+(W,\lambda)$, we have
	\[
	d_{\rm CBM}(\alpha_H,\alpha_G)\leq d_{\rm T}(H,G).
	\]
	Indeed, if $\frac{1}{C}W_G\subset W_H$, then $f_{H,G}\leq \ln C$. Similarly, if $\frac{1}{C}W_H\subset W_G$, then $f_{G,H}\leq \ln C$, i.e. $f_{H,G}\geq -\ln C$.
By Proposition~\ref{prop-up-bound}, if $\frac{1}{C}G\leq H\leq CG$, then $\frac{1}{C}W_G\subset W_H$ and $\frac{1}{C}W_H\subset W_G$. Subsequently,  we have $|f_{H,G}|\leq \ln C$, which implies that
$
d_{\rm CBM}(\alpha_H,\alpha_G)\leq \ln C.
$
Then for any $H,G\in\cH^+(W,\lambda)$, we have 
\[
d_{\rm CBM}(\alpha_H,\alpha_G)\leq\inf\left\{\ln C\geq 0\,\left|\,\frac{1}{C}H\leq G\leq C H\right\}\right..
\]
By Lemma~\ref{lemma-reparametrization}, we have $(W_{G},\hat{\lambda})\cong(W_{G\circ \varphi},\hat{\lambda})$ for any $\varphi\in{\rm Symp}(W,\lambda)$, which implies that $(\partial W_{G},\alpha_G)\cong (\partial W_{G\circ \varphi},\alpha_{G\circ\varphi})$. Then we have
\[
d_{\rm CBM}(\alpha_H,\alpha_G)\leq\inf\left\{\ln C\,\left|\,\exists\,\varphi\in {\rm Symp}(W,\lambda),\text{ s.t. } \frac{1}{C}H\leq G\circ\varphi \leq C H\right\}\right.=d_{\rm T}(H,G).
\]
This completes the proof.
\end{proof}

\section{Proof of Theorem~\ref{thm-Thompson-rate}}

Let $a=[\psi_H]$ and $b=[\psi_G]$. Set $s=d_{\rm T,p}(H,G)$. Fix $\varepsilon>0$ and choose $\varphi\in\ham(X,\omega)$ such that, with $C=e^{s+\varepsilon}$, $\frac{1}{C} H\leq G\circ\varphi\leq CH$. Choose a strict contact lift $\widetilde\varphi$ of $\varphi$ and write $c=[\widetilde\varphi]$. Thus $p\circ\widetilde\varphi=\varphi\circ p$, $\widetilde\varphi^*\alpha=\alpha$, and the conjugated class $b_\varphi\coloneqq c^{-1}bc$ is represented by the autonomous Hamiltonian $G\circ\varphi$. Hence $[\psi_{H/C}]\preceq b_\varphi\preceq[\psi_{CH}]$.

We use only integer iteration. If $K$ is autonomous and $m\in\bN$, concatenation of the autonomous path gives $[\psi_{mK}]=[\psi_K]^m$. Let $k_l=\lceil Cl\rceil$. Since $G\circ\varphi\leq CH$ and the order is monotone under multiplication, $b_\varphi^l=[\psi_{l(G\circ\varphi)}]\preceq[\psi_{lCH}]\preceq[\psi_{k_lH}]=a^{k_l}$.

Since $a$ is dominant, we can choose $N\in\bN$ such that $a^N\succeq c$ and $a^N\succeq c^{-1}$. By the bi-invariance of the order, $b^l=c b_\varphi^l c^{-1}\preceq c a^{k_l}c^{-1}\preceq a^{k_l+2N}$. Consequently,
\[
\rho_{\succeq}^+(a,b)\leq\lim_{l\to\infty}\frac{k_l+2N}{l}=C.
\]

For the reverse relative growth rate, compose the Hamiltonian inequality with $\varphi^{-1}$, i.e. $\frac1C H\circ\varphi^{-1}\leq G\leq C H\circ\varphi^{-1}$. The class $a_{\varphi^{-1}}\coloneqq cac^{-1}$ is represented by $H\circ\varphi^{-1}$. Therefore, with the same $k_l$, $a_{\varphi^{-1}}^l=[\psi_{l(H\circ\varphi^{-1})}]\preceq[\psi_{lCG}]\preceq[\psi_{k_lG}]=b^{k_l}$. Choose $M\in\bN$ with $b^M\succeq c$ and $b^M\succeq c^{-1}$. Again by bi-invariance, $a^l=c^{-1}a_{\varphi^{-1}}^l c\preceq c^{-1}b^{k_l}c\preceq b^{k_l+2M}$, so $\rho_{\succeq}^-(a,b)=\rho_{\succeq}^+(b,a)\leq C$.

For dominants the two relative growth rates are positive, and hence
\[
\gamma_{\succeq}(a,b)=\max\{\rho_{\succeq}^+(a,b),\rho_{\succeq}^-(a,b)\}\leq C.
\]
It follows that $d_\gamma(a,b)\leq\ln C=s+\varepsilon$. Letting $\varepsilon\to0$ proves the claim.

\section{Proof of Theorem~\ref{thm-entropy}}
  
The following proposition relates the entropies of two flows that differ by a time reparameterization.
\begin{prop}\label{prop-reparameterization}
Let $(X,d)$ be a compact metric space, and let $\varphi=\{\varphi_t\}_{t\in\bR}$ and $\psi=\{\psi_t\}_{t\in\bR}$ be continuous flows on $X$ without fixed points. 
Assume that $\psi_t$ is a time reparameterization of $\varphi_t$, namely there exists a continuous map
$
\tau:X\times \mathbb R\to \mathbb R
$
such that for every $x\in X$, the map $t\mapsto \tau(x,t)$ is strictly increasing, surjective, and
$
\psi_t(x)=\varphi_{\tau(x,t)}(x)$, for all $x\in X$ and $t\in \mathbb R.
$
Suppose moreover that there exist constants $0<c\le C$ such that for all $x\in X$ and $t\geq 0$,
$
ct\le \tau(x,t)\le Ct$. Then
$
c\,h_{\mathrm{top}}(\varphi_t)\le h_{\mathrm{top}}(\psi_t)\le C\,h_{\mathrm{top}}(\varphi_t).
$
\end{prop}
\begin{proof}
This is a special case of Ohno's theorem on weakly equivalent flows without fixed points; see Theorem~1 in \cite{Oh80}. In the present situation, the bounds $ct\le \tau(x,t)\le Ct$ imply that the additive functional defining the time change has time-one value bounded from below by $c$ and from above by $C$. Therefore $c\,h_{\rm top}(\varphi_1)\le h_{\rm top}(\psi_1)\le C\,h_{\rm top}(\varphi_1)$.
\end{proof}

Before proving Theorem~\ref{thm-entropy}, let us recall the definition of the topological entropy of a flow on some subset. For $T>0$ and $\varepsilon>0$, denote by $N_U(T,\varepsilon)$ the maximal cardinality of a $(T,\varepsilon)$-separated set of $\varphi$ in $U$. 
The topological entropy of $\varphi^t$ on $U$ is defined as
\[
h_{\mathrm{top}}(\varphi, U) 
= \lim_{\varepsilon\to0} \limsup_{T\to\infty} \frac{1}{T} \ln N_U(T,\varepsilon).
\]
If $U_1\subset U_2$, then clearly $h_{\mathrm{top}}(\varphi, U_1)\leq 
h_{\mathrm{top}}(\varphi, U_2)$.

\begin{prop}\label{prop-open-dense}
	If $U$ is a dense subset of $X$, then we  have
	$
	h_{\rm top}(\varphi,U)=h_{\rm top}(\varphi,X).
	$
\end{prop}
\begin{proof}
 Since $U\subset X$, we have $h_{\rm top}(\varphi,U) \le h_{\rm top}(\varphi)$.

For the reverse inequality, fix $\varepsilon>0$.  Because $\varphi^t$ is uniformly continuous on the compact space $[0,T]\times X$, there exists $\delta>0$ such that if $d(x,y)<\delta$, then $\sup_{0\le t\le T}d(\varphi^t(x),\varphi^t(y)) < \frac{\varepsilon}{4}$. Let $E \subset X$ be a maximal $(T,\varepsilon)$-separated set in $X$ (so $|E| = N_X(T,\varepsilon)$).  Since $U$ is dense, for each $x \in E$ we can choose a point $x' \in U$ with $d(x,x') < \delta$.  Set $E' = \{ x' : x \in E \} \subset U$.  We claim that $E'$ is $(T,\varepsilon/2)$-separated. Indeed, take two distinct points $x', y' \in E'$ with corresponding $x,y \in E$.  Because $E$ is $(T,\varepsilon)$-separated, there exists $t\in[0,T]$ such that $d(\varphi^t(x),\varphi^t(y)) \ge \varepsilon$. Then, using the choice of $\delta$,
\begin{align*}
d(\varphi^t(x'),\varphi^t(y')) &\ge d(\varphi^t(x),\varphi^t(y)) 
                          - d(\varphi^t(x),\varphi^t(x')) 
                          - d(\varphi^t(y),\varphi^t(y')) \\
& > \varepsilon - \frac{\varepsilon}{4} - \frac{\varepsilon}{4}
   = \frac{\varepsilon}{2}.
\end{align*}
Thus $E'$ is $(T,\varepsilon/2)$-separated and $|E'| = |E|$. Consequently,
$N_U(T,\varepsilon/2) \ge N_X(T,\varepsilon)$,
which implies that $h_{\rm top}(\varphi,U) \ge h_{\mathrm{top}}(\varphi)$.
\end{proof}

In the relative situation, one considers a factor map $\pi$ between a continuous dynamical system $T\colon X\ra X$ and a second system $S\colon Y \ra Y$ . That is, $\pi\colon X \ra Y$ is continuous and onto and satisfies $\pi\circ T = S\circ \pi$. Then we have the following result.
\begin{theorem}[Theorem~17 in \cite{Bow71}, Proposition 3.1.6 in \cite{Kat80}]\label{thm-rel-entropy}
	Let $(X,d)$ and $(Y,d')$ be compact metric spaces and $T\colon X \ra X$ and $S\colon Y\ra Y$ be continuous. Let $\pi\colon X \ra Y$ be a continuous factor map. Then
	\[
	h_{\rm top}(S)\leq h_{\rm top}(T)\leq h_{\rm top}(S)+\sup_{y\in Y}h_{\rm top}(T,\pi^{-1}y).
	\]
\end{theorem}

Now, let us give the proof of Theorem~\ref{thm-entropy}.
 
\begin{proof}[Proof of Theorem~\ref{thm-entropy}] We fix some notations first. Set 
\[\tau_{\min}\coloneqq \min_{S^1\times W}\tau_{\lambda,H}\quad \mbox{and} \quad \tau_{\max}\coloneqq \max_{S^1\times W}\tau_{\lambda,H}.\]
Denote $X\coloneqq S^1\times W$, $X^\circ\coloneqq S^1\times {\rm Int}(W)$, and $M^\circ\coloneqq \partial W_H\setminus (\partial W\times\{0\})$.

For $(t,w)\in S^1\times W$, define $h(s,t,w)$ by
\[
\int_0^{h(s,t,w)} \tau_{\lambda,H}(t+\sigma,\varphi_H^\sigma(t,w))\,d\sigma=s.
\]
Since $\tau_{\lambda,H}>0$, the function $u\mapsto \int_0^u \tau_{\lambda,H}(t+\sigma,\varphi_H^\sigma(t,w))\,d\sigma$ is strictly increasing and surjective, so $h$ is well-defined. Moreover, for all $(s,t,w)\in [0,\infty)\times S^1\times W$, we have 
\[
\frac{s}{\tau_{\max}}\le h(s,t,w)\le \frac{s}{\tau_{\min}}.
\]

Now define two flows on $X$ by
\[
\Psi^s(t,w)\coloneqq (t+s,\varphi_H^s(t,w))
\]
and
\[
R^s(t,w)\coloneqq (t+h(s,t,w),\varphi_H^{h(s,t,w)}(t,w)).
\]
Then $R^s$ is a time reparameterization of $\Psi^s$, with time-change function $h$. Since both flows have no fixed points, Proposition~\ref{prop-reparameterization} gives
\[
\frac{1}{\tau_{\max}}\,h_{\rm top}(\Psi^1)\le h_{\rm top}(R^1)\le \frac{1}{\tau_{\min}}\,h_{\rm top}(\Psi^1).
\]

Next define
\[
\Phi:X\to \partial W_H\qquad \mbox{by} \qquad
\Phi(t,w)\coloneqq \left(w,\sqrt{\frac{H(t,w)}{\pi}}\,e^{2\pi i t}\right).
\]
This map is continuous and surjective, and its restriction $\Phi|_{X^\circ}:X^\circ\to M^\circ$ is a homeomorphism. By Corollary~\ref{cor-formula}, on $X^\circ$ we have
$
\Phi\circ R^s=\varphi_{\hat\lambda|_{\partial W_H}}^s\circ \Phi.
$
So, $R^1|_{X^\circ}$ is topologically conjugate to $\varphi_{\hat\lambda|_{\partial W_H}}^1|_{M^\circ}$ and 
$
h_{\rm top}(R^1|_{X^\circ})=h_{\rm top}(\varphi_{\hat\lambda|_{\partial W_H}}^1|_{M^\circ}).
$
Since $X^\circ$ is open and dense in $X$, and $M^\circ$ is open and dense in $\partial W_H$, Proposition~\ref{prop-open-dense} implies
$
h_{\rm top}(R^1)=h_{\rm top}(\varphi_{\hat\lambda|_{\partial W_H}}^1).
$

It remains to compute $h_{\rm top}(\Psi^1)$. By definition,
$
\Psi^1(t,w)=(t,\varphi_H^1(t,w)).
$
For each $t\in S^1$, the subset $X_t\coloneqq \{t\}\times W$ is $\Psi^1$-invariant, and $\Psi^1|_{X_t}$ is exactly the time-one map $\varphi_H^1(t,\cdot)$. Hence
$
h_{\rm top}(\Psi^1)\ge \sup_{t\in S^1} h_{\rm top}(\varphi_H^1(t,\cdot)).
$ For the reverse inequality, let $\pi:X\to S^1$ be the projection $\pi(t,w)=t$. Then $\pi$ is a factor map from $\Psi^1$ to $\id_{S^1}$. By Theorem~\ref{thm-rel-entropy},
\[
h_{\rm top}(\Psi^1)\le h_{\rm top}(\id_{S^1})+\sup_{t\in S^1} h_{\rm top}(\Psi^1,X_t).
\]
Since $h_{\rm top}(\id_{S^1})=0$ and $h_{\rm top}(\Psi^1,X_t)=h_{\rm top}(\varphi_H^1(t,\cdot))$, we get
$
h_{\rm top}(\Psi^1)\le \sup_{t\in S^1} h_{\rm top}(\varphi_H^1(t,\cdot)).
$
Therefore
$
h_{\rm top}(\Psi^1)=\sup_{t\in S^1} h_{\rm top}(\varphi_H^1(t,\cdot)).
$

Finally, all the maps $\varphi_H^1(t,\cdot)$ are topologically conjugate. Indeed, if we write $\phi_t\coloneqq \varphi_H^t(0,\cdot)$, then the $1$-periodicity of $H$ implies
$
\varphi_H^1(t,\cdot)=\phi_t\circ \varphi_H^1(0,\cdot)\circ \phi_t^{-1}.
$
Here, $\varphi_H^1(t,\cdot)$ denotes the time-one map from time $t$ to time $t+1$.
Therefore, $h_{\rm top}(\varphi_H^1(t,\cdot))$ is independent of $t$, and thus
$
h_{\rm top}(\Psi^1)=h_{\rm top}(\varphi_H^1),
$
where $\varphi_H^1$ denotes the time-one map starting at time $0$.

Combining this with the previous inequalities, we get
\[
\frac{1}{\tau_{\max}}\,h_{\rm top}(\varphi_H^1)\le h_{\rm top}(\varphi_{\hat\lambda|_{\partial W_H}}^1)\le \frac{1}{\tau_{\min}}\,h_{\rm top}(\varphi_H^1).
\]
Equivalently,
\[
\tau_{\min}\,h_{\rm top}(\varphi_{\hat\lambda|_{\partial W_H}}^1)\le h_{\rm top}(\varphi_H^1)\le \tau_{\max}\,h_{\rm top}(\varphi_{\hat\lambda|_{\partial W_H}}^1).
\]
This completes the proof.
\end{proof}
\begin{proof}[Proof of Corollary~\ref{cor-pos-entropy-2}]
Since $H$ is autonomous, the Reeb flow of $(\alpha_H)_n$ on $(\partial W_H)_n$ is given by $\varphi^t_{(\alpha_H)_n}(w,z)=(\varphi_H^{h(t,w)}(w),e^{2\pi i h(t,w)}z),$ where $z=(z_1,\dots,z_n)\in \bC^n$, and $h:\bR\times W\to \bR$ is determined by
\[
\int_0^{h(t,w)} \tau_{\lambda,H}(\varphi_H^\sigma(w))\,d\sigma=t.
\]
Since $\tau_{\lambda,H}>0$, for every $w\in W$ and $t\ge 0$ we have
\[
\frac{t}{\max_W\tau_{\lambda,H}}\le h(t,w)\le \frac{t}{\min_W\tau_{\lambda,H}}.
\]

Now define $\psi^t(w,z)\coloneqq (\varphi_H^t(w),e^{2\pi it}z).$ We first check that $\psi^t$ is a continuous flow on $(\partial W_H)_n$. If $(w,z)\in (\partial W_H)_n$, then $\pi\sum_{j=1}^n |z_j|^2=H(w)$. Since $H$ is autonomous, $H(\varphi_H^t(w))=H(w)$, and 
\[
\pi\sum_{j=1}^n |e^{2\pi it}z_j|^2=\pi\sum_{j=1}^n |z_j|^2=H(w)=H(\varphi_H^t(w)).
\]
Thus $\psi^t(w,z)\in (\partial W_H)_n$, so $\psi^t$ is well-defined. It is clearly a continuous flow.

Moreover, $\varphi^t_{(\alpha_H)_n}$ is a time reparameterization of $\psi^t$, with time-change function $((w,z),t)\mapsto h(t,w)$. If $z\neq 0$, then $\psi^t(w,z)\neq (w,z)$ for $t\notin \bZ$, so $(w,z)$ is not a fixed point of $\psi$. If $z=0$, then $H(w)=0$, hence $w\in \partial W$, and on $\partial W$ we have $\tau_{\lambda,H}=\lambda(X_H)>0$. Therefore $X_H(w)\neq 0$, so $(w,0)$ is not a fixed point either. Thus $\psi$ has no fixed points. The Reeb flow has no fixed points as well. Then Proposition~\ref{prop-reparameterization} applies and yields
\[
\frac{1}{\max_W\tau_{\lambda,H}}\,h_{\rm top}(\psi^1)\le h_{\rm top}(\varphi^1_{(\alpha_H)_n})\le \frac{1}{\min_W\tau_{\lambda,H}}\,h_{\rm top}(\psi^1).
\]

It remains to prove that $h_{\rm top}(\psi^1)=h_{\rm top}(\varphi_H^1)$. Consider the projection
\[
{\rm pr}_W:(\partial W_H)_n\to W,\qquad {\rm pr}_W(w,z)=w.
\]
Then ${\rm pr}_W$ is a continuous surjective factor map from $\psi^1$ to $\varphi_H^1$, since ${\rm pr}_W\circ \psi^1=\varphi_H^1\circ {\rm pr}_W.$ So, Theorem~\ref{thm-rel-entropy} gives
\[
h_{\rm top}(\varphi_H^1)\le h_{\rm top}(\psi^1)\le h_{\rm top}(\varphi_H^1)+\sup_{w\in W} h_{\rm top}(\psi^1,{\rm pr}_W^{-1}(w)).
\]
For each $w\in W$, the fiber ${\rm pr}_W^{-1}(w)=\bigl\{z\in \bC^n\mid \pi\sum_{j=1}^n |z_j|^2=H(w)\bigr\}$ is either a sphere $S^{2n-1}$ of radius $\sqrt{H(w)/\pi}$ if $H(w)>0$, or a point if $H(w)=0$. Moreover, $\psi^1\bigl({\rm pr}_W^{-1}(w)\bigr)={\rm pr}_W^{-1}\bigl(\varphi_H^1(w)\bigr),$ and, unless $w$ is fixed by $\varphi_H^1$, this is not a self-map of the fiber. Nevertheless, the relative entropy on each fiber vanishes, as can be seen directly from Bowen separated sets. Fix a metric $d_W$ on $W$ and equip $(\partial W_H)_n\subset W\times\bC^n$ with the restriction of the product metric  
\[
d\bigl((w,z),(w',z')\bigr)\coloneqq d_W(w,w')+\| z-z'\|.
\] 
For $N\in\bN$, let $d_N(x,y)\coloneqq\max_{0\leq k<N}d\bigl((\psi^1)^k(x),(\psi^1)^k(y)\bigr)$ be the corresponding Bowen metric. Since $(\psi^1)^k(w,z)=\bigl((\varphi_H^1)^k(w),z\bigr)$ for every $k\in\bN$, any two points $(w,z),(w,z')\in{\rm pr}_W^{-1}(w)$ satisfy $d_N\bigl((w,z),(w,z')\bigr)=\| z-z'\|$, independently of $N$. Thus, for every $\varepsilon>0$, the cardinality of an $(N,\varepsilon)$-separated subset of ${\rm pr}_W^{-1}(w)$ is bounded by the maximal cardinality of an $\varepsilon$-separated subset of this fixed compact sphere (or point). This bound is finite and independent of $N$. Consequently, $h_{\rm top}(\psi^1,{\rm pr}_W^{-1}(w))=0$ for every $w\in W$, and hence $h_{\rm top}(\psi^1)=h_{\rm top}(\varphi_H^1).$

Substituting this into the previous estimate, we obtain
\[
\frac{1}{\max_W\tau_{\lambda,H}}\,h_{\rm top}(\varphi_H^1)\le h_{\rm top}(\varphi^1_{(\alpha_H)_n})\le \frac{1}{\min_W\tau_{\lambda,H}}\,h_{\rm top}(\varphi_H^1),
\]
which is equivalent to
\[
(\min_W\tau_{\lambda,H})h_{\rm top}(\varphi^1_{(\alpha_H)_n})\le h_{\rm top}(\varphi_H^1)\le (\max_W\tau_{\lambda,H})h_{\rm top}(\varphi^1_{(\alpha_H)_n}).
\]

For the final statement, assume $h_{\rm top}(\varphi^1_{\lambda|_{\partial W}})>0$. By Corollary~\ref{cor-pos-entropy}, we have $h_{\rm top}(\varphi_H^1)>0$, and then
\[
h_{\rm top}(\varphi^1_{(\alpha_H)_n})\ge \frac{1}{\max_W\tau_{\lambda,H}}\,h_{\rm top}(\varphi_H^1)>0.
\]
This completes the proof. 
\end{proof}

\section{Proof of Theorem~\ref{thm-barcode-extension-c0}}

We first show that every smooth Hamiltonian in $\cH(W,\lambda)$ belongs to the $C^0$-strong class.
\begin{prop}\label{prop-smooth-into-strong}
For every $H\in \cH(W,\lambda)$, one has $H\in \cH^0(W,\lambda)$.
\end{prop}

\begin{proof}
Fix $H\in \cH(W,\lambda)$. Since $H|_{S^1\times \partial W}=0$ and $H>0$ on $S^1\times {\rm Int}(W)$, it remains to prove that for every $\varepsilon>0$,
\[
c_H(\varepsilon)\coloneqq \min_{(t,w)\in S^1\times W}(H(t,\sL_\lambda^{-\varepsilon}(w))-e^{-\varepsilon}H(t,w))>0.
\]
Indeed, fix $(t,w)\in S^1\times W$ and define $f_{t,w}(s)\coloneqq e^sH(t,\sL_\lambda^{-s}(w))$ for $s\in [0,\varepsilon]$. Since $Y_\lambda$ is the Liouville vector field, we have
\[
\frac{d}{ds}f_{t,w}(s)=e^s(H_t-d_wH_t(Y_\lambda))(\sL_\lambda^{-s}(w)).
\]
With our convention $\iota_{X_{H_t}}d\lambda=d_wH_t$, one has $d_wH_t(Y_\lambda)=d\lambda(X_{H_t},Y_\lambda)=-\lambda(X_{H_t})$. Therefore
\[
\frac{d}{ds}f_{t,w}(s)=e^s(\lambda(X_{H_t})+H_t)(\sL_\lambda^{-s}(w))=e^s\tau_{\lambda,H}(t,\sL_\lambda^{-s}(w)).
\]
It follows that
\[
H(t,\sL_\lambda^{-\varepsilon}(w))-e^{-\varepsilon}H(t,w)=e^{-\varepsilon}\int_0^\varepsilon e^s\tau_{\lambda,H}(t,\sL_\lambda^{-s}(w))\,ds.
\]
Since $\tau_{\lambda,H}>0$ is continuous on the compact set $S^1\times W$, the right-hand side is bounded below by
\[
e^{-\varepsilon}\int_0^\varepsilon e^s\,ds\cdot \min_{S^1\times W}\tau_{\lambda,H}>0.
\]
Taking the minimum over $(t,w)\in S^1\times W$, we obtain $c_H(\varepsilon)>0$. Thus $H\in \cH^0(W,\lambda)$.
\end{proof}

The next proposition shows that the smooth class $\cH(W,\lambda)$ is $C^0$-dense in $\cH^0(W,\lambda)$.

\begin{prop}\label{prop-density-smooth}
The subset $\cH(W,\lambda)$ is dense in $\cH^0(W,\lambda)$ with respect to the $C^0$-topology.
\end{prop}

We divide the proof of Proposition~\ref{prop-density-smooth} into two lemmas.
\begin{lemma}\label{lem-liouville-averaging}
Let $H\in\cH^0(W,\lambda)$. For every \(\eta>0\), there exist a Liouville collar $[-\rho,0]\times\partial W\subset W$ with \(0<\rho<\log 2\), and a smooth function $A:S^1\times W\to(0,\infty)$ such that $\tau_{\lambda,A}>0, \|A-H\|_{C^0}<\eta$ and $\|A|_{[-\rho,0]\times\partial W}\|_{C^0}<\eta$.
\end{lemma}
\begin{lemma}\label{lem-boundary-straightening}
Let \(A\in C^\infty(S^1\times W)\) satisfy $A>0$ and $\tau_{\lambda,A}>0$. Fix a Liouville collar $[-\rho,0]\times\partial W\subset W$. Then there exists \(K\in \cH(W,\lambda)\) such that \(K=A\) outside the collar and $\|K-A\|_{C^0}\le e^\rho\|A|_{[-\rho,0]\times\partial W}\|_{C^0}$.
\end{lemma}
Assuming these two lemmas, the proposition follows immediately.

\begin{proof}[Proof of Proposition~\ref{prop-density-smooth}]
Fix \(H\in\cH^0(W,\lambda)\) and \(\varepsilon>0\). Apply Lemma~\ref{lem-liouville-averaging} with \(\eta=\varepsilon/4\). We obtain a Liouville collar \( [-\rho,0]\times\partial W\subset W\), where \(0<\rho<\log 2\), and a smooth positive function \(A\) such that \(\tau_{\lambda,A}>0\), \(\|A-H\|_{C^0}<\varepsilon/4\), and $\|A|_{[-\rho,0]\times\partial W}\|_{C^0}<\frac{\varepsilon}{4}$.

Apply Lemma~\ref{lem-boundary-straightening} to \(A\). It gives \(K\in\cH(W,\lambda)\) satisfying
\[
\|K-A\|_{C^0}\le e^\rho\|A|_{[-\rho,0]\times\partial W}\|_{C^0}<\frac{\varepsilon}{2},
\]
where we have used \(e^\rho<2\). Therefore,
\[
\|K-H\|_{C^0}\le\|K-A\|_{C^0}+\|A-H\|_{C^0}<\frac{3\varepsilon}{4}<\varepsilon.
\]
Thus \(\cH(W,\lambda)\) is \(C^0\)-dense in \(\cH^0(W,\lambda)\).
\end{proof}

\begin{proof}[Proof of Lemma~\ref{lem-liouville-averaging}]
Choose a Liouville collar \((-\rho_0,0]\times\partial W\subset W\), in which \(Y_\lambda=\partial_r\). Since \(H\) is continuous and \(H|_{S^1\times\partial W}=0\), we may choose \(0<\rho<\min\{\rho_0/2,\log 2\}\) such that
\[
\sup_{0\le s\le\rho}\left\|e^sH(t,\sL_\lambda^{-s}(w))-H(t,w)\right\|_{C^0(S^1\times W)}<\frac{\eta}{2}\,\,\text{ and }\,\, e^\rho\|H\|_{C^0(S^1\times[-2\rho,0]\times\partial W)}<\frac{\eta}{2}.
\]
Since \(H\in\cH^0(W,\lambda)\), we have \(c_H(\rho)>0\). Choose \(\delta>0\) such that \(e^\rho\delta<\eta/2\) and $(1+e^\rho)\delta<\frac12e^\rho c_H(\rho)$. Choose a smooth positive function \(F:S^1\times W\to(0,\infty)\) satisfying \(\|F-H\|_{C^0}<\delta\), and define
\[
A(t,w)\coloneqq\frac1\rho\int_0^\rho e^sF(t,\sL_\lambda^{-s}(w))\,ds.
\]
Then \(A\) is smooth and strictly positive.

 Recall that \(\tau_{\lambda,A}=A-d_wA(Y_\lambda)\). Since $\frac{d}{ds}\sL_\lambda^{-s}(w)=-Y_\lambda(\sL_\lambda^{-s}(w))$, the chain rule gives $\frac{d}{ds}\left(e^sF(t,\sL_\lambda^{-s}(w))\right)=e^s\left(F-d_wF(Y_\lambda)\right)(t,\sL_\lambda^{-s}(w))$. Then we obtain
\[
\tau_{\lambda,A}(t,w)=\frac1\rho\int_0^\rho\frac{d}{ds}\left(e^sF(t,\sL_\lambda^{-s}(w))\right)\,ds=\frac{e^\rho F(t,\sL_\lambda^{-\rho}(w))-F(t,w)}{\rho}.
\]
By the definition of \(c_H(\rho)\), $e^\rho H(t,\sL_\lambda^{-\rho}(w))-H(t,w)\ge e^\rho c_H(\rho)$. It follows that $\tau_{\lambda,A}(t,w)>0$.

We next show that \(A\) is \(C^0\)-close to \(H\). Adding and subtracting \(e^sH(t,\sL_\lambda^{-s}(w))\) inside the integral gives
\[
\|A-H\|_{C^0}\le e^\rho\|F-H\|_{C^0}+\sup_{0\le s\le\rho}\left\|e^sH(t,\sL_\lambda^{-s}(w))-H(t,w)\right\|_{C^0}<\eta.
\]
It remains to estimate \(A\) near the boundary. If \((r,x)\in[-\rho,0]\times\partial W\) and \(0\le s\le\rho\), then $\sL_\lambda^{-s}(r,x)=(r-s,x)\in[-2\rho,0]\times\partial W$. Using \(\|F-H\|_{C^0}<\delta\), we obtain
\[
\|A|_{[-\rho,0]\times\partial W}\|_{C^0}\le e^\rho\left(\|H\|_{C^0(S^1\times[-2\rho,0]\times\partial W)}+\delta\right)<\eta.
\]
Thus \(A\) has all the required properties.
\end{proof}
Before we prove Lemma~\ref{lem-boundary-straightening}, we need the following lemma.
\begin{lemma}
\label{lem-monotone-collar-interpolation}
Let \(\Sigma\) be a compact manifold and let $u:S^1\times[-\rho,0]\times\Sigma\to(0,\infty)$ be smooth with \(\partial_ru<0\). Then there exists a smooth function $\widetilde u:S^1\times[-\rho,0]\times\Sigma\to[0,\infty)$ such that:

\begin{enumerate}
\item \(\widetilde u=u\) near \(r=-\rho\);
\item \(\widetilde u(t,0,x)=0\) and \(\widetilde u(t,r,x)>0\) for \(r<0\);
\item \(\partial_r\widetilde u<0\);
\item \(\widetilde u\) is independent of \(t\) near \(r=0\);
\item $0\le\widetilde u(t,r,x)\le\max_{(t,x)\in S^1\times\Sigma}u(t,-\rho,x)$.
\end{enumerate}
\end{lemma}
\begin{proof}
Since \(u>0\) and \(S^1\times\Sigma\) is compact, the function \(u(t,-\rho,x)\) has a uniform positive lower bound. We first choose a smooth positive function \(p_0\) on \(S^1\times[-\rho,0]\times\Sigma\) such that \(p_0=-\partial_ru\) near \(r=-\rho\), \(p_0\) is independent of \(t\) near \(r=0\), and
\[
\int_{-\rho}^0p_0(t,s,x)\,ds<u(t,-\rho,x)
\]
for every \((t,x)\).

To see that this is possible, retain \(-\partial_ru\) only on a sufficiently small neighborhood of \(r=-\rho\), choose a sufficiently small positive \(t\)-independent function near \(r=0\), and join the two functions smoothly and positively across the middle of the collar. By making the two end neighborhoods and the intermediate function sufficiently small, the total integral of \(p_0\) can be made uniformly smaller than \(u(t,-\rho,x)\).

Choose a fixed nonnegative smooth function \(\beta\), supported in the interior of \((-\rho,0)\), such that \(\int_{-\rho}^0\beta(s)\,ds=1\). Then define
\[
p(t,r,x)\coloneqq p_0(t,r,x)+\left(u(t,-\rho,x)-\int_{-\rho}^0p_0(t,s,x)\,ds\right)\beta(r).
\]
Since \(\beta\) is supported away from the two ends of the collar, the function \(p\) still satisfies \(p=-\partial_ru\) near \(r=-\rho\) and is independent of \(t\) near \(r=0\). Moreover, \(p>0\) and $\int_{-\rho}^0p(t,s,x)\,ds=u(t,-\rho,x)$.

Now define $\widetilde u(t,r,x)\coloneqq\int_r^0p(t,s,x)\,ds$. Because \(p>0\), the function \(\widetilde u\) is strictly positive for \(r<0\), vanishes at \(r=0\), and satisfies \(\partial_r\widetilde u=-p<0\). At the inner end of the collar, the integral identity gives \(\widetilde u(t,-\rho,x)=u(t,-\rho,x)\). Since \(p=-\partial_ru\) near \(r=-\rho\), we also have \(\partial_r\widetilde u=\partial_ru\) there. So, \(\widetilde u=u\) near \(r=-\rho\). Near \(r=0\), the function \(p\) is independent of \(t\), and therefore so is \(\widetilde u\). Finally, positivity of \(p\) gives $0\le\widetilde u(t,r,x)\le\int_{-\rho}^0p(t,s,x)\,ds=u(t,-\rho,x)$, which proves the required estimate.
\end{proof}

\begin{proof}[Proof of Lemma~\ref{lem-boundary-straightening}]
Write the fixed Liouville collar as \( [-\rho,0]\times\partial W\), where \(Y_\lambda=\partial_r\). Define \(u(t,r,x)\coloneqq e^{-r}A(t,r,x)\). Since \(\tau_{\lambda,A}=A-\partial_rA>0\), we have $\partial_ru=-e^{-r}\tau_{\lambda,A}<0$. Apply Lemma~\ref{lem-monotone-collar-interpolation} to \(u\), with \(\Sigma=\partial W\), and denote the resulting function by \(\widetilde u\). Define \(K(t,r,x)\coloneqq e^r\widetilde u(t,r,x)\) on the collar and set \(K\coloneqq A\) outside the collar.

Since \(\widetilde u=u\) near \(r=-\rho\), the two definitions agree near the inner boundary of the collar, and hence \(K\) is smooth. Since \(\widetilde u(t,0,x)=0\), the function \(K\) vanishes on \(S^1\times\partial W\). Moreover, \(K\) is independent of \(t\) near \(\partial W\) and is strictly positive on \(S^1\times\operatorname{Int}(W)\). On the collar, $\tau_{\lambda,K}=K-\partial_rK=-e^r\partial_r\widetilde u>0$. Outside the collar, \(K=A\), so \(\tau_{\lambda,K}=\tau_{\lambda,A}>0\). Therefore \(K\in\cH(W,\lambda)\).

Finally, Lemma~\ref{lem-monotone-collar-interpolation} gives
\[
0\le\widetilde u(t,r,x)\le\max_{(t,x)\in S^1\times\partial W}u(t,-\rho,x)\le e^\rho\|A|_{[-\rho,0]\times\partial W}\|_{C^0}.
\]
Since \(e^r\le1\), the same upper bound holds for \(K\) on the collar. Because \(A>0\), we conclude that $\|K-A\|_{C^0}\le e^\rho\|A|_{[-\rho,0]\times\partial W}\|_{C^0}$. This completes the proof.
\end{proof}
 
For $H,G\in \cH^0(W,\lambda)$, define
\[
\delta_\lambda(H,G)\coloneqq \inf\Bigl\{\varepsilon\ge 0 \,\Bigm|\, \sL_{\hat\lambda}^{-\varepsilon}(W_H)\subset W_G \text{ and } \sL_{\hat\lambda}^{-\varepsilon}(W_G)\subset W_H\Bigr\}.
\]
\begin{proof}[Proof of Theorem~\ref{thm-barcode-extension-c0}]
Let $H,G\in \cH(W,\lambda)$, and choose $\varepsilon>\delta_\lambda(H,G)$. By the definition of $\delta_\lambda$, the Liouville flow satisfies $\sL_{\hat\lambda}^{-\varepsilon}(W_H)\subset W_G$ and $\sL_{\hat\lambda}^{-\varepsilon}(W_G)\subset W_H$. These inclusions are admissible for the symplectic Banach--Mazur distance of the smooth Liouville domains $(W_H,\hat\lambda)$ and $(W_G,\hat\lambda)$. Hence, $d_{\rm SBM}((W_H,\hat\lambda),(W_G,\hat\lambda))\le \varepsilon$. Proposition~\ref{prop-interleaving} gives $d_{\rm bot}(\cB(H),\cB(G))\le \varepsilon$, and letting $\varepsilon\ra \delta_\lambda(H,G)$ yields
\begin{equation}\label{eq:smooth-delta-lipschitz}
d_{\rm bot}(\cB(H),\cB(G))\le \delta_\lambda(H,G).
\end{equation}

We now define the barcode for $H\in \cH^0(W,\lambda)$. Choose smooth Hamiltonians $H_k\in \cH(W,\lambda)$ with $H_k\to H$ in $C^0$. We show that $\cB(H_k)$ has a limit in bottleneck distance. Fix $\varepsilon>0$. Since $c_H(\varepsilon/2)>0$, for all sufficiently large $k$ one has $\|H_k-H\|_{C^0}<c_H(\varepsilon/2)$. For such $k$, we claim that $\delta_\lambda(H_k,H)\le \varepsilon/2$. Indeed, if $(w,z)\in W_H$, with $z=re^{2\pi it}$, then $\pi|z|^2\le H(t,w)$, and therefore 
\begin{align*}
\pi|e^{-\varepsilon/4}z|^2&=e^{-\varepsilon/2}\pi|z|^2\le e^{-\varepsilon/2}H(t,w)\\
& <H(t,\sL_\lambda^{-\varepsilon/2}(w))-c_H(\varepsilon/2)\le H_k(t,\sL_\lambda^{-\varepsilon/2}(w)).
\end{align*}
 Thus $\sL_{\hat\lambda}^{-\varepsilon/2}(W_H)\subset W_{H_k}$. Conversely, if $(w,z)\in W_{H_k}$, then $\pi|z|^2\le H_k(t,w)\le H(t,w)+c_H(\varepsilon/2)$, so $e^{-\varepsilon/2}\pi|z|^2<e^{-\varepsilon/2}H(t,w)+c_H(\varepsilon/2)\le H(t,\sL_\lambda^{-\varepsilon/2}(w))$. Therefore, $\sL_{\hat\lambda}^{-\varepsilon/2}(W_{H_k})\subset W_H$, as claimed. Then it follows from the triangle inequality for $\delta_\lambda$ that, for all sufficiently large $k,\ell$, $\delta_\lambda(H_k,H_\ell)\le \delta_\lambda(H_k,H)+\delta_\lambda(H,H_\ell)\le \varepsilon$. Applying \eqref{eq:smooth-delta-lipschitz} to the smooth pair $(H_k,H_\ell)$, we get 
\[
d_{\rm bot}(\cB(H_k),\cB(H_\ell))\le \varepsilon.
\] 
Hence $\{\cB(H_k)\}$ is Cauchy, and we define $\cB(H)\coloneqq \lim_{k\to\infty}\cB(H_k)$.

Note that this definition does not depend on the chosen smooth approximation. Indeed, if $G_k\in \cH(W,\lambda)$ is another sequence with $G_k\to H$ in $C^0$, then the previous paragraph gives $\delta_\lambda(H_k,H)\le \varepsilon/2$ and $\delta_\lambda(G_k,H)\le \varepsilon/2$ for all large $k$. Hence $\delta_\lambda(H_k,G_k)\le \varepsilon$, and \eqref{eq:smooth-delta-lipschitz} implies $d_{\rm bot}(\cB(H_k),\cB(G_k))\le \varepsilon$. The two approximating sequences therefore have the same bottleneck limit. Thus $\cB(H)$ is well defined. Since smooth Hamiltonians in $\cH(W,\lambda)$ are $C^0$-dense in $\cH^0(W,\lambda)$ by Proposition \ref{prop-density-smooth}, this extension is unique.

Moreover, we claim that the estimate \eqref{eq:smooth-delta-lipschitz} passes to the $C^0$-strong class. Indeed, let $H,G\in \cH^0(W,\lambda)$, and choose smooth approximations $H_k\to H$ and $G_k\to G$ in $C^0$. Given $\eta>0$, the preceding argument gives $\delta_\lambda(H_k,H)<\eta$ and $\delta_\lambda(G_k,G)<\eta$ for all large $k$. Thus $\delta_\lambda(H_k,G_k)\le \delta_\lambda(H,G)+2\eta$. Using \eqref{eq:smooth-delta-lipschitz} for $H_k$ and $G_k$, and then passing to the bottleneck limit, gives $d_{\rm bot}(\cB(H),\cB(G))\le \delta_\lambda(H,G)+2\eta$. Since $\eta>0$ is arbitrary, we have
\begin{equation}\label{eq:c0-delta-lipschitz}
d_{\rm bot}(\cB(H),\cB(G))\le \delta_\lambda(H,G).
\end{equation}
In particular, the extended barcode map is $\delta_\lambda$-continuous. 

Finally, for smooth $H,G \in \cH(W,\lambda)$, the standard collar argument as in Proposition \ref{prop-density-smooth} identifies $d_{\rm SBM}(W_H,W_G)$ in Definition \ref{def-open-sbm} with the symplectic Banach--Mazur distance between the Liouville domains $(W_H,\hat\lambda)$ and $(W_G,\hat\lambda)$ defined at the beginning of Section \ref{sec-intro}. Then Proposition~\ref{prop-interleaving} gives
\begin{equation}\label{eq:smooth-sbm-lipschitz}
d_{\rm bot}(\cB(H),\cB(G))\le d_{\rm SBM}(W_H,W_G).
\end{equation}
Now let $H,G\in \cH^0(W,\lambda)$, and choose smooth $H_\eta,G_\eta\in \cH(W,\lambda)$ such  that $\delta_\lambda(H_\eta,H)<\eta$ and $\delta_\lambda(G_\eta,G)<\eta$. Then, by triangle inequalities of $d_{\rm SBM}$ and $d_{\rm bot}$, we have
\begin{align*}
d_{\rm bot}(\cB(H),\cB(G)) &\leq d_{\rm bot}(\cB(H), \cB(H_\eta)) + d_{\rm bot}(\cB(H_{\eta}),\cB(G_{\eta})) + d_{\rm bot}(\cB(G_\eta),\cB(G)) \\
& \leq d_{\rm SBM}(W_{H_{\eta}}, W_{G_{\eta}}) + 2\eta\\
& \leq d_{\rm SBM}(W_{H_{\eta}}, W_H) + d_{\rm SBM}(W_{H}, W_{G}) + d_{\rm SBM}(W_{G_{\eta}}, W_{G}) \\
& \leq d_{\rm SBM}(W_{H}, W_{G}) + 4\eta.
\end{align*}
Here, the second inequality comes from (\ref{eq:smooth-sbm-lipschitz}). Letting $\eta\to 0$ proves the desired inequality $d_{\rm bot}(\cB(H),\cB(G))\le d_{\rm SBM}(W_H,W_G)$ for $H, G \in \cH^0(W,\lambda)$. 
\end{proof}
\section{Proof of Theorem~\ref{thm-capacity} }
\begin{proof}[Proof of Lemma~\ref{lem:Gwidth-continuity-c0-time-dependent}]
Fix $H\in \cH^0(W,\lambda)$. Let $\sL_{\hat\lambda}^s$ denote the Liouville flow of $\hat\lambda$ on $W\times \bC$. Then $\sL_{\hat\lambda}^s(w,z)=(\sL_\lambda^s(w),e^{s/2}z)$, and $w_G(\sL_{\hat\lambda}^s(U))=e^s\,w_G(U)$ for every open subset $U\subset W\times \bC$. Fix $\varepsilon>0$. Recall that 
\[
c_H(\varepsilon)=\min_{(t,w)\in S^1\times W}(H(t,\sL_\lambda^{-\varepsilon}(w))-e^{-\varepsilon}H(t,w))>0.
\]
Choose $0<\delta<c_H(\varepsilon)$, and let $G\in \cH^0(W,\lambda)$ satisfy $\|G-H\|_{C^0}<\delta$. We claim that
$
\sL_{\hat\lambda}^{-\varepsilon}(U_H)\subset U_G\subset\sL_{\hat\lambda}^{\varepsilon}(U_H).
$
Indeed, let $(w,z)\in U_H$, then $\pi |z|^2<H(t,w)$. Applying $\sL_{\hat\lambda}^{-\varepsilon}$, we get $\sL_{\hat\lambda}^{-\varepsilon}(w,z)=(\sL_\lambda^{-\varepsilon}(w),e^{-\varepsilon/2}z)$ and hence $\pi |e^{-\varepsilon/2}z|^2=e^{-\varepsilon}\pi |z|^2<e^{-\varepsilon}H(t,w)$. By the definition of  $c_H(\varepsilon)$ and the choice of $\delta$,  
\[
e^{-\varepsilon}H(t,w)<H(t,\sL_\lambda^{-\varepsilon}(w))-\delta<G(t,\sL_\lambda^{-\varepsilon}(w)).
\] 
Therefore $\sL_{\hat\lambda}^{-\varepsilon}(w,z)\in U_G$ and hence $\sL_{\hat\lambda}^{-\varepsilon}(U_H)\subset U_G$. Conversely, it suffices to prove $\sL_{\hat\lambda}^{-\varepsilon}(U_G)\subset U_H$. Let $(w,z)\in U_G$, then $\pi |z|^2<G(t,w)\le H(t,w)+\delta$. Therefore 
\[
\pi |e^{-\varepsilon/2}z|^2=e^{-\varepsilon}\pi |z|^2<e^{-\varepsilon}H(t,w)+e^{-\varepsilon}\delta<e^{-\varepsilon}H(t,w)+\delta.
\]
 Since $\delta<c_H(\varepsilon)$, we have $\pi |e^{-\varepsilon/2}z|^2<e^{-\varepsilon}H(t,w)+\delta<H(t,\sL_\lambda^{-\varepsilon}(w)),$ so $\sL_{\hat\lambda}^{-\varepsilon}(w,z)\in U_H$. Thus $\sL_{\hat\lambda}^{-\varepsilon}(U_G)\subset U_H$, and hence $U_G\subset \sL_{\hat\lambda}^{\varepsilon}(U_H)$.

By monotonicity and conformality of the Gromov width, we have \[ e^{-\varepsilon}w_G(U_H)\le w_G(U_G)\le e^{\varepsilon}w_G(U_H). \]
Now, let $H_\nu\to H$ in $\cH^0(W,\lambda)$ with respect to the $C^0$-topology. For all sufficiently large $\nu$, we have $\|H_\nu-H\|_{C^0}<\delta$, and hence $e^{-\varepsilon}w_G(U_H)\le w_G(U_{H_\nu})\le e^{\varepsilon}w_G(U_H)$. Since $\varepsilon>0$ is arbitrary, it follows that $w_G(U_{H_\nu})\to w_G(U_H)$, that is, $H\mapsto w_G(W_H,{\rm d}\hat\lambda)$ is continuous.
\end{proof}

Now, we are ready to give the proof of Theorem \ref{thm-capacity}. 
\begin{proof}[Proof of Theorem~\ref{thm-capacity} ]
We first prove the upper bound in (i). Since $H(t,w)\le M$ on $S^1\times W$, we have $W_{aH}\subset W\times \overline{B}^{\,2}(aM)$, where $\overline{B}^{\,2}(aM)\subset\bC$ is the closed two-dimensional standard ball. Therefore, by Theorem~1.4 in \cite{LM95a}, $w_G(W_{aH})\le aM$ and then $\limsup_{a\to0^+}\frac{w_G(W_{aH})}{a}\le M$.

We now prove the lower bound in (i). Fix $\varepsilon>0$ and choose $w_0\in {\rm Int}(W)$ such that $\overline H(w_0)>\overline M-\varepsilon$. Since $H(t,w_0)>0$ for all $t\in S^1$,  we may choose a number $\mu>0$ such that $0<\mu<\varepsilon\text{ and }H(t,w_0)-\mu>0\text{ for all }t\in S^1$. By continuity of $H$, after shrinking to a sufficiently small Darboux chart around $w_0$,  there exist $\rho>0$ and a symplectic embedding $\phi:B^{2n}(\rho)\hookrightarrow {\rm Int}(W)$ such that $H(t,\phi(x))\ge H(t,w_0)-\mu\text{ for all }(t,x)\in S^1\times B^{2n}(\rho).$ Then, for $a>0$, denote a region in $\bC$ as follows, 
\[ \Omega_{a,\mu}\coloneqq\left\{re^{2\pi it}\in \bC\;\middle|\;\pi r^2<a(H(t,w_0)-\mu)\right\}.\]
Then by the defining properties, we have $\phi(B^{2n}(\rho))\times \Omega_{a,\mu}\subset W_{aH}$. 

Moreover, ${\rm Area}(\Omega_{a,\mu})=a\int_0^1 (H(t,w_0)-\mu)\,dt=a(\overline H(w_0)-\mu)$. For brevity, set $A_\mu:=\overline H(w_0)-\mu (>0)$. Then for any $0<\delta<A_\mu$, since $\Omega_{a,\mu}\subset \bC$ is simply connected, a standard two-dimensional Moser argument yields a symplectic embedding $B^2(A_\mu-\delta)\hookrightarrow \Omega_{1,\mu}$, so $B^2(a(A_\mu-\delta))\hookrightarrow \Omega_{a,\mu}$. This implies that 
\[ B^{2n}(\rho)\times B^2(a(A_\mu-\delta)) \hookrightarrow W_{aH}.\]
Then since $B^{2n+2}(\min\{\rho,\ a(A_\mu-\delta)\})\subset B^{2n}(\rho)\times B^2(a(A_\mu-\delta))$, we get $w_G(W_{aH})\ge \min\{\rho,\ a(A_\mu-\delta)\}$. For sufficiently small $a>0$, the second term obtains the minimum. Hence, $w_G(W_{aH})\ge a(A_\mu-\delta)$. Dividing by $a$ and taking $\liminf$, we obtain 
\[
\liminf_{a\to0^+}\frac{w_G(W_{aH})}{a}\ge A_\mu-\delta=\overline H(w_0)-\mu-\delta.
\] 
Letting $\delta\to0^+$, and then using $\mu<\varepsilon$ together with $\overline H(w_0)>\overline M-\varepsilon$, gives 
\[
\liminf_{a\to0^+}\frac{w_G(W_{aH})}{a}\ge \overline M-2\varepsilon.
\] 
Since $\varepsilon>0$ was arbitrary, it follows that $\liminf_{a\to0^+}\frac{w_G(W_{aH})}{a} \ge \overline M$. This proves (i).

\medskip

We now prove (ii). Let $(w,z)\in {\rm Int}(W)\times \bC$. Write $z=re^{2\pi it}$. Since $H(t,w)>0$, choosing $a>\frac{\pi |z|^2}{H(t,w)}$ gives $(w,z)\in U_{aH}$. Thus $\bigcup_{a>0}U_{aH}={\rm Int}(W)\times \bC$. The family $\{U_{aH}\}_{a>0}$ is clearly increasing in $a$. Therefore, by monotonicity of the Gromov width, $\lim_{a\to\infty} w_G(U_{aH})=\sup_{a>0} w_G(U_{aH})\le w_G({\rm Int}(W)\times\mathbb C)$. It remains to prove the opposite inequality. 

Let $c<w_G({\rm Int}(W)\times\mathbb C)$. Choose $c'<c$ and a symplectic embedding $\Psi:B^{2n+2}(c)\hookrightarrow {\rm Int}(W)\times\mathbb C$. The image $K=\Psi(\overline{B}^{\,2n+2}(c'))$ is compact in ${\rm Int}(W)\times\mathbb C$.  Since $\{U_{aH}\}_{a>0}$ is an increasing open cover of ${\rm Int}(W)\times\mathbb C$, there exists $a_0>0$ such that $K\subset U_{aH}\text{ for all }a\ge a_0$. It follows that $w_G(U_{aH})\ge c' \text{ for all }a\ge a_0$. So, $\lim_{a\to\infty}w_G(U_{aH})\ge c'$. Letting $c'\to c$, and then $c\to w_G(\operatorname{Int}(W)\times \mathbb C)$, we obtain $\lim_{a\to\infty}w_G(U_{aH})\ge w_G({\rm Int}(W)\times\mathbb C)$. Therefore $\lim_{a\to\infty}w_G(W_{aH})=w_G({\rm Int}(W)\times\mathbb C) = w_G(W\times\mathbb C)$.
\end{proof}

In what follows, we will give the proofs of Proposition \ref{prop-wG-finite} and Proposition \ref{prop-lift-capacity}.  
 
\begin{proof}[Proof of Proposition~\ref{prop-wG-finite}] We will give the proof for each individual case. 

\medskip

\noindent For (i): We first prove the lower bound. Let $c<w_G(M,\omega)$, and then there exists a symplectic embedding $\psi\colon B^{2n}(r)\hookrightarrow M$ with $r>c$. Taking the product with the standard inclusion $B^2(r)\hookrightarrow \mathbb C$, we obtain a symplectic embedding $B^{2n}(r)\times B^2(r)\hookrightarrow M\times \mathbb C$. Since $B^{2n+2}(r)\subset B^{2n}(r)\times B^2(r),$ it follows that $c<w_G(M\times \mathbb C,\omega\oplus \omega_{\mathrm{std}})$. As $c<w_G(M,\omega)$ was arbitrary, we conclude that $w_G(M,\omega)\le w_G(M\times \mathbb C,\omega\oplus \omega_{\mathrm{std}}).$

For the reverse inequality, let $r>0$ and suppose that there exists a symplectic embedding $\varphi\colon M\hookrightarrow Z^{2n}(r)$. Then $\varphi\times \id_{\mathbb C}\colon M\times \mathbb C \hookrightarrow Z^{2n}(r)\times \mathbb C=Z^{2n+2}(r)$ is a symplectic embedding, so $c_Z(M\times \mathbb C,\omega\oplus \omega_{\mathrm{std}})\le c_Z(M,\omega)$. Combining this with the general inequality $w_G\le c_Z$, we obtain
\[ w_G(M\times \mathbb C,\omega\oplus \omega_{\mathrm{std}})\le c_Z(M\times \mathbb C,\omega\oplus \omega_{\mathrm{std}})\le c_Z(M,\omega)=w_G(M,\omega).\]
This completes the proof of case (i). 

\medskip

\noindent For (ii), set $A\coloneqq \int_W\omega$. We first recall a simple computation of the Gromov width of symplectic manifolds in the form of $S^2(a)\times \bC$. Let $S^2(a)$ denote the two-sphere equipped with an area form of total area $a$. Since $S^2(a)\setminus\{\infty\}$ is symplectomorphic to $B^2(a)$, if $a\leq b$, then $B^4(a)\subset B^2(a)\times B^2(a)\subset B^2(a)\times B^2(b)\hookrightarrow S^2(a)\times S^2(b)$. Hence $w_G(S^2(a)\times S^2(b))\geq a$. For the reverse inequality, let $\Delta=[0,a]\times[0,b]$ be the moment polytope of $S^2(a)\times S^2(b)$. By Theorem~3.8 and Proposition~3.9 of \cite{AHN23}, 
\begin{equation} \label{equ-width}
w_G(S^2(a)\times S^2(b))\leq {\rm width}(\Delta),
\end{equation}
where ${\rm width}(\Delta)\coloneqq \min_{u\in(\bZ^2)^*\setminus\{0\}}\max_{x,y\in\Delta}|u(x)-u(y)|$. Taking $u=(1,0)$ gives ${\rm width}(\Delta)\leq a$. By symmetry, the general case gives $w_G(S^2(a)\times S^2(b))=\min\{a,b\}$.

It follows that $w_G(S^2(a)\times\bC)=a$. Indeed, the inclusion $B^2(a)\times\bC\hookrightarrow S^2(a)\times\bC$, together with the equality $w_G(B^2(a))=c_Z(B^2(a))=a$ and part (i), gives $w_G(S^2(a)\times\bC)\geq a$. Conversely, suppose $c<w_G(S^2(a)\times\bC)$. Then there is a symplectic embedding $\varphi\colon B^4(r)\hookrightarrow S^2(a)\times\bC$ with $r>c$. For any $0<\varepsilon<r$, the compact set $\varphi(\overline{B}^{\,4}(r-\varepsilon))$ projects to a compact subset of the $\bC$-factor, hence is contained in $S^2(a)\times B^2(b)$ for some $b>0$. Since $B^2(b)$ embeds symplectically into $S^2(b)$, we get $B^4(r-\varepsilon)\hookrightarrow S^2(a)\times S^2(b)$. Then (\ref{equ-width}) above gives 
\[ r-\varepsilon\leq w_G(S^2(a)\times S^2(b))=\min\{a,b\}\leq a.\]
Letting \(\varepsilon\to0\) gives \(r\le a\). Since \(c<r\), we obtain \(c<a\). As \(c<w_G(S^2(a)\times\bC)\) was arbitrary, it follows that \(w_G(S^2(a)\times\bC)\le a\).

Now, return to $W=S^2\setminus  \bigcup_{j=1}^n{\rm Int}(D_j^2)$. By Lemma~\ref{lem:surface-width-area}, $w_G(W,\omega)=A$. Given $\varepsilon>0$, the surface $(W,\omega)$ symplectically embeds into a sphere $S^2(A+\varepsilon)$ with total area $A+\varepsilon$, which implies that $W\times\bC$ embeds into $S^2(A+\varepsilon)\times\bC$. Then the computation above gives 
\[ w_G(W\times\bC)\leq w_G(S^2(A+\varepsilon)\times\bC)=A+\varepsilon.\]
Letting $\varepsilon\to 0$, we obtain $w_G(W\times\bC)\leq A$. For the opposite inequality, fix $\varepsilon>0$. Since $w_G(W,\omega)=A$, there exists a symplectic embedding $B^2(r)\hookrightarrow W$ with $r>A-\varepsilon$. Taking the product with the standard inclusion $B^2(r)\hookrightarrow\bC$ gives $B^4(r)\hookrightarrow W\times\bC$. Thus $w_G(W\times\bC)>A-\varepsilon$. Since $\varepsilon>0$ is arbitrary, $w_G(W\times\bC)\geq A$. Therefore $w_G(W\times\bC)=A$, as claimed.

\medskip

\noindent For (iii), we will consider in two different cases.

\smallskip

\emph{Case 1: $N$ is non-orientable.}
Then $N\cong k\mathbb{RP}^2$ for some $k\ge 1$. By Proposition~2.12 and the proof of Proposition~3.1 in \cite{DHL19}, there exists a rational symplectic $4$-manifold $(M,\omega)$ containing an embedded Lagrangian copy of $k\mathbb{RP}^2$. By the Weinstein neighborhood theorem, for every Riemannian metric $g$ on $N$, the disk cotangent bundle $D_g^*N$ symplectically embeds into $M$ after rescaling $\omega$. Since rational symplectic $4$-manifolds are symplectically uniruled, Proposition~\ref{prop:uniruled-finite} gives $w_G(M\times \mathbb C)<\infty$. By monotonicity, we have $w_G(D_g^*N\times \mathbb C)<\infty.$

\smallskip

\emph{Case 2: $N=S^2$ or $N=\mathbb T^2$.}
The anti-diagonal in $S^2\times S^2$ is a Lagrangian sphere, and the product of equators $S^1\times S^1\subset S^2\times S^2$ is a Lagrangian torus. Again, by the Weinstein neighborhood theorem, for every metric $g$ on $S^2$ or $\mathbb T^2$, the disk cotangent bundle $D_g^*N$ symplectically embeds into $S^2\times S^2$ after rescaling the symplectic form. Since $S^2\times S^2$ is a Hamiltonian $S^1$-manifold, it is symplectically uniruled, and Proposition~\ref{prop:uniruled-finite} implies that $w_G(S^2\times S^2\times \mathbb C)<\infty$. Again, monotonicity yields $w_G(D_g^*N\times \mathbb C)<\infty$. This completes the proof of case (iii). 
\end{proof}

\begin{proof}[Proof of Proposition~\ref{prop-lift-capacity}]
Since $\Sigma$ has positive genus, it contains a symplectic open subset $U\subset \Sigma$ symplectomorphic to a once-punctured torus. Fix such a $U$. By the main lemma in \cite{Gut08} and  Lemma~5.1 in \cite{PV15}, for every $m\ge 1$, there exist constants $C_U>0$ and $A_0>0$, depending only on $U$, such that for every $A\ge A_0$, there is a symplectic embedding $B^{2m+2}(A)\hookrightarrow U\times B^{2m}(C_UA^2)$. 

Now suppose that $\dim X=2m$. Since $w_G(X)=\infty$, for every $S>0$ there exists a symplectic embedding $B^{2m}(S)\hookrightarrow X$. Let $A\ge A_0$. Applying the preceding embedding result with this value of $A$, we get $B^{2m+2}(A)\hookrightarrow U\times B^{2m}(C_UA^2)$. Taking $S=C_UA^2$ and composing with a symplectic embedding $B^{2m}(C_UA^2)\hookrightarrow X$ and with the inclusion $U\subset\Sigma$, we obtain a symplectic embedding 
\[ B^{2m+2}(A)\hookrightarrow U\times B^{2m}(C_UA^2)\hookrightarrow \Sigma\times X.\] Since $A$ can be taken arbitrarily large, it follows that $w_G(\Sigma\times X)=\infty.$
\end{proof}
\section{Proof of Theorem \ref{thm-large-scale-dT}}

In this section, we will give the proof of Theorem \ref{thm-large-scale-dT}. 

\begin{proof}[Proof of Theorem \ref{thm-large-scale-dT}]
Let $c_N>0$ and $\{X_v\}_{v\in\bR^N}$ be the family given by Proposition \ref{prop-CGH}. In particular, for each $v\in\bR^N$, we can express
\[ X_v=\left\{(u,z)\in\bC^2 \, \big|\, \pi|u|^2\le 1,\ \pi|z|^2\le f_v(\pi|u|^2)\right\},\]
where $f_v:[0,1]\to[0,\infty)$ is smooth and concave, $f_v(1)=0$, and
$d_{\rm BM}(X_v,X_w)\ge c_N\|v-w\|_\infty$ for all $v,w\in\bR^N$.

Next, we use the profile function $f_v$ to construct a Hamiltonian on the disk of size $a$. Rescale $f_v$ from $[0,1]$ to $[0,a]$ by defining $\widetilde f_v:[0,a]\to[0,\infty)$, $\widetilde f_v(\mu)\coloneqq a\,f_v(\mu/a)$. Then $\widetilde f_v$ is again smooth and concave with $\widetilde f_v(a)=0$. Define an autonomous radial Hamiltonian on $\overline{B}^{\,2}(a)$ by 
\[ H_v(u)\coloneqq \widetilde f_v(\pi|u|^2) = \widetilde f_v(\mu) \,\,\,\,\mbox{if $\mu\coloneqq \pi|u|^2$}.\]
Note that $H_v\in \mathcal H^+(\overline{B}^{\,2}(a),\lambda_{\mathrm{std}})$. Indeed, since $H_v$ is autonomous on all of $\overline{B}^{\,2}(a)$ and vanishes on $\partial \overline{B}^{\,2}(a)$, it remains to verify the positivity conditions. For such a radial (symmetric) Hamiltonian, it is readily verified that 
\[ \lambda_{\mathrm{std}}(X_{H_v})=-\mu \widetilde f_v'(\mu) \,\,\,\,\mbox{and}\,\,\,\, \tau_{\lambda_{\mathrm{std}},H_v}=\widetilde f_v(\mu)-\mu \widetilde f_v'(\mu).\]
Since $X_v$ is a graph-type toric domain and its profile function is non-increasing, $\widetilde f_v'(\mu)\le 0$, which implies that $\lambda_{\mathrm{std}}(X_{H_v})\ge 0$. Moreover, $\widetilde f_v(\mu)>0$ for $0\le \mu<a$, so on ${\rm Int}(\overline{B}^{\,2}(a))$ we have 
\[ \tau_{\lambda_{\mathrm{std}},H_v}=\widetilde f_v(\mu)-\mu \widetilde f_v'(\mu)\ge \widetilde f_v(\mu)>0.\] 
Therefore, consider the following map 
\[ \Phi_{N,a}: \bR^N \to {\mathcal H}^+(\overline{B}^{\,2}(a),\lambda_{\mathrm{std}})\,\,\,\,\mbox{by $v \mapsto \Phi_{N,a}(v)\coloneqq H_v$} \]
as constructed above. 

Next we identify the associated tube domain. By definition,
${\rm U}(H_v)=\{(u,z)\in \overline{B}^{\,2}(a)\times\bC \mid \pi|z|^2\le \widetilde f_v(\pi|u|^2)\}$.
Substituting the definition of $\widetilde f_v$, we get 
$${\rm U}(H_v)=\{(u,z)\in\bC^2 \mid \pi|u|^2\le a,\ \pi|z|^2\le a\,f_v(\pi|u|^2/a)\}.$$
This is exactly the Liouville rescaling of $X_v$ by the factor $a$; in other words,
${\rm U}(H_v)=aX_v$.

Thus, we obtain
$d_{\rm SBM}({\rm U}(H_v),{\rm U}(H_w))=d_{\rm SBM}(X_v,X_w)$, since $d_{\rm SBM}$ is invariant under simultaneous Liouville rescaling.
Then Theorem~\ref{thm-estimation} gives
\[ d_{\rm T}(H_v,H_w)\ge d_{\rm SBM}({\rm U}(H_v),{\rm U}(H_w))=d_{\rm SBM}(X_v,X_w).\]
Using $d_{\rm SBM}\ge d_{\rm BM}$ and the lower bound from Proposition \ref{prop-CGH}, we conclude that
$d_{\rm T}(H_v,H_w)\ge d_{\rm BM}(X_v,X_w)\ge c_N\|v-w\|_\infty$ as requested.
\end{proof}

\section{Proof of Theorem \ref{thm-cat}}

First, let us give the proof of Lemma \ref{lemma-dT-thick}. 

\begin{proof}[Proof of Lemma \ref{lemma-dT-thick}] 
By the defining condition on the right-hand side of (\ref{dT-thickness}), for any such $\psi\in{\rm Symp}(W,\lambda)$, the corresponding Hamiltonian functions $H_-,H_+\in\mathcal H(W,\lambda)$ satisfy $H_-\le H$, $H_-\le G\circ\psi$, $H\le H_+$, and $G\circ\psi\le H_+$. Therefore, for every $(t,w)\in S^1\times{\rm Int}(W)$, we have
\begin{equation} \label{ratio-bound-1}
\frac{H(t,w)}{G(t,\psi(w))}\leq \frac{H_+(t,w)}{H_-(t,w)}, \qquad \frac{G(t,\psi(w))}{H(t,w)}\leq \frac{H_+(t,w)}{H_-(t,w)}.
\end{equation}
Hence, by taking logarithms, we get
\begin{equation} \label{proof-ln}
\ln\left(\sup_{(t,w)\in S^1\times{\rm Int}(W)}\max\left\{\frac{H(t,w)}{G(t,\psi(w))},\frac{G(t,\psi(w))}{H(t,w)}\right\}\right)\leq \ell(W_{H_-,H_+}).
\end{equation}
Taking the infimum over all such $\psi\in{\rm Symp}(W,\lambda)$ gives $d_{\rm T}(H,G)\leq \mbox{the right-hand side of }(\ref{dT-thickness})$.

For the reverse inequality, we aim to prove that there exist $H_-, H_+ \in \mathcal H(W, \lambda)$, satisfying the condition on the right-hand side of (\ref{dT-thickness}), such that $\ell(W_{H_-,H_+})\le d_{\rm T}(H,G)$. Note that the direct taking $H_- = \min\{H, G \circ \psi\}$ and $H_+ = \max\{H, G \circ \psi\}$ does {\it not} always work, since it is not guaranteed that these $H_-, H_+ \in \mathcal H(W, \lambda)$. By viewing ``$\max$'' and ``$\min$'' as limits of certain $L_p$-norms, let us consider an intermediate step: for $p \geq 1$ and defined in a pointwise manner, 
\[ 
K_+^{(p)}\coloneqq (H^p+(G \circ \psi)^p)^{1/p},\qquad K_-^{(p)}\coloneqq (H^{-p}+(G \circ \psi)^{-p})^{-1/p}
\]
for any fixed $\psi\in \mathrm{Symp}(W,\lambda)$. Then we claim such $K_+^{(p)}, K_-^{(p)} \in \mathcal H(W, \lambda)$. 

For brevity, let us formulate a more general and an abstract setting: (labelling $f\coloneqq H$ and $g\coloneqq G \circ \psi$) if $F:(0,\infty)^2\to (0,\infty)$ is smooth, positively homogeneous of degree $1$, and nondecreasing in each variable, then we will show that $F(f,g)\in \mathcal H(W,\lambda)$ whenever $f,g\in \mathcal H(W,\lambda)$. Indeed, by definition $dF(f,g) =(\partial_1F)(f,g)\,df+(\partial_2F)(f,g)\,dg$, which implies that 
\[
X_{F(f,g)}=(\partial_1F)(f,g)\,X_f+(\partial_2F)(f,g)\,X_g.
\]
Then for the $\tau$ function in (\ref{H-con}), we have 
\begin{align*}
\tau_{\lambda,F(f,g)} &= \iota_{X_{F(f,g)}}\lambda+F(f,g)\\
& = (\partial_1F)(f,g)\,\iota_{X_f}\lambda + (\partial_2F)(f,g)\,\iota_{X_g}\lambda +  F(f,g)\\
& = (\partial_1F)(f,g)\,\iota_{X_f}\lambda + (\partial_2F)(f,g)\,\iota_{X_g}\lambda +  (\partial_1F)(f,g) f+ (\partial_2F)(f,g) g\\
& = (\partial_1F)(f,g)(\iota_{X_f}\lambda+f)
+
(\partial_2F)(f,g)(\iota_{X_g}\lambda+g)\\
& =  (\partial_1F)(f,g) \tau_{\lambda, f} + (\partial_2F)(f,g)\tau_{\lambda, g}
\end{align*}
where the third equality uses Euler's identity for positively homogeneous functions of degree $1$. Since $\partial_1F,\partial_2F\ge 0$,  $f,g\in \mathcal H(W,\lambda)$ (hence, $\tau_{\lambda,f}>0$ and $\tau_{\lambda,g}>0$)
 implies that $\tau_{\lambda,F(f,g)}>0$. Applying this to the function $F(a,b) = (a^p + b^p)^{1/p}$ and $F(a,b) = (a^{-p} + b^{-p})^{-1/p}$ respectively, we get the desired claim that $K_+^{(p)}, K_-^{(p)} \in \mathcal H(W, \lambda)$. 

Next, by the elementary inequalities
\[
\max(a,b)\le (a^p+b^p)^{1/p}\le 2^{1/p}\max(a,b),
\]
and
\[
2^{-1/p}\min(a,b)\le (a^{-p}+b^{-p})^{-1/p}\le \min(a,b),
\]
we obtain $K_-^{(p)}\le \min\{H, G\circ \psi\}\le \max\{H, G \circ \psi\}\le K_+^{(p)}$. Therefore, $K_-^{(p)}$ and $K_+^{(p)}$ satisfy the condition on the right-hand side of (\ref{dT-thickness}). Moreover, the inequalities above also imply that 
\begin{align*}
\ell(W_{K_-^{(p)},K_+^{(p)}})
&=\ln\left(\sup_{(t,w)\in S^1\times{\rm Int}(W)}\frac{K_+^{(p)}(t,w)}{K_-^{(p)}(t,w)}\right) \nonumber\\
&\leq \ln\left(\sup_{(t,w)\in S^1\times{\rm Int}(W)}\frac{\max\{H(t,w),G(t,\psi(w))\}}{\min\{H(t,w),G(t,\psi(w))\}}\right)+\frac{2\ln 2}{p} \nonumber\\
&=\ln\left(\sup_{(t,w)\in S^1\times{\rm Int}(W)}\max\left\{\frac{H(t,w)}{G(t,\psi(w))},\frac{G(t,\psi(w))}{H(t,w)}\right\}\right)+\frac{2\ln 2}{p}. \label{thick-K-bound}
\end{align*}
Finally, letting $p \to \infty $ and then taking an infimum over all $\psi\in \mathrm{Symp}(W,\lambda)$ yields the inequality $(\mbox{right-hand side of (\ref{dT-thickness})}) \leq d_{\rm T}(H, G)$, as requested.  
\end{proof}

Then we will give the proof of Theorem \ref{thm-cat}. 

\begin{proof} [Proof of Theorem \ref{thm-cat}] We first verify the two axioms in the statement. For (i), taking $\psi={\rm id}$ in $(\ref{hom-tun-cat})$ gives 
\[
{\rm hom}_{{\rm Tun}(W,\lambda)}(H,H)\leq \max\left\{\ln (\sup_{(t,w)\in S^1\times{\rm Int}(W)} H(t,w)/H(t,w)),0\right\}=0.
\] 
Since the right-hand side of $(\ref{hom-tun-cat})$ is always nonnegative, we obtain ${\rm hom}_{{\rm Tun}(W,\lambda)}(H,H)=0$.

For (ii), let $H_1,H_2,H_3\in\mathcal H(W,\lambda)$, and fix $\epsilon>0$. If either ${\rm hom}_{{\rm Tun}(W,\lambda)}(H_1,H_2)$ or ${\rm hom}_{{\rm Tun}(W,\lambda)}(H_2,H_3)$ is infinite, then the desired inequality is immediate. Thus we may assume that both quantities are finite.

By the definition of the infimum, there exist $\psi,\varphi\in{\rm Symp}(W,\lambda)$ such that, with $A_\psi\coloneqq \sup_{(t,w)\in S^1\times{\rm Int}(W)} H_1(t,w)/H_2(t,\psi(w))$ and $B_\varphi\coloneqq \sup_{(t,w)\in S^1\times{\rm Int}(W)} H_2(t,w)/H_3(t,\varphi(w))$, we have
\[
\max\{\ln A_\psi,0\}\leq {\rm hom}_{{\rm Tun}(W,\lambda)}(H_1,H_2)+\epsilon,\,\,\,\max\{\ln B_\varphi,0\}\leq {\rm hom}_{{\rm Tun}(W,\lambda)}(H_2,H_3)+\epsilon .
\]
For every $(t,w)\in S^1\times{\rm Int}(W)$, we have
\[
\frac{H_1(t,w)}{H_3(t,\varphi(\psi(w)))}=\frac{H_1(t,w)}{H_2(t,\psi(w))}\cdot\frac{H_2(t,\psi(w))}{H_3(t,\varphi(\psi(w)))} .
\]
Taking the supremum and using that $\psi$ maps ${\rm Int}(W)$ diffeomorphically onto itself, we get
\[
\sup_{(t,w)\in S^1\times{\rm Int}(W)}\frac{H_1(t,w)}{H_3(t,\varphi(\psi(w)))}\leq A_\psi B_\varphi .
\]
Therefore, ${\rm hom}_{{\rm Tun}(W,\lambda)}(H_1,H_3)\leq \max\{\ln(A_\psi B_\varphi),0\}\leq \max\{\ln A_\psi,0\}+\max\{\ln B_\varphi,0\}$. Combining the last inequality with the choices of $\psi$ and $\varphi$, we obtain
\[
{\rm hom}_{{\rm Tun}(W,\lambda)}(H_1,H_3)\leq{\rm hom}_{{\rm Tun}(W,\lambda)}(H_1,H_2)+{\rm hom}_{{\rm Tun}(W,\lambda)}(H_2,H_3)+2\epsilon.
\]
Letting $\epsilon\to0$ proves the triangle inequality.

It remains to identify the symmetrization with the Thompson-type metric. For fixed $\psi\in{\rm Symp}(W,\lambda)$, the condition $(1/C)H(t,w)\leq G(t,\psi(w))\leq C H(t,w)$ for all $(t,w)\in S^1\times{\rm Int}(W)$ is equivalent to
\[
\max\left\{\frac{H(t,w)}{G(t,\psi(w))},\frac{G(t,\psi(w))}{H(t,w)}\right\}\leq C
\]
for all $(t,w)\in S^1\times{\rm Int}(W)$. Hence, for this fixed $\psi$, the smallest possible value of $\ln C$ is
\[
\ln\left(\sup_{(t,w)\in S^1\times{\rm Int}(W)}\max\left\{\frac{H(t,w)}{G(t,\psi(w))},\frac{G(t,\psi(w))}{H(t,w)}\right\}
\right).
\]
Taking the infimum over all $\psi\in{\rm Symp}(W,\lambda)$, and comparing with Definition~\ref{dfn-T-type}, gives ${\rm hom}_{{\rm Tun}(W,\lambda)}^{\rm sym}(H,G)=d_{\rm T}(H,G)$. This completes the proof.
\end{proof}

\section{Proof of Lemma \ref{prop:winding-monotonicity}} \label{ssec-Proof of Lemma-winding}

We first recall the standard positivity-of-intersection input that will be used below. By assumption, $D$ is $J$-holomorphic and $X_{K_N}$ is tangent to $D$ along $D$, the Floer equation is compatible with the divisor $D$. Concretely, near any point of $D$, choose local complex coordinates $(x,z)$, where $D=\{z=0\}$. The assumption $JTD=TD$ says that the normal direction to $D$ is $J$-complex, while the condition $X_{K_N}|_D\subset TD$ implies that the normal component of the Hamiltonian perturbation vanishes along $D$. Therefore, the normal component $u_{\bC}$ of $u$ satisfies a Cauchy--Riemann type equation with lower-order terms. By the similarity principle, the zeros of $u_{\bC}$ are isolated and have positive order, unless $u_{\bC}\equiv0$ on a connected component. See, for example, Appendix~E in \cite{MS12} for positivity of intersections and Lemma~7.1 in \cite{CM07}  for the corresponding divisor argument in the presence of compatible perturbations.

In our situation, $u$ cannot be entirely contained in $D$. Indeed, $D$ is preserved by the Hamiltonian flow of $K_N$, so a Hamiltonian orbit which meets $D$ is contained in $D$. Since neither $\gamma_-$ nor $\gamma_+$ is contained in $D$, both asymptotic orbits are disjoint from $D$. By the exponential convergence of finite-energy Floer cylinders to their asymptotic orbits, $u$ is also disjoint from $D$ for $|s|$ sufficiently large. Therefore $u^{-1}(D)$ is a finite set, and the intersection number is given by $u\cdot D=\sum_{p\in u^{-1}(D)}{\rm mult}_p(u,D)$. Each local multiplicity is positive, so $u\cdot D\geq0$, and $u\cdot D=0$ if and only if $u^{-1}(D)=\varnothing$, equivalently $u\cap D=\varnothing$.

It remains to identify this intersection number with the change of winding number. Let $u_{\bC}$ denote the $z$-component of $u$. On the complement of $u^{-1}(D)$, the one-form $d\theta_{\bC}=d\arg z$ is defined. Choose $R>0$ large enough so that $u$ is disjoint from $D$ on $((-\infty,-R]\cup[R,\infty))\times S^1$, and choose small pairwise disjoint disks $B_\epsilon(p)$ around the finitely many points $p\in u^{-1}(D)\cap([-R,R]\times S^1)$. Set $\Sigma_{R,\epsilon}\coloneqq [-R,R]\times S^1\setminus\bigcup_p B_\epsilon(p)$. Since $d(d\theta_{\bC})=0$ on $\bC^*$, Stokes' theorem gives
\[
0=\int_{\partial\Sigma_{R,\epsilon}}u_{\bC}^*d\theta_{\bC}=\int_{\{R\}\times S^1}u_{\bC}^*d\theta_{\bC}-\int_{\{-R\}\times S^1}u_{\bC}^*d\theta_{\bC}-\sum_{p\in u^{-1}(D)}\int_{\partial B_\epsilon(p)}u_{\bC}^*d\theta_{\bC}.
\]
With the boundary orientation induced from $\Sigma_{R,\epsilon}$, the circles $\partial B_\epsilon(p)$ are negatively oriented. Thus $(2\pi)^{-1}\int_{\partial B_\epsilon(p)}u_{\bC}^*d\theta_{\bC}={\rm ord}_p(u_{\bC})$, where ${\rm ord}_p(u_{\bC})$ is the vanishing order of the normal component at $p$. Therefore,
\[
\sum_{p\in u^{-1}(D)}{\rm ord}_p(u_{\bC})=\frac{1}{2\pi}\int_{\{R\}\times S^1}u_{\bC}^*d\theta_{\bC}-\frac{1}{2\pi}\int_{\{-R\}\times S^1}u_{\bC}^*d\theta_{\bC}.
\]
Letting $R\to\infty$, the exponential convergence of $u$ to $\gamma_+$ at $+\infty$ and to $\gamma_-$ at $-\infty$ implies that the two boundary integrals converge to the winding numbers $k(\gamma_+)$ and $k(\gamma_-)$, respectively. Since the local intersection multiplicity with $D=\{z=0\}$ is exactly the vanishing order of $u_{\bC}$, the left-hand side is $u\cdot D$. Therefore, $u\cdot D=k(\gamma_+)-k(\gamma_-)$.

Since $u\cdot D\geq0$, we obtain $k(\gamma_+)\geq k(\gamma_-)$. Thus, with the convention in this lemma that a Floer cylinder contributing to the differential goes from the negative end $\gamma_-$ to the positive end $\gamma_+$, the Floer differential cannot decrease the winding number.

\bibliographystyle{amsplain}
\bibliography{biblio}

\,

\end{document}